\documentclass[12pt]{amsart}
\usepackage{amsmath}
\usepackage{amsfonts}
\usepackage{graphics}
\usepackage{epsfig}
\usepackage{amssymb}
\usepackage{amscd}
\usepackage[all]{xy}
\usepackage{latexsym}
\usepackage{graphicx}
\usepackage{multirow}
\usepackage{geometry}
\usepackage{mathrsfs}

\theoremstyle{plain}
\newtheorem{thm}{Theorem}[section]
\newtheorem{lem}[thm]{Lemma}
\newtheorem{cor}[thm]{Corollary}
\newtheorem{prop}[thm]{Proposition}

\theoremstyle{definition}
\newtheorem{defn}[thm]{Definition}

\theoremstyle{remark}

\numberwithin{equation}{section} \numberwithin{figure}{section}

\renewcommand*{\to}{\rightarrow}
\renewcommand*{\bar}[1]{\overline{#1}}

\renewcommand{\Re}{\operatorname{Re}}
\renewcommand{\Im}{\operatorname{Im}}

\newcommand{\Res}{\operatorname{Res}}

\newcommand{\supp}{\operatorname{supp}}

\newcommand{\mb}[1]{\mathbb{#1}} 

\newcommand{\mc}[1]{\mathcal{#1}}
\newcommand{\mk}[1]{\mathfrak{#1}}

\newcommand{\Tr}{\operatorname{Tr}}

\newcommand{\Spec}{\operatorname{Spec}}

\newcommand{\id}{\operatorname{id}}

\newcommand{\mf}[1]{\mathbf{#1}}

\usepackage{url}
\usepackage{pgfplots}
\pgfplotsset{compat=1.18}
\usetikzlibrary{calc} 
\usepackage{mathtools}
\usepackage{tikz-cd}
\usepackage{graphicx}
\usepackage{xcolor}
\usepackage{tikz}
\usepackage{amsmath,amssymb}
\usetikzlibrary{arrows.meta,calc,positioning}

\definecolor{oldblue}{RGB}{214,231,248}
\definecolor{collaryellow}{RGB}{250,235,185}
\definecolor{newgreen}{RGB}{225,241,211}
\definecolor{targetpink}{RGB}{249,226,239}
\definecolor{glueorange}{RGB}{180,118,15}

\tikzset{
  surface/.style={
    draw=black,
    line width=0.9pt,
    fill=white
  },
  olddisc/.style={
    draw=black,
    dashed,
    line width=0.7pt,
    fill=oldblue
  },
  collar/.style={
    draw=black,
    line width=0.7pt,
    fill=collaryellow
  },
  newpiece/.style={
    draw=black,
    line width=0.8pt,
    fill=newgreen
  },
  core/.style={
    draw=black,
    dashed,
    line width=0.7pt,
    fill=white
  },
  insert/.style={
    draw=black,
    dashed,
    line width=0.7pt,
    fill=newgreen
  },
  panelnumber/.style={
    circle,
    draw=black,
    minimum size=7mm,
    inner sep=0pt,
    font=\large
  },
  every node/.style={
    font=\small
  }
}
\usepackage{xcolor}

\usetikzlibrary{
  arrows.meta,
  calc,
  decorations.markings,
  positioning
}

\definecolor{discblue}{RGB}{222,237,250}
\definecolor{pathblue}{RGB}{44,104,174}
\definecolor{oldred}{RGB}{177,65,65}
\definecolor{newgreen}{RGB}{65,133,86}
\definecolor{piecegreen}{RGB}{225,242,228}
\definecolor{sheetgray}{RGB}{240,240,240}

\tikzset{
  >=Stealth,
  panel/.style={
    draw=black!55,
    rounded corners=3pt,
    line width=.65pt
  },
  title/.style={
    font=\bfseries\normalsize,
    anchor=west
  },
  sphere/.style={
    draw=black,
    line width=.85pt,
    fill=white
  },
  surface/.style={
    draw=black,
    line width=.85pt,
    fill=white
  },
  disc/.style={
    draw=pathblue,
    fill=discblue,
    line width=.75pt
  },
  oldpoint/.style={
    circle,
    fill=oldred,
    inner sep=1.7pt
  },
  newpoint/.style={
    circle,
    fill=newgreen,
    inner sep=1.7pt
  },
  boundary/.style={
    draw=black,
    line width=.75pt,
    fill=sheetgray
  },
  attachdisc/.style={
    draw=newgreen!90!black,
    line width=.8pt,
    fill=piecegreen
  },
  path/.style={
    draw=pathblue,
    line width=.95pt
  },
  loop/.style={
    draw=oldred,
    line width=.95pt,
    postaction={decorate},
    decoration={
      markings,
      mark=at position .58 with {\arrow{Stealth}}
    }
  },
  stagearrow/.style={
    ->,
    line width=1pt,
    black!70
  },
  gluearrow/.style={
    ->,
    line width=.9pt,
    newgreen!80!black
  },
  small/.style={
    font=\scriptsize
  },
  note/.style={
    font=\footnotesize,
    align=center
  }
}

\title{Spectral Geometry of Hurwitz Spaces with Arbitrary Ramification}

\begin{document}

\author{Jia-Ming (Frank) Liou}
\address{Department of Mathematics\\
National Cheng Kung University\\
No.1, University Road, Tainan City 701, Taiwan\\ fjmliou@mail.ncku.edu.tw}

\begin{abstract}
We study the Friedrichs Laplacian associated with the pullback of
the round metric on \(\mb P^1\) by a nonconstant meromorphic function
\(\varphi:X\to\mb P^1\) on a compact Riemann surface.
For arbitrary ramification profiles, including several ramification
points over the same branch value, we prove the local formula
\[
\operatorname{Det}_{\zeta}(\Delta_{[\varphi],\mc F})
=C\,\det\operatorname{Im}B\,|\tau_B|^2
\prod_{k=1}^N\rho(z_k,\overline{z_k})^{c_k}.
\]
The zero eigenvalue is omitted. Here \(B\) is the period matrix,
\(\tau_B\) is the local Bergman
tau-function, \(z_k\) are the branch-value coordinates,
\(\rho(z,\overline z)=4(1+|z|^2)^{-2}\), and
\(c_k=\frac1{12}\sum_j(n_{kj}-n_{kj}^{-1})\), where \(n_{kj}\)
are the ramification indices over the \(k\)-th branch value.
The constant \(C>0\)
is independent of the Hurwitz coordinates, and
\(\det\operatorname{Im}B\) is taken to be \(1\) in genus zero.
The proof combines smooth trivializations and trace-norm variation
of resolvent powers with
matrix comparison to spherical conic models. The Davies--Gaffney
estimate provides the required uniform high-energy control, while
the zero-energy terms are identified through the Schiffer
bidifferential and Rauch variational formulas.
\end{abstract}
\maketitle

\section{Introduction}\label{sec:introduction}

A nonconstant meromorphic function \(\varphi:X\to\mb P^1\) on a
compact connected Riemann surface pulls the round metric on
\(\mb P^1\) back to a metric \(g_\varphi\) on \(X\). This metric has
curvature \(1\) away from the ramification points, while a point of
ramification index \(n\) becomes a conical singularity of angle
\(2\pi n\). Branched coverings therefore provide a natural family
in which the complex geometry of Hurwitz spaces and the spectral
geometry of conic surfaces meet. The purpose of this paper is to
determine the zeta-regularized determinant of the Friedrichs
Laplacian of \(g_\varphi\) for arbitrary ramification profiles.

We write \(\Delta_{[\varphi],\mc F}\) for this Friedrichs Laplacian
and
\(\mathcal D([\varphi])
:=\operatorname{Det}_{\zeta}(\Delta_{[\varphi],\mc F})\), with the
zero eigenvalue omitted. Zeta regularization goes back to Ray and
Singer \cite{RaySinger1971}; the analytic theory of Laplacians on
conic spaces is developed in, among other works,
\cite{Cheeger1983,Mooers1999,Kalvin2021}. Equivalent coverings
produce isometric pullback metrics, so \(\mathcal D\) is naturally
a function on the corresponding Hurwitz space.

The present problem belongs to a line of work relating determinants
of conic Laplacians to boundary asymptotics, \(S\)-matrices, and
tau-functions. Hillairet and Kokotov
\cite{HillairetKokotov2013} established determinant comparison
formulas for self-adjoint extensions on Euclidean conic surfaces.
Hillairet, Kalvin, and Kokotov
\cite{HillairetKalvinKokotov2018} treated the flat metrics
\(|d\varphi|^2\), including higher-order zeros of \(d\varphi\), and
related the zero-energy data to the Schiffer projective connection.
For spherical pullback metrics, Kalvin and Kokotov
\cite{KalvinKokotov2019} obtained an explicit formula in the
simple-ramification case and indicated an extension to general
ramification. Related formulas for pullbacks of more general conic
metrics were obtained by Kalvin \cite{Kalvin2019}. On the
complex-analytic side, Rauch's variational formulas
\cite{Rauch1959} and the Hurwitz-space tau-functions of Kokotov and
Korotkin \cite{KokotovKorotkin2004} supply the natural language for
the variation of periods and determinants.

The passage from simple to arbitrary ramification is not formal.
A ramification point of index \(n\) carries a space of critical
boundary modes whose dimension grows with \(n\). Consequently, the
comparison relevant to the determinant variation can no longer be
reduced to the single coefficient occurring in the simple case and
must instead be organized at the matrix level.
Furthermore, several ramification points, possibly of different
indices, may lie over one branch value. Their local spectral data
are distinct, but their variations occur in the same Hurwitz
coordinate and must ultimately combine into a single invariant
contribution. Finally, the deformation moves both the complex
structure and the conic metric, and hence must be placed on a fixed
operator domain before the resolvent trace can be differentiated.

The flat-metric analysis of Hillairet, Kalvin, and Kokotov already
allows zeros of \(d\varphi\) of arbitrary order and contains the
corresponding higher-angle zero-energy \(S\)-matrix calculation.
The spectral problem here is different: we study the absolute zeta
determinant of the compact spherical pullback metric
\(\varphi^*ds_{\mathrm{rd}}^2\), whereas the flat metric
\(|d\varphi|^2\) leads naturally to a relative determinant problem
with Euclidean or conical ends. The compact spherical problem
requires a separate high-energy comparison and produces the
round-metric factors in the final formula.

This paper grew out of an attempt to carry the
simple-ramification formula of \cite{KalvinKokotov2019} through for
arbitrary ramification. Although a general extension was indicated
there, the explicit calculation and proof were carried out under
simple-ramification assumptions. We give a complete formula for
arbitrary ramification profiles, including fibers containing
several ramification points of different indices. The variational
strategy follows the preceding works, but its implementation in
this setting requires the matrix-valued boundary comparison, the
simultaneous organization of local terms by branch value, and the
fixed-domain operator analysis developed below.

The key structural device is the intrinsic boundary-data curve, or
Weyl curve: the spectral parameter determines
a curve of boundary-data subspaces in the Grassmannian of the
boundary symplectic space. Relative to a Lagrangian decomposition,
the Weyl function is its graph coordinate, and the local
\(S\)-matrix is the matrix representation of that coordinate.
Viewing the local matrices in this way makes it possible to compare
the global boundary data with those of the spherical conic models,
even when a single branch value has several ramification points of
different indices. Thus the \(S\)-matrices are not introduced as
unrelated collections of coefficients: each is a coordinate
representation of the corresponding intrinsic boundary-data curve.
This observation organizes the matrix comparison on which the
arbitrary-ramification calculation rests.

The second organizing principle in the present argument is the use
of classical Hurwitz coordinates. Once a
branch value is used as a deformation coordinate, all ramification
points lying over that value move in the same deformation direction,
while their local ramification indices remain visible in the
individual boundary blocks. This makes clear both how the local
model comparisons extend to arbitrary ramification and why the
resulting residue contributions must be grouped by branch value.
Together, the Weyl-curve viewpoint and the Hurwitz-coordinate
description reveal the structure of the general calculation.

Cutoff diffeomorphisms realizing a moving conic point on a fixed
underlying surface already appear in the flat setting of Hillairet,
Kalvin, and Kokotov \cite{HillairetKalvinKokotov2018} and in the
spherical simple-ramification argument of Kalvin and Kokotov
\cite{KalvinKokotov2019}. We adapt this idea to the present gluing
construction. For every ramification point over a moving branch
value, we construct a local interpolation compatible with the gluing
maps and combine these local maps into a smooth trivialization. The
varying conic metrics are thereby pulled back to one fixed smooth
surface. We then prove uniform equivalence of the relevant graph
norms and identify the maximal, minimal, and Friedrichs domains
throughout the deformation. This gives a common functional-analytic
setting for the resolvents and their traces. The fixed-surface
analysis thereby connects the geometric picture of the Weyl curve
with the operator variation required by the determinant formula.

The Polyakov--Alvarez approach to conformally related conic metrics
\cite{Kalvin2021} provides a complementary point of view. In the
present Hurwitz-space problem, both the complex structure and the
conic metric vary. We instead transport the Laplacians to a fixed
smooth surface and analyze their variation directly. This approach
also retains the local Weyl data from which the determinant
variation is assembled.

We now state the main result. Let
\(H_{g,d,N}^{\mb P^1}\) be a Hurwitz space of connected degree-\(d\)
coverings of genus \(g\) with \(N\) distinct branch values. We assume
\(d\geq2\); the degree-one case is immediate. On a Hurwitz coordinate
ball, the ramification profiles are fixed. Label the branch values
by \(Q_1,\ldots,Q_N\). For each \(k\), choose one of the two
standard affine charts \(\xi_k\) on \(\mb P^1\) containing
\(Q_k\), and set \(z_k=\xi_k(Q_k)\). Write
\(\varphi^{-1}(Q_k)=\{P_{k1},\ldots,P_{ks_k}\}\), with ramification
indices \(n_{k1},\ldots,n_{ks_k}\). Let \(B\) be the period matrix
associated with a local homological marking, let \(\tau_B\) be the
local Bergman tau-function with the residue normalization of
Lemma~\ref{lem:bergman-tau-residue}, and set
\[
c_k:=\frac1{12}\sum_{j=1}^{s_k}
\left(n_{kj}-\frac1{n_{kj}}\right).
\]
In a standard affine coordinate on \(\mb P^1\), write
\(\rho(z,\overline z)=4(1+|z|^2)^{-2}\) for the density of the round
metric.

\begin{thm}\label{thm:introduction-determinant}
The function
\(\mathcal D:H_{g,d,N}^{\mb P^1}\to\mb R_{>0}\) is smooth.
On each Hurwitz coordinate ball, there exists a constant \(C>0\),
independent of the Hurwitz coordinates, such that
\[
\mathcal D
=C\,\det\operatorname{Im}B\,|\tau_B|^2
\prod_{k=1}^N\rho(z_k,\overline{z_k})^{c_k}.
\]
When \(g=0\), the empty determinant
\(\det\operatorname{Im}B\) is understood to be \(1\).
\end{thm}

Each ramification point of index \(n\) contributes
\((n^2-1)/(12n)\) to the exponent belonging to its branch value.
Thus several ramification points in the same fiber contribute to
the same Hurwitz coordinate. If each branch value has exactly one
ramification point, of index two, then \(c_k=1/8\), and the theorem
recovers the formula of
\cite{KalvinKokotov2019}. Branch values at infinity are handled by
the other standard affine chart. The constant \(C\) reflects the
multiplicative normalization of \(\tau_B\).

For a fixed conic surface, the comparison formula of
\cite[Theorem~5.3 and Corollary~5.4]{LiouWeyl2026} relates the
Friedrichs determinant in the theorem to the positive-spectrum
zeta determinants of other self-adjoint realizations satisfying the
zeta-regularity hypotheses imposed there.

The proof brings together three ingredients. Smooth
trivializations and common graph domains permit the spectral problem
to be differentiated on a fixed surface; matrix comparison with
spherical conic models determines the high-energy contribution; and
the Schiffer bidifferential and Rauch formulas identify the
zero-energy contribution with the variation of the Schiffer
tau-function. The nontrivial point is that these three descriptions
are compatible with the grouping by branch value. Their combination
produces a globally defined variation, which integrates to the
formula above.

Sections~\ref{sec:conic-analysis}--\ref{sec:local-model-comparison}
develop the conic operator theory and local spectral comparison.
Sections~\ref{sec:hurwitz-spaces}--\ref{sec:rauch-formulas} treat
Hurwitz deformations and the required variational formulas.
Sections~\ref{sec:resolvent-variation}--\ref{sec:hurwitz-determinants}
derive the determinant formula, and
Section~\ref{sec:conclusion} concludes the paper.
The technical estimates for the local interpolation maps are
collected in Appendix~\ref{app:local-interpolation}.

\section{Conic Laplacians and Boundary Symplectic Spaces}\label{sec:conic-analysis}

This section passes from the global Laplacian to its finite-dimensional
boundary symplectic space. We first introduce the maximal and minimal
domains and describe their quotient by critical asymptotic modes. We then discuss
self-adjoint extensions and show that, after choosing a Lagrangian
decomposition, the Weyl function is precisely the graph coordinate
of the intrinsic boundary-data subspace.

A compact conic Riemann surface is a triple
\((X,S,ds_X^2)\), where \(X\) is a compact connected Riemann surface,
\(S\subset X\) is a finite set, and \(ds_X^2\) is a Hermitian metric on
\(X'=X\setminus S\) with the following property: for every \(P\in S\),
there exist a neighborhood \(U\) of \(P\), a holomorphic coordinate
\(z:U\to\mb C\) centered at \(P\), an integer \(n_P\geq2\), and a
smooth positive function \(\rho\in C^\infty(U)\) such that
\[
    ds_X^2
    =
    \rho(z)|z|^{2n_P-2}|dz|^2
\]
on \(U'=U\setminus\{P\}\). The integer \(n_P\) is called the index at
\(P\), and the corresponding cone angle is \(2\pi n_P\). Thus, in this
paper, we restrict attention to conic metrics whose cone angles are
integer multiples of \(2\pi\).
For the spectral analysis of conic singularities and the corresponding
heat-kernel constructions, see \cite{Cheeger1983,Mooers1999}.

Let
\(
    \Delta_{X,c}:C_c^\infty(X')\to C_c^\infty(X')
\)
be the Laplace--Beltrami operator associated with \(ds_X^2\), initially
defined on \(C_c^\infty(X')\). It is a densely defined symmetric
operator on \(L^2(X',dA_X)\), where \(dA_X\) is the area form associated
with \(ds_X^2\). For simplicity, we write
\(L^2(X')=L^2(X',dA_X)\).
All \(L^2\) inner products are linear in the first argument.

\begin{defn}
Let \(u\in L^2(X')\). We say that the distributional Laplacian of \(u\)
belongs to \(L^2(X')\) if there exists \(F\in L^2(X')\) such that
\[
    \langle u,\Delta_{X,c}\varphi\rangle_{L^2(X')}
    =
    \langle F,\varphi\rangle_{L^2(X')}
\]
for every \(\varphi\in C_c^\infty(X')\). In this case, \(F\) is unique
and is denoted by \(\Delta_Xu\).
\end{defn}

Thus \(\Delta_X\) will denote the maximal Laplacian. Its domain is
\[
    \mc D_{\max}(\Delta_X)
    =
    \left\{
        u\in L^2(X'):
        \Delta_Xu\in L^2(X')
    \right\}.
\]
Equipped with the graph norm
\[
    \|u\|_{\Delta_X}
    =
    \left(
        \|u\|_{L^2(X')}^2
        +
        \|\Delta_Xu\|_{L^2(X')}^2
    \right)^{1/2},
\]
the space \(\mc D_{\max}(\Delta_X)\) is a Hilbert space.

The minimal domain \(\mc D_{\min}(\Delta_X)\) is the graph-norm
closure of \(C_c^\infty(X')\) in \(\mc D_{\max}(\Delta_X)\).
The corresponding minimal operator is
\(\Delta_{X,\min}:=\Delta_X|_{\mc D_{\min}(\Delta_X)}\).

We first record a comparison that allows us to use the flat conic
model without imposing a radiality assumption on the smooth metric
factor.

\begin{lem}\label{lem:conformal-domain-comparison}
Let \(g_1\) and \(g_2\) be conic metrics on the same compact
surface \(X\), with the same conic set, and suppose that
\(g_2=h g_1\) for a smooth positive function \(h\) on \(X\).
Write \(\Delta_i\) for the corresponding maximal Laplacians.
After identifying the two \(L^2\) spaces as sets of measurable
functions, one has
\[
\mc D_{\max}(\Delta_1)=\mc D_{\max}(\Delta_2),
\qquad
\mc D_{\min}(\Delta_1)=\mc D_{\min}(\Delta_2),
\]
with equivalent graph norms. The Green forms defined below agree,
and the Friedrichs extensions have the same operator domain.
On the common maximal domain, \(\Delta_2u=h^{-1}\Delta_1u\).
\end{lem}

\begin{proof}
In dimension two, \(dA_{g_2}=h\,dA_{g_1}\) and
\(\Delta_2=h^{-1}\Delta_1\) on \(X'\), also in the distributional
sense. Since \(h\) and \(h^{-1}\) are bounded, these identities
give equality of the maximal domains and equivalence of their
graph norms. Taking the closures of \(C_c^\infty(X')\) gives
equality of the minimal domains. The same identities give
\(\langle\Delta_2u,v\rangle_{L^2(g_2)}
=\langle\Delta_1u,v\rangle_{L^2(g_1)}\), so the Green forms agree.

The Dirichlet forms agree on \(C_c^\infty(X')\) by conformal
invariance in dimension two. Their form norms are equivalent, so
their closed form domains coincide. On this common form domain,
the identity \(\langle F,v\rangle_{L^2(g_2)}
=\langle hF,v\rangle_{L^2(g_1)}\) shows that a weak equation
represented by an \(L^2(g_2)\) function is equivalently represented
by an \(L^2(g_1)\) function. The characterization of the operator
associated with a closed form therefore gives equality of the
Friedrichs operator domains.
\end{proof}

Apply this lemma to \(g=ds_X^2\) by choosing a smooth positive
function \(h\) on \(X\) that agrees with \(\rho\) in a smaller
coordinate neighborhood of each conic point. Then \(h^{-1}g\)
equals \(|z|^{2n_P-2}|dz|^2\) in those neighborhoods.
Consequently, the domain decomposition and normalized boundary
pairings for the flat conic model apply to \(g\) as well.
This supplies the passage from the radial setting of
\cite[Section~2]{LiouWeyl2026} to the general smooth positive factors
allowed here.

It is known that the quotient
\(\mc D_{\max}(\Delta_X)/\mc D_{\min}(\Delta_X)\) is
finite-dimensional over \(\mb C\) and carries a natural
nondegenerate skew-Hermitian form; see, for example,
\cite[Section~2]{LiouWeyl2026}. Indeed, consider the skew-Hermitian
sesquilinear form
\(\mk q_X:\mc D_{\max}(\Delta_X)\times
\mc D_{\max}(\Delta_X)\to\mb C\) defined by
\[
    \mk q_X(u,v)
    =
    \langle\Delta_Xu,v\rangle_{L^2(X')}
    -
    \langle u,\Delta_Xv\rangle_{L^2(X')}.
\]
Its radical is
\[
    \operatorname{rad}\mk q_X
    =
    \left\{
        u\in\mc D_{\max}(\Delta_X):
        \mk q_X(u,v)=0
        \text{ for every }
        v\in\mc D_{\max}(\Delta_X)
    \right\},
\]
and one has
\(
    \operatorname{rad}\mk q_X
    =
    \mc D_{\min}(\Delta_X).
\)
Consequently, \(\mk q_X\) induces a nondegenerate skew-Hermitian form
on
\(
    \mc D_{\max}(\Delta_X)/
    \mc D_{\min}(\Delta_X).
\)

We now pass from the operator domains to their finite-dimensional
boundary model. More precisely, we describe this quotient and its
induced form in terms of the critical asymptotic modes. For each
\(P\in S\), choose a
coordinate disk \(U_P\) with holomorphic coordinate \(z_P\) centered
at \(P\), such that
\(
    ds_X^2
    =
    \rho_P(z_P)|z_P|^{2n_P-2}|dz_P|^2
\)
on \(U_P'=U_P\setminus\{P\}\). Write
\(z_P=r_Pe^{i\theta_P}\).

Define
\[
    f_{P,0}
    =
    \frac{1}{\sqrt{2\pi}},
    \qquad
    f_{P,0}^\#
    =
    -\frac{\log r_P}{\sqrt{2\pi}},
\]
and, for \(1\leq |j|\leq n_P-1\),
\[
    f_{P,j}
    =
    \frac{1}{\sqrt{4\pi|j|}}
    r_P^{|j|}e^{ij\theta_P},
    \qquad
    f_{P,j}^\#
    =
    \frac{1}{\sqrt{4\pi|j|}}
    r_P^{-|j|}e^{ij\theta_P}.
\]
The critical asymptotic space at \(P\) is
\[
    V_P
    =
    \operatorname{span}_{\mb C}
    \left\{
        f_{P,j},f_{P,j}^\#:
        |j|\leq n_P-1
    \right\}.
\]
The elements \(f_{P,j}\) and \(f_{P,j}^\#\) are understood as local
asymptotic terms. After multiplication by a cutoff function supported
in \(U_P\) and equal to \(1\) near \(P\), they determine elements of
\(\mc D_{\max}(\Delta_X)\) and hence classes in
\(
    \mc D_{\max}(\Delta_X)/
    \mc D_{\min}(\Delta_X).
\)
These classes are independent of the choice of cutoff function.

The local boundary form \(\omega_P\) on \(V_P\) is determined by
\[
    \omega_P(f_{P,j},f_{P,k}^\#)
    =
    \delta_{jk},
\]
together with
\[
    \omega_P(f_{P,j},f_{P,k})
    =
    \omega_P(f_{P,j}^\#,f_{P,k}^\#)
    =
    0,
    \qquad
    \omega_P(f_{P,j}^\#,f_{P,k})
    =
    -\delta_{jk}.
\]
For \(u\in\mc D_{\max}(\Delta_X)\), let
\(\pi_P:\mc D_{\max}(\Delta_X)\to V_P\) denote the map which
assigns to \(u\) its critical asymptotic part at
\(P\). Set
\[
    \pi_S
    =
    \bigoplus_{P\in S}\pi_P:
    \mc D_{\max}(\Delta_X)
    \longrightarrow
    \bigoplus_{P\in S}V_P.
\]
The map \(\pi_S\) is surjective and satisfies
\(
    \ker\pi_S=\mc D_{\min}(\Delta_X).
\)
It therefore induces a linear isomorphism
\[
    \mc D_{\max}(\Delta_X)/
    \mc D_{\min}(\Delta_X)
    \longrightarrow
    \bigoplus_{P\in S}V_P.
\]
We identify the quotient with the direct sum
\[
V_S:=\bigoplus_{P\in S}V_P,
\qquad
\omega_S(\xi,\eta):=\sum_{P\in S}\omega_P(\xi_P,\eta_P).
\]
For all
\(u,v\in\mc D_{\max}(\Delta_X)\),
\[
    \mk q_X(u,v)
    =
    \omega_S\bigl(\pi_Su,\pi_Sv\bigr).
\]
Following \cite[Section~2]{LiouWeyl2026}, we call
\((V_S,\omega_S)\) the boundary symplectic space associated with
\(X\). Here the symplectic structure is the nondegenerate
skew-Hermitian form \(\omega_S\). For
\(u\in\mc D_{\max}(\Delta_X)\), its boundary data are
\(\pi_Su\in V_S\); we also call \(\pi_S\) the boundary trace.

The asymptotic projection \(\pi_S\) is continuous in the maximal
graph norm. Indeed, choose a smooth cutoff \(\chi_P\) supported in
\(U_P\) and equal to \(1\) near \(P\). The coefficients of
\(f_{P,j}\) and \(f_{P,j}^{\#}\) in \(\pi_Pu\) are respectively
\(\mk q_X(u,\chi_Pf_{P,j}^{\#})\) and
\(-\mk q_X(u,\chi_Pf_{P,j})\).
The Cauchy--Schwarz inequality gives
\(|\mk q_X(u,v)|\leq2\|u\|_{\Delta_X}\|v\|_{\Delta_X}\),
so each coefficient is a continuous linear functional on
\(\mc D_{\max}(\Delta_X)\). Since \(V_S\) is finite dimensional,
the claimed continuity follows.

It is known that the self-adjoint extensions of
\(\Delta_{X,\min}\) are parametrized by the Lagrangian Grassmannian
\(\operatorname{LGr}(V_S,\omega_S)\); see
\cite[Theorem~3.2]{LiouWeyl2026} and \cite{HillairetKokotov2013}.
For each
\(\mc V\in\operatorname{LGr}(V_S,\omega_S)\), define
\[
    \mc D_{X,\mc V}
    =
    \pi_S^{-1}(\mc V)
    =
    \left\{
        u\in\mc D_{\max}(\Delta_X):
        \pi_S(u)\in\mc V
    \right\},
\]
and let
\(
    \Delta_{X,\mc V}
    =
    \left.
        \Delta_X
    \right|_{\mc D_{X,\mc V}}.
\)
Every self-adjoint extension is uniquely of this form, with
\(\mc V\) equal to the boundary trace of its domain.
\begin{defn}
A Lagrangian pair in \((V_S,\omega_S)\) is an ordered pair
\((\mc V,\mc V^\#)\) of Lagrangian subspaces such that
\(
    V_S
    =
    \mc V\oplus\mc V^\#.
\)
\end{defn}
Let \((\mc V,\mc V^\#)\) be a Lagrangian pair in
\((V_S,\omega_S)\). Denote by \(P_{\mc V}:V_S\longrightarrow\mc V\) and \(P_{\mc V^\#}:V_S\longrightarrow\mc V^\#\) the projections associated with the decomposition
\(V_S=\mc V\oplus\mc V^\#\). Define
\(
    \widetilde\pi_{\mc V}
    :=
    P_{\mc V}\circ\pi_S:
    \mc D_{\max}(\Delta_X)
    \longrightarrow
    \mc V
\)
and
\(
    \widetilde\pi_{\mc V^\#}
    :=
    P_{\mc V^\#}\circ\pi_S:
    \mc D_{\max}(\Delta_X)
    \longrightarrow
    \mc V^\#.
\)

\begin{defn}\label{def:formal-solution}
Let \(\lambda\in\mb C\). An element \(u_\lambda\in\mc D_{\max}(\Delta_X)\) is called
a formal solution of
\begin{equation}\label{eq:model}
    (\Delta_{X}-\lambda)u_\lambda=0
\end{equation}
if this equation holds on \(X'\) in the sense of distributions.
\end{defn}

\begin{defn}
Following \cite[Section~6]{LiouWeyl2026}, for
\(\lambda\in\mb C\) we define the intrinsic boundary-data subspace
of \(\Delta_X-\lambda\) by
\[
\mc K_X(\lambda)
:=\pi_S\bigl(\ker(\Delta_X-\lambda)\bigr)\subset V_S.
\]
Thus \(\mc K_X(\lambda)\) consists of the critical boundary data
of global \(L^2\) formal solutions, without imposing a self-adjoint
boundary condition. It is defined independently of a Lagrangian pair.
\end{defn}

The Weyl function will describe this subspace in coordinates
associated with a Lagrangian decomposition of \(V_S\).
We first construct the formal solution with prescribed
\(\mc V^\#\)-component of its boundary data.

Fix \((\mc V,\mc V^\#)\) and
\(\lambda\in\rho(\Delta_{X,\mc V})\). For \(h\in\mc V^\#\),
denote by \(u_\lambda=\Gamma_{\mc V,\mc V^\#}^X(\lambda)h\)
the unique formal solution with
\(\widetilde\pi_{\mc V^\#}u_\lambda=h\).
To construct it, choose
\(u\in\mc D_{\max}(\Delta_X)\) such that
\(\widetilde\pi_{\mc V^\#}u=h\), and set
\[
    u_\lambda
    =
    u
    +
    R_{X,\mc V}(\lambda)
    (\Delta_X-\lambda)u,
\]
where
\(
    R_{X,\mc V}(\lambda)
    =
    (\lambda-\Delta_{X,\mc V})^{-1}.
\)
When the surface is understood, we write
\(\mc D_{\mc V}=\mc D_{X,\mc V}\) and
\(R_{\mc V}=R_{X,\mc V}\).
The resolvent identity gives \((\Delta_X-\lambda)u_\lambda=0\).
Since the correction lies in \(\mc D_{X,\mc V}\), it leaves the
\(\mc V^\#\)-component unchanged. Uniqueness follows because the
difference of two such solutions belongs to
\(\ker(\Delta_{X,\mc V}-\lambda)=\{0\}\).

On \(\mc D_{\max}(\Delta_X)\), define
\(Y_{X,\mc V}(\lambda):=I+R_{X,\mc V}(\lambda)(\Delta_X-\lambda)\).
Then \(Y_{X,\mc V}(\lambda)\) is a projection satisfying
\[
    \ker Y_{X,\mc V}(\lambda)
    =
    \mc D_{X,\mc V},
    \qquad
    \operatorname{Im}Y_{X,\mc V}(\lambda)
    =
    \ker(\Delta_X-\lambda).
\]
It therefore induces a linear isomorphism
\[
    \overline Y_{X,\mc V}(\lambda):
    \mc D_{\max}(\Delta_X)/\mc D_{X,\mc V}
    \longrightarrow
    \ker(\Delta_X-\lambda).
\]
Likewise, \(\widetilde\pi_{\mc V^\#}\) induces an isomorphism
\(
    \overline\pi_{\mc V^\#}:
    \mc D_{\max}(\Delta_X)/\mc D_{X,\mc V}
    \longrightarrow
    \mc V^\#.
\)
Consequently,
\[
    \Gamma_{\mc V,\mc V^\#}^X(\lambda)
    =
    \overline Y_{X,\mc V}(\lambda)
    \circ
    \overline\pi_{\mc V^\#}^{-1}.
\]

\begin{defn}
Fix a Lagrangian pair \((\mc V,\mc V^\#)\). The Weyl function
associated with this pair is the operator-valued function
\[
    M_{\mc V,\mc V^\#}^X:
    \rho(\Delta_{X,\mc V})
    \longrightarrow
    \operatorname{Hom}_{\mb C}(\mc V^\#,\mc V)
\]
defined by
\(
    M_{\mc V,\mc V^\#}^X(\lambda)
    =
    \widetilde\pi_{\mc V}
    \circ
    \Gamma_{\mc V,\mc V^\#}^X(\lambda).
\)
Equivalently, for \(h^\#\in\mc V^\#\), the unique formal solution
\(
    u_\lambda
    =
    \Gamma_{\mc V,\mc V^\#}^X(\lambda)h^\#
\)
has boundary data
\(
    \pi_S(u_\lambda)
    =
    M_{\mc V,\mc V^\#}^X(\lambda)h^\#+h^\#.
\)
The Weyl function extends meromorphically in \(\lambda\), with possible
poles contained in the spectrum of \(\Delta_{X,\mc V}\).
\end{defn}

The preceding construction depends on a Lagrangian pair, whereas the
underlying boundary data do not. We now make the relation between the
Weyl function and the boundary-data subspace precise. Set
\(m:=\tfrac12\dim_{\mb C}V_S\) and
\(\Sigma_X:=\bigcup_{\mc V\in\operatorname{LGr}(V_S,\omega_S)}
\rho(\Delta_{X,\mc V})\).
This is an open subset of \(\mb C\) containing
\(\mb C\setminus\mb R\), since every \(\Delta_{X,\mc V}\) is
self-adjoint. We consider the geometry of \(\mc K_X(\lambda)\)
for \(\lambda\in\Sigma_X\).

\begin{prop}\label{prop:boundary-data-weyl-graph}
Fix a Lagrangian pair \((\mc V,\mc V^\#)\), and let
\(\lambda\in\rho(\Delta_{X,\mc V})\).
The boundary trace restricts to a linear isomorphism
\(\pi_S:\ker(\Delta_X-\lambda)\to\mc K_X(\lambda)\).
Moreover, \(\dim_{\mb C}\mc K_X(\lambda)=m\),
\(\mc K_X(\lambda)\cap\mc V=\{0\}\), and
\[
\mc K_X(\lambda)
=\left\{
M_{\mc V,\mc V^\#}^X(\lambda)h^\#+h^\#:
h^\#\in\mc V^\#
\right\}.
\]
Thus the boundary-data subspace is the graph of the Weyl function
relative to \(V_S=\mc V\oplus\mc V^\#\).
\end{prop}

\begin{proof}
Surjectivity of the restricted trace follows from the definition
of \(\mc K_X(\lambda)\). If a formal solution \(u\) has
\(\pi_Su=0\), then \(u\in\mc D_{\min}(\Delta_X)
\subset\mc D_{X,\mc V}\). Since
\(\lambda\in\rho(\Delta_{X,\mc V})\), this implies \(u=0\).
The trace is therefore injective.

Every formal solution is uniquely of the form
\(u=\Gamma_{\mc V,\mc V^\#}^X(\lambda)h^\#\), and its boundary
trace is \(M_{\mc V,\mc V^\#}^X(\lambda)h^\#+h^\#\).
This proves the graph representation. Projection onto
\(\mc V^\#\) is consequently an isomorphism on
\(\mc K_X(\lambda)\), which gives both its dimension and its
transversality to \(\mc V\).
\end{proof}

Let \(\operatorname{Gr}_m(V_S)\) denote the complex Grassmannian
of \(m\)-dimensional subspaces of \(V_S\). We use the graph-chart
notation of \cite[Section~6]{LiouWeyl2026}. For
\(E,F\in\operatorname{Gr}_m(V_S)\) with \(V_S=E\oplus F\), set
\(\mc U_E:=\{\mc L\in\operatorname{Gr}_m(V_S):
\mc L\cap E=\{0\}\}\).
Let \(P_E\) and \(P_F\) be the corresponding projections.
For \(\mc L\in\mc U_E\), define
\(A_{\mc L}:=P_E\circ(P_F|_{\mc L})^{-1}\).
Then \(\mc L=\operatorname{graph}_{(E,F)}(A_{\mc L})
:=\{A_{\mc L}h+h:h\in F\}\), and
\[
\kappa_{(E,F)}:\mc U_E\longrightarrow
\operatorname{Hom}_{\mb C}(F,E),
\qquad \mc L\longmapsto A_{\mc L},
\]
is a coordinate chart.

For \(\lambda\in\rho(\Delta_{X,\mc V})\),
Proposition~\ref{prop:boundary-data-weyl-graph} gives
\(\mc K_X(\lambda)\in\mc U_{\mc V}\). Its graph coordinate is
\[
\kappa_{(\mc V,\mc V^\#)}\bigl(\mc K_X(\lambda)\bigr)
=M_{\mc V,\mc V^\#}^X(\lambda).
\]
The resolvent construction of \(\Gamma_{\mc V,\mc V^\#}^X\)
shows that this coordinate depends holomorphically on \(\lambda\)
in \(\rho(\Delta_{X,\mc V})\). Indeed, one may choose a fixed
linear lift of \(\mc V^\#\) into the maximal domain and use the
holomorphic resolvent in the displayed construction of the formal
solution. Since the resolvent sets cover \(\Sigma_X\), this proves
that \(\mc K_X:\Sigma_X\to\operatorname{Gr}_m(V_S)\) is
holomorphic. This map is the Weyl curve of \(X\).
The Weyl functions associated with different Lagrangian pairs
are therefore local graph representations of this same intrinsic
family of boundary-data subspaces.

We record the corresponding change-of-basis convention for later
matrix calculations. Ordered bases are regarded as row vectors and
coordinates as column vectors. For a linear map \(T:E\to F\),
\([T]_\beta^\gamma\) denotes its matrix with domain basis
\(\beta\) and codomain basis \(\gamma\); thus
\([Tu]_\gamma=[T]_\beta^\gamma[u]_\beta\).
Choose bases \((\beta,\beta^\#)\) and
\((\gamma,\gamma^\#)\) adapted to Lagrangian decompositions
\(V_S=\mc V\oplus\mc V^\#=\mc W\oplus\mc W^\#\), and write
\[
    \bigl[\id_{V_S}\bigr]
_{(\beta,\beta^\#)}^{(\gamma,\gamma^\#)}
=\begin{bmatrix}A&B\\C&D\end{bmatrix}.
\]
In particular, \((\beta,\beta^\#)\) equals
\((\gamma,\gamma^\#)\begin{bmatrix}A&B\\C&D\end{bmatrix}\).
Let \(S_\beta\) and \(S_\gamma\) be the matrices of the Weyl
functions for these two choices, at a common resolvent point.
A graph vector with old coordinates
\(\binom{S_\beta h}{h}\) has new coordinates
\(\binom{(AS_\beta+B)h}{(CS_\beta+D)h}\).
Transversality to \(\mc W\) makes \(CS_\beta+D\) invertible,
so the change of graph coordinates is
\begin{equation}\label{eq:weyl-chart-transition}
S_\gamma=(AS_\beta+B)(CS_\beta+D)^{-1}.
\end{equation}
This formula accounts for changes of the singular complement as
well as changes of basis within a fixed decomposition.

The Green form relates the boundary-data subspaces at conjugate
spectral values.
For a subspace \(E\subset V_S\), write
\(E^{\perp_{\omega_S}}:=\{\xi\in V_S:
\omega_S(\xi,\eta)=0\text{ for every }\eta\in E\}\).

\begin{prop}\label{prop:boundary-data-green-pairing}
For every \(\lambda\in\Sigma_X\),
\(\mc K_X(\lambda)^{\perp_{\omega_S}}
=\mc K_X(\overline\lambda)\).
In particular, \(\mc K_X(\lambda)\) is Lagrangian when
\(\lambda\in\Sigma_X\cap\mb R\).
\end{prop}

\begin{proof}
For \(u\in\ker(\Delta_X-\lambda)\) and
\(v\in\ker(\Delta_X-\mu)\), the Green identity gives
\[
\omega_S(\pi_Su,\pi_Sv)
=(\lambda-\overline\mu)\langle u,v\rangle_{L^2(X')}.
\]
Taking \(\mu=\overline\lambda\) shows that the two boundary-data
subspaces are orthogonal. Both have dimension \(m\), since
\(\Sigma_X\) is invariant under conjugation. Nondegeneracy of
\(\omega_S\) then gives the stated equality. For real
\(\lambda\), the space equals its own symplectic orthogonal and
is therefore Lagrangian.
\end{proof}

In the following sections, the \(S\)-matrix is the matrix
representation of the Weyl function in chosen boundary bases.
It therefore records the same boundary-data subspace in explicit
coordinates. Its dependence on the Lagrangian pair, including the
choice of the complementary subspace, will be relevant when we
compare local conic coordinates.

\section{Isomorphisms of Conic Riemannian Surfaces}\label{sec:conic-isomorphisms}
We work here with oriented smooth surfaces and conic Riemannian
metrics. This viewpoint is needed in
Section~\ref{sec:resolvent-variation}: the smooth trivializations
preserve orientation and the conic sets, but need not preserve the
original complex structures. Equipping the source with the
pullback metric makes each trivialization an isometry. We establish
the resulting transport of Laplacians, domains, Green forms,
Friedrichs extensions, and boundary symplectic spaces.

A compact conic Riemannian surface is a compact connected oriented
smooth surface \(X\) without boundary, equipped with a smooth
symmetric covariant two-tensor \(ds_X^2\) and a finite subset
\(S\subset X\), such that \(ds_X^2\) is a Riemannian metric on
\(X':=X\setminus S\). For each \(P\in S\), we require positively
oriented smooth local coordinates \((x,y)\) centered at \(P\),
an integer \(n_P\geq2\), and a smooth positive function \(\rho_P\)
such that, writing \(z=x+\sqrt{-1}\,y\),
\[
ds_X^2=\rho_P(z)|z|^{2n_P-2}|dz|^2
\]
near \(P\), where \(|dz|^2=dx^2+dy^2\).
The operators and boundary data are defined as in the preceding
section, starting from the Laplace--Beltrami operator on
\(C_c^\infty(X')\subset L^2(X',dA_X)\).
They depend only on the Riemannian structure of \(X'\); their
asymptotic descriptions use the conic coordinates above.

Let \((X,S,ds_X^2)\) and \((Y,T,ds_Y^2)\) be conic Riemannian
surfaces. We say that \(X\) and \(Y\) are isomorphic if there exists
an orientation-preserving diffeomorphism \(\Theta:X\to Y\)
satisfying \(ds_X^2=\Theta^*ds_Y^2\). Such a map \(\Theta\) is called
an isomorphism. Since \(ds_X^2\) and \(ds_Y^2\) vanish precisely on
\(S\) and \(T\), respectively, this identity also implies
\(\Theta(S)=T\). In what follows, we fix an isomorphism \(\Theta:X\to Y\).

The isomorphism \(\Theta\) restricts to a diffeomorphism
\(X'\to Y'\), where \(X'=X\setminus S\) and \(Y'=Y\setminus T\).
Thus pullback defines a \(\mb C\)-algebra isomorphism
\(\Theta^*:C_c^\infty(Y')\to C_c^\infty(X')\), given by
\(\Theta^*u=u\circ\Theta\). Since \(ds_X^2=\Theta^*ds_Y^2\), we have
\(\Theta^*dA_Y=dA_X\). The change-of-variables formula therefore
gives \(\int_{X'}\Theta^*u\,dA_X=\int_{Y'}u\,dA_Y\) for every
\(u\in C_c^\infty(Y')\). Applying this identity to \(|u|^2\)
yields \(\|\Theta^*u\|_{L^2(X')}=\|u\|_{L^2(Y')}\).
Since the spaces of compactly supported smooth functions are
dense in the respective \(L^2\)-spaces, \(\Theta^*\) extends
uniquely to a unitary isomorphism
\(\Theta^*:L^2(Y',dA_Y)\to L^2(X',dA_X)\), with inverse
\((\Theta^{-1})^*\).

Denote by
\(\Delta_{X,c}:C_c^\infty(X')\to C_c^\infty(X')\) and
\(\Delta_{Y,c}:C_c^\infty(Y')\to C_c^\infty(Y')\)
the Laplace--Beltrami operators associated with \(ds_X^2\) and
\(ds_Y^2\), respectively. Since \(\Theta:X'\to Y'\) is an
isometry, the naturality of the Laplace--Beltrami operator gives
\(\Delta_{X,c}(\Theta^*u)=\Theta^*(\Delta_{Y,c}u)\) for every
\(u\in C_c^\infty(Y')\).

\begin{thm}\label{thm:conic-isomorphism}
Let \(\Theta:X\to Y\) be an isomorphism of compact conic
Riemannian surfaces. Its unitary pullback restricts to unitary
isomorphisms between the maximal domains, and between the minimal
domains, equipped with their graph norms. Moreover,
\[
\Delta_X\Theta^*u=\Theta^*\Delta_Yu,
\qquad u\in\mc D_{\max}(\Delta_Y).
\]
The induced map \(\widetilde\Theta^*[u]=[\Theta^*u]\) on the
quotients is a symplectic isomorphism
\(\widetilde\Theta^*:(V_T,\omega_T)\to(V_S,\omega_S)\), satisfying
\(\widetilde\Theta^*\circ\pi_T=\pi_S\circ\Theta^*\).
\end{thm}

\begin{proof}
Let \(u\in\mc D_{\max}(\Delta_Y)\) and
\(v=\Delta_Yu\). For \(\psi\in C_c^\infty(X')\), set
\(\eta=(\Theta^{-1})^*\psi\). The intertwining identity on
test functions and unitarity give
\[
\langle\Theta^*u,\Delta_{X,c}\psi\rangle_{L^2(X')}
=\langle u,\Delta_{Y,c}\eta\rangle_{L^2(Y')}
=\langle v,\eta\rangle_{L^2(Y')}
=\langle\Theta^*v,\psi\rangle_{L^2(X')}.
\]
Thus \(\Theta^*u\in\mc D_{\max}(\Delta_X)\) and
\(\Delta_X\Theta^*u=\Theta^*v\). Applying the same argument to
\(\Theta^{-1}\) proves surjectivity. Unitarity then implies
\(\|\Theta^*u\|_{\Delta_X}=\|u\|_{\Delta_Y}\).
Since pullback maps \(C_c^\infty(Y')\) onto \(C_c^\infty(X')\),
it also maps their graph-norm closures onto each other. Hence
\(\Theta^*\mc D_{\min}(\Delta_Y)=\mc D_{\min}(\Delta_X)\),
and the induced quotient map is well defined and invertible.

For \(u,v\in\mc D_{\max}(\Delta_Y)\), the same two identities
give \(\mk q_X(\Theta^*u,\Theta^*v)=\mk q_Y(u,v)\).
Passing to the quotient proves preservation of the boundary form.
\end{proof}

These maps fit into the commutative diagram with exact rows
\[
\begin{tikzcd}[column sep=large, row sep=large]
0\arrow[r]&\mc D_{\min}(\Delta_Y)\arrow[r,hook]\arrow[d,"\Theta^*"']
&\mc D_{\max}(\Delta_Y)\arrow[r,"\pi_T"]\arrow[d,"\Theta^*"']
&V_T\arrow[r]\arrow[d,"\widetilde\Theta^*"']&0\\
0\arrow[r]&\mc D_{\min}(\Delta_X)\arrow[r,hook]
&\mc D_{\max}(\Delta_X)\arrow[r,"\pi_S"]&V_S\arrow[r]&0.
\end{tikzcd}
\]
For composable isomorphisms \(\Theta_1\) and \(\Theta_2\),
\((\Theta_2\circ\Theta_1)^*=\Theta_1^*\circ\Theta_2^*\),
and the same rule holds for the induced quotient maps.
Pullback by the identity is the identity. Thus the maximal and
minimal graph domains, and the boundary symplectic spaces, define
contravariant functors on the category of compact conic Riemannian surfaces with
isomorphisms as morphisms. Their values are respectively Hilbert
spaces with unitary isomorphisms and boundary symplectic spaces
with symplectic isomorphisms.

We next describe this boundary isomorphism in local conic
coordinates. For each \(Q\in T\), choose a coordinate
neighborhood \(U_Q\), containing no other point of \(T\), with
a positively oriented smooth coordinate
\(w_Q=x_Q+\sqrt{-1}\,y_Q\) centered at \(Q\), such that
\[
ds_Y^2=\rho_Q(w_Q)|w_Q|^{2n_Q-2}|dw_Q|^2,
\]
where \(\rho_Q\) is smooth and positive. Write
\(w_Q=r_Qe^{\sqrt{-1}\theta_Q}\) on
\(U_Q':=U_Q\setminus\{Q\}\).

Let \(J_n:=\{j\in\mb Z:|j|\leq n-1\}\), ordered as
\(0,1,-1,\ldots,n-1,-(n-1)\). Recall that the normalized
critical asymptotic terms at \(Q\) are
\(f_{Q,0}=1/\sqrt{2\pi}\) and
\(f_{Q,0}^{\#}=-(\log r_Q)/\sqrt{2\pi}\), together with
\[
f_{Q,j}
=\frac{r_Q^{|j|}e^{\sqrt{-1}j\theta_Q}}{\sqrt{4\pi|j|}},
\qquad
f_{Q,j}^{\#}
=\frac{r_Q^{-|j|}e^{\sqrt{-1}j\theta_Q}}{\sqrt{4\pi|j|}},
\qquad 1\leq |j|\leq n_Q-1.
\]
Thus \(V_Q=\operatorname{span}_{\mb C}
\{f_{Q,j},f_{Q,j}^{\#}:j\in J_{n_Q}\}\).
As above, these terms represent boundary classes after
multiplication by a smooth cutoff equal to \(1\) near \(Q\).

For \(P\in S\), set \(Q:=\Theta(P)\),
\(U_P:=\Theta^{-1}(U_Q)\), and
\(z_P:=w_Q\circ\Theta\). Since \(\Theta\) preserves orientation,
\((U_P,z_P)\) is a positively oriented smooth coordinate chart
centered at \(P\). The metric identity gives
\(ds_X^2=\Theta^*ds_Y^2
=\rho_Q(z_P)|z_P|^{2n_Q-2}|dz_P|^2\).
In particular, \(n_P=n_Q\). The normalized asymptotic basis
\(\{\widehat f_{P,j},\widehat f_{P,j}^{\#}:j\in J_{n_P}\}\)
determined by \(z_P\) satisfies
\(\widehat f_{P,j}=\Theta^*f_{Q,j}\) and
\(\widehat f_{P,j}^{\#}=\Theta^*f_{Q,j}^{\#}\).
Here pullback of a local asymptotic term means composition
with \(\Theta\).

Fix \(Q\in T\), and let \(P:=\Theta^{-1}(Q)\).
For \(v\in\mc D_{\max}(\Delta_Y)\), the local asymptotic
decomposition recalled above gives \(v=f+R\) near \(Q\),
where \(f=\pi_Qv\in V_Q\). If \(\chi\) is a smooth cutoff
supported in a sufficiently small coordinate disc about \(Q\)
and equal to \(1\) near \(Q\), then \(\chi R\), extended by
zero, belongs to \(\mc D_{\min}(\Delta_Y)\).
Pulling back gives \(\Theta^*v=\Theta^*f+\Theta^*R\) near
\(P\), and Theorem~\ref{thm:conic-isomorphism} implies
\((\Theta^*\chi)\Theta^*R=\Theta^*(\chi R)
\in\mc D_{\min}(\Delta_X)\).
Since \(\Theta^*f\in V_P\), the induced boundary map sends
the summand \(V_Q\) onto \(V_P\). Denote its restriction by
\(\widetilde\Theta_P^*:V_{\Theta(P)}\to V_P\).
In the chosen local representatives,
\(\widetilde\Theta_P^*f=\Theta^*f\).

With the source summands indexed through the bijection
\(\Theta:S\to T\), we therefore have
\[
\widetilde\Theta^*
=\bigoplus_{P\in S}\widetilde\Theta_P^*,
\qquad
\widetilde\Theta_P^*\circ\pi_{\Theta(P)}
=\pi_P\circ\Theta^*.
\]
Each \(\widetilde\Theta_P^*\) is a symplectic isomorphism,
with inverse induced by \(\Theta^{-1}\).

Since \(\widetilde\Theta^*\) is a symplectic isomorphism,
it maps the Lagrangian subspaces of \(V_T\) bijectively onto
those of \(V_S\).

This correspondence also transports the associated self-adjoint
extensions.

\begin{lem}
Let \(\mc V\in\operatorname{LGr}(V_T,\omega_T)\).
Then \(\Delta_{Y,\mc V}\) and
\(\Delta_{X,\widetilde\Theta^*\mc V}\) are unitarily
equivalent. For every \(\lambda\in\mb C\), pullback
restricts to a linear isomorphism
\(\Theta^*:\ker(\Delta_{Y,\mc V}-\lambda)
\to\ker(\Delta_{X,\widetilde\Theta^*\mc V}-\lambda)\).
In particular, the two operators have the same spectrum
and resolvent set, and their eigenvalues have the same
multiplicities.
\end{lem}

\begin{proof}
The trace identity
\(\pi_S\Theta^*=\widetilde\Theta^*\pi_T\), applied also to
\(\Theta^{-1}\), gives
\(\Theta^*\mc D_{Y,\mc V}=\mc D_{X,\widetilde\Theta^*\mc V}\).
Hence Theorem~\ref{thm:conic-isomorphism} yields
\(\Delta_{X,\widetilde\Theta^*\mc V}
=\Theta^*\Delta_{Y,\mc V}(\Theta^{-1})^*\) as an equality of
unbounded operators, including their domains.
Unitary equivalence gives the eigenspace and spectral assertions.
\end{proof}

The intertwining identity for the maximal Laplacians also gives
\(\Theta^*\ker(\Delta_Y-\lambda)\allowbreak
=\ker(\Delta_X-\lambda)\).
Since boundary trace commutes with pullback, it follows that
\(\widetilde\Theta^*\mc K_Y(\lambda)=\mc K_X(\lambda)\) for
\(\lambda\in\mb C\).
Thus the boundary-data subspaces are transported by the boundary
symplectic isomorphism. We now express this relation in terms of
Weyl functions and their associated \(S\)-matrices.

\begin{prop}\label{prop:weyl-function-functoriality}
Let \((\mc V,\mc V^\#)\) be a Lagrangian pair in
\((V_T,\omega_T)\). Then
\((\widetilde\Theta^*\mc V,\widetilde\Theta^*\mc V^\#)\)
is a Lagrangian pair in \((V_S,\omega_S)\).
For every \(\lambda\in\rho(\Delta_{Y,\mc V})
=\rho(\Delta_{X,\widetilde\Theta^*\mc V})\),
\[
M^X_{\widetilde\Theta^*\mc V,\widetilde\Theta^*\mc V^\#}(\lambda)
=
\bigl(\widetilde\Theta^*|_{\mc V}\bigr)
\circ M^Y_{\mc V,\mc V^\#}(\lambda)
\circ\bigl(\widetilde\Theta^*|_{\mc V^\#}\bigr)^{-1}.
\]
\end{prop}

\begin{proof}
A symplectic isomorphism preserves Lagrangian subspaces and
direct sums. Thus, setting \(\mc W:=\widetilde\Theta^*\mc V\)
and \(\mc W^\#:=\widetilde\Theta^*\mc V^\#\), we obtain
a Lagrangian pair \((\mc W,\mc W^\#)\) in \((V_S,\omega_S)\).

Fix \(\lambda\in\rho(\Delta_{Y,\mc V})
=\rho(\Delta_{X,\mc W})\) and \(h^\#\in\mc V^\#\).
Let \(v_\lambda:=\Gamma^Y_{\mc V,\mc V^\#}(\lambda)h^\#\).
The intertwining identity implies that
\(u_\lambda:=\Theta^*v_\lambda\) belongs to
\(\ker(\Delta_X-\lambda)\). Moreover,
\(\pi_Su_\lambda=\widetilde\Theta^*(\pi_Tv_\lambda)\), where
\(\pi_Tv_\lambda=M^Y_{\mc V,\mc V^\#}(\lambda)h^\#+h^\#\).
Hence the \(\mc W^\#\)-component of \(\pi_Su_\lambda\) is
\(k^\#:=\widetilde\Theta^*h^\#\).
By uniqueness of the formal solution with this prescribed
component, \(u_\lambda=\Gamma^X_{\mc W,\mc W^\#}(\lambda)k^\#\).
Comparing the \(\mc W\)-components of its boundary data gives
\[
M^X_{\mc W,\mc W^\#}(\lambda)k^\#
=\widetilde\Theta^*\bigl(
M^Y_{\mc V,\mc V^\#}(\lambda)h^\#\bigr).
\]
Since \(\widetilde\Theta^*|_{\mc V^\#}:\mc V^\#\to\mc W^\#\)
is a bijection, this identity proves the stated formula.
\end{proof}

For each \(Q\in T\), let
\(\mc F_{Y,Q}:=\operatorname{span}_{\mb C}
\{f_{Q,j}:j\in J_{n_Q}\}\) and
\(\mc F_{Y,Q}^{\#}:=\operatorname{span}_{\mb C}
\{f_{Q,j}^{\#}:j\in J_{n_Q}\}\).
Set \(\mc F_Y:=\bigoplus_{Q\in T}\mc F_{Y,Q}\)
and \(\mc F_Y^{\#}:=\bigoplus_{Q\in T}\mc F_{Y,Q}^{\#}\).
Then \(\mc F_Y\) is the Friedrichs subspace and
\(V_T=\mc F_Y\oplus\mc F_Y^{\#}\).
The singular complement \(\mc F_Y^{\#}\) is determined
by the chosen conic coordinates.

For \(\lambda\in\rho(\Delta_{Y,\mc F_Y})\), the Weyl
function \(M^Y_{\mc F_Y,\mc F_Y^{\#}}(\lambda):
\mc F_Y^{\#}\to\mc F_Y\) has blocks indexed by pairs
of conic points. Let
\(\iota_{Q'}:\mc F_{Y,Q'}^{\#}\hookrightarrow\mc F_Y^{\#}\)
be the canonical inclusion, and let
\(\operatorname{pr}_Q:\mc F_Y\to\mc F_{Y,Q}\)
be the canonical projection. Define
\[
M_{Q,Q'}^Y(\lambda)
:=\operatorname{pr}_Q\circ
M^Y_{\mc F_Y,\mc F_Y^{\#}}(\lambda)\circ\iota_{Q'}:
\mc F_{Y,Q'}^{\#}\longrightarrow\mc F_{Y,Q}.
\]
Thus the \(Q\)-component of
\(M^Y_{\mc F_Y,\mc F_Y^{\#}}(\lambda)(h_{Q'}^{\#})_{Q'\in T}\)
is \(\sum_{Q'\in T}M_{Q,Q'}^Y(\lambda)h_{Q'}^{\#}\).

Let \(\gamma_Q:=(f_{Q,j})_{j\in J_{n_Q}}\) and
\(\gamma_Q^{\#}:=(f_{Q,j}^{\#})_{j\in J_{n_Q}}\)
be the ordered bases of \(\mc F_{Y,Q}\) and
\(\mc F_{Y,Q}^{\#}\), respectively. For \(Q,Q'\in T\),
define the block \(S\)-matrix by
\[
S_{Q,Q'}^{Y,(\gamma_Q,\gamma_{Q'}^{\#})}(\lambda)
:=\bigl[M_{Q,Q'}^Y(\lambda)\bigr]_{\gamma_{Q'}^{\#}}^{\gamma_Q}
=\left(
S_{Q,Q',\ell p}^{Y,(\gamma_Q,\gamma_{Q'}^{\#})}(\lambda)
\right)_{\ell\in J_{n_Q},\,p\in J_{n_{Q'}}}.
\]
Equivalently, for \(p\in J_{n_{Q'}}\),
\[
M_{Q,Q'}^Y(\lambda)f_{Q',p}^{\#}
=\sum_{\ell\in J_{n_Q}}
S_{Q,Q',\ell p}^{Y,(\gamma_Q,\gamma_{Q'}^{\#})}(\lambda)
f_{Q,\ell}.
\]
Since \(|J_n|=2n-1\), this block is a
\((2n_Q-1)\times(2n_{Q'}-1)\) matrix.

Choose an ordering of \(T\), and let \(\gamma_Y\) and
\(\gamma_Y^{\#}\) be the corresponding concatenated
ordered bases of \(\mc F_Y\) and \(\mc F_Y^{\#}\).
The full \(S\)-matrix is
\[
S^{Y,(\gamma_Y,\gamma_Y^{\#})}(\lambda)
:=\bigl[M^Y_{\mc F_Y,\mc F_Y^{\#}}(\lambda)
\bigr]_{\gamma_Y^{\#}}^{\gamma_Y}
=\left(
S_{Q,Q'}^{Y,(\gamma_Q,\gamma_{Q'}^{\#})}(\lambda)
\right)_{Q,Q'\in T}.
\]

For each \(P\in S\), let
\(\mc F_{X,P}:=\operatorname{span}_{\mb C}
\{\widehat f_{P,j}:j\in J_{n_P}\}\) and
\(\mc F_{X,P}^{\#}:=\operatorname{span}_{\mb C}
\{\widehat f_{P,j}^{\#}:j\in J_{n_P}\}\).
Set \(\mc F_X:=\bigoplus_{P\in S}\mc F_{X,P}\)
and \(\mc F_X^{\#}:=\bigoplus_{P\in S}\mc F_{X,P}^{\#}\).
Define the ordered bases
\(\widehat\gamma_P:=(\widehat f_{P,j})_{j\in J_{n_P}}\)
and
\(\widehat\gamma_P^{\#}:=(\widehat f_{P,j}^{\#})_{j\in J_{n_P}}\).
We use the analogous notation for the blocks and matrices
on \(X\) with respect to these bases.

Because the conic coordinates on \(X\) were transported
from \(Y\), the local boundary map satisfies
\(\widetilde\Theta_P^*f_{\Theta(P),j}=\widehat f_{P,j}\)
and
\(\widetilde\Theta_P^*f_{\Theta(P),j}^{\#}
=\widehat f_{P,j}^{\#}\).
It therefore maps \(\mc F_{Y,\Theta(P)}\) onto
\(\mc F_{X,P}\) and \(\mc F_{Y,\Theta(P)}^{\#}\)
onto \(\mc F_{X,P}^{\#}\). Taking direct sums gives
\[
\widetilde\Theta^*\mc F_Y=\mc F_X,
\qquad
\widetilde\Theta^*\mc F_Y^{\#}=\mc F_X^{\#}.
\]
Thus the Friedrichs subspaces are preserved, as are the
singular complements defined by these transported coordinates.
In particular, the Friedrichs extensions satisfy
\(\Delta_{X,\mc F_X}
=\Theta^*\Delta_{Y,\mc F_Y}(\Theta^{-1})^*\).

\begin{cor}\label{cor:s-matrix-functoriality}
Let \(P,P'\in S\), and set \(Q=\Theta(P)\) and
\(Q'=\Theta(P')\). With the ordered bases defined above,
for every \(\lambda\in\rho(\Delta_{Y,\mc F_Y})
=\rho(\Delta_{X,\mc F_X})\),
\[
S_{P,P'}^{X,(\widehat\gamma_P,\widehat\gamma_{P'}^{\#})}(\lambda)
=S_{Q,Q'}^{Y,(\gamma_Q,\gamma_{Q'}^{\#})}(\lambda).
\]
\end{cor}

\begin{proof}
Proposition~\ref{prop:weyl-function-functoriality} gives
\(M^X_{\mc F_X,\mc F_X^{\#}}(\lambda)
\circ\widetilde\Theta^*
=\widetilde\Theta^*\circ
M^Y_{\mc F_Y,\mc F_Y^{\#}}(\lambda)\)
on \(\mc F_Y^{\#}\).
Since \(\widetilde\Theta^*\) respects the local summands,
the corresponding blocks satisfy
\[
M_{P,P'}^X(\lambda)\circ\widetilde\Theta_{P'}^*
=\widetilde\Theta_P^*\circ M_{Q,Q'}^Y(\lambda)
\]
on \(\mc F_{Y,Q'}^{\#}\).
For \(p\in J_{n_{Q'}}=J_{n_{P'}}\), applying this identity
to \(f_{Q',p}^{\#}\) gives
\[
\begin{aligned}
M_{P,P'}^X(\lambda)\widehat f_{P',p}^{\#}
&=\widetilde\Theta_P^*\bigl(
M_{Q,Q'}^Y(\lambda)f_{Q',p}^{\#}\bigr)
\\
&=\sum_{\ell\in J_{n_Q}}
S_{Q,Q',\ell p}^{Y,(\gamma_Q,\gamma_{Q'}^{\#})}(\lambda)
\widehat f_{P,\ell}.
\end{aligned}
\]
Since \(n_P=n_Q\) and \(\widehat\gamma_P\) is an
ordered basis of \(\mc F_{X,P}\), comparison with the
definition of the block matrix on \(X\) proves the result.
\end{proof}

The operators act between different spaces, but their matrices
coincide in transported bases. Transporting the ordering of the
conic points also gives equality of the full \(S\)-matrices.
The Friedrichs extension is intrinsic; the matrix identities use
the transported conic coordinates and ordered bases.

\section{The Davies--Gaffney Estimate}\label{sec:davies-gaffney}
In the subsequent analysis, we need off-diagonal estimates for the
heat semigroup and the resolvent between subsets separated by a
positive distance. We first recall the standard pairing form of the
Davies--Gaffney estimate and then derive the localized operator-norm
estimate used below. The pairing estimate follows from
\cite[Theorem~3.3]{CoulhonSikora2008}; for the Dirichlet-form
formulation, see \cite[Theorem~2.8]{HinoRamirez2003}.

The point of this estimate in the present paper is to separate local
conic geometry from global geometry.  A Weyl solution whose singular
boundary data are prescribed at one cone point is obtained from a
localized resolvent.  After the cutoffs have been chosen with disjoint
supports, the Davies--Gaffney estimate shows that the contribution
propagating across the intervening positive distance is rapidly
decreasing on the negative spectral axis.  Consequently, the
high-energy behavior of the corresponding diagonal Weyl block is
determined, up to a rapidly decreasing error, by any model that agrees
with the original metric near that cone point.  This is the mechanism
used in Section~\ref{sec:local-model-comparison} to compare the global
surface with the spherical model.  We write \(O(r^{-\infty})\) for a
quantity that is \(O(r^{-N})\) for every \(N>0\); the notation
\(O(|\lambda|^{-\infty})\) has the analogous meaning as
\(\lambda\to-\infty\).

For Banach spaces \(E\) and \(F\), let \(\mathcal L(E,F)\)
denote the Banach space of bounded linear operators from \(E\) to
\(F\), equipped with the operator norm. We write
\(\mathcal L(E):=\mathcal L(E,E)\). For \(g\in L^\infty(X')\), let
\(M_g\in\mathcal L(L^2(X'))\) be the multiplication operator defined by
\((M_gu)(x):=g(x)u(x)\) for \(u\in L^2(X')\). Its operator norm is
\(\|M_g\|_{\mathcal L(L^2(X'))}=\|g\|_{L^\infty(X')}\). For a measurable
subset \(E\subset X\), we use the same notation \(\chi_E\) for the
restriction to \(X'\) of its characteristic function.

\begin{prop}[Davies--Gaffney estimate]\label{prop:davies-gaffney-pairing}
Let \(L:=\Delta_{X,\mc F_X}\), regarded as a nonnegative self-adjoint
operator on \(L^2(X')\), and let \(d\) denote the metric on \(X\)
obtained by completing the Riemannian distance induced by \(ds_X^2\)
on \(X'\). Let \(U_1,U_2\subset X'\) be open, put
\(r:=d(U_1,U_2)\), and let \(f_i\in L^2(X')\) vanish almost
everywhere outside \(U_i\). Then, for every \(t>0\),
\[
\left|\left\langle e^{-tL}f_1,f_2\right\rangle_{L^2(X')}\right|
\leq
\exp\left(-\frac{r^2}{4t}\right)
\|f_1\|_{L^2(X')}\|f_2\|_{L^2(X')}.
\]
\end{prop}
\begin{proof}
The smooth locus \(X'\) is a Riemannian manifold, not necessarily
complete, and \(L\) is the Friedrichs extension of its nonnegative
Laplace--Beltrami operator initially defined on \(C_c^\infty(X')\).
Its closed quadratic form is the strongly local regular Dirichlet
form obtained by closing the Riemannian Dirichlet energy, and the
associated intrinsic distance is the Riemannian distance on \(X'\).
Thus metric completeness is not required for the Friedrichs
realization considered here.
Therefore \cite[Theorem~3.3]{CoulhonSikora2008}, with \(M=X'\) and
zero potential, gives precisely the pairing estimate
\cite[(3.2)]{CoulhonSikora2008}. The distance in that result is the
intrinsic Riemannian distance on \(X'\), which is the restriction of
the completion metric \(d\).
\end{proof}

\begin{cor}\label{cor:davies-gaffney-cutoffs}
For closed subsets \(E,F\subset X\), set
\(d(E,F):=\inf\{d(x,y):x\in E,\ y\in F\}\).
If \(d(E,F)>0\), then, for every \(t>0\),
\[
\left\|M_{\chi_F}e^{-tL}M_{\chi_E}\right\|_{\mathcal L(L^2(X'))}
\leq \exp\left(-\frac{d(E,F)^2}{4t}\right).
\]
\end{cor}
\begin{proof}
We may assume that \(E\) and \(F\) are nonempty. Put
\(\delta:=d(E,F)>0\). For \(0<\varepsilon<\delta/2\), define
\(E_\varepsilon:=\{x\in X':d(x,E)<\varepsilon\}\) and
\(F_\varepsilon:=\{x\in X':d(x,F)<\varepsilon\}\).
These sets are open in \(X'\), and the triangle inequality gives
\(d(E_\varepsilon,F_\varepsilon)\geq\delta-2\varepsilon\).

For \(u,v\in L^2(X')\), the functions \(M_{\chi_E}u\) and
\(M_{\chi_F}v\) vanish outside \(E_\varepsilon\) and
\(F_\varepsilon\), respectively. Applying
Proposition~\ref{prop:davies-gaffney-pairing},
we obtain
\[
\left|
\left\langle
e^{-tL}M_{\chi_E}u,M_{\chi_F}v
\right\rangle_{L^2(X')}
\right|
\leq
\exp\left(-\frac{(\delta-2\varepsilon)^2}{4t}\right)
\|u\|_{L^2(X')}\|v\|_{L^2(X')}.
\]
Letting \(\varepsilon\downarrow0\) and taking the supremum over
unit vectors \(u,v\) proves the assertion.
\end{proof}

The heat-semigroup estimate also controls spectral derivatives
of the localized resolvent.

\begin{prop}\label{prop:resolvent-power-decay}
Let \(\chi_1,\chi_2\in C^\infty(X)\) have disjoint supports,
and let \(q\geq0\) be an integer. For every \(M>0\),
\[
\left\|M_{\chi_1}\frac{d^q}{dr^q}R_{\mc F_X}(-r)
M_{\chi_2}\right\|_{\mathcal L(L^2(X'))}
=O(r^{-M})\qquad(r\to\infty).
\]
\end{prop}

\begin{proof}
Put \(L=\Delta_{X,\mc F_X}\) and let \(\delta>0\) be the
distance between the two supports. The heat-semigroup representation
gives
\[
\frac{d^q}{dr^q}R_{\mc F_X}(-r)
=q!\bigl(R_{\mc F_X}(-r)\bigr)^{q+1}
=(-1)^{q+1}\int_0^\infty s^q e^{-rs}e^{-sL}\,ds.
\]
Here \(\bigl(R_{\mc F_X}(-r)\bigr)^{q+1}\) denotes the
\((q+1)\)-st operator power of the resolvent evaluated at the
spectral parameter \(-r\).
After inserting the cutoff multipliers,
Corollary~\ref{cor:davies-gaffney-cutoffs}
and \(rs+\delta^2/(4s)\geq rs/2+\delta\sqrt{r/2}\) imply
\[
\left\|M_{\chi_1}\frac{d^q}{dr^q}R_{\mc F_X}(-r)
M_{\chi_2}\right\|
\leq\|\chi_1\|_\infty\|\chi_2\|_\infty
q!\left(\frac2r\right)^{q+1}e^{-\delta\sqrt{r/2}}.
\]
This proves the assertion. In particular, the constants can be
chosen uniformly when the cutoff norms are bounded and the distance
between their supports has a fixed positive lower bound.
\end{proof}

\begin{lem}\label{lem:weyl-solution-decay}
Let \(K\Subset X'\), fix a conic point \(P\in S\) and
\(h^\#\in\mc F_{X,P}^\#\), and set
\(u_{-r}:=\Gamma_{\mc F_X,\mc F_X^\#}^X(-r)h^\#\) for \(r>0\).
For every integer \(q\geq0\) and every \(M>0\),
\[
\|\partial_r^q u_{-r}\|_{L^2(K)}=O(r^{-M})
\qquad(r\to\infty).
\]
\end{lem}

\begin{proof}
Choose a cutoff representative \(H\in\mc D_{\max}(\Delta_X)\)
of \(h^\#\), supported in a small neighborhood of \(P\)
disjoint from \(K\). Let \(\chi_1\) equal one on \(K\), and
choose \(\chi_2\) equal to one on \(\supp H\), with their
supports disjoint. The construction of the Weyl solution gives
\[
M_{\chi_1}u_{-r}
=M_{\chi_1}R_{\mc F_X}(-r)M_{\chi_2}(\Delta_X+r)H.
\]
For \(q\geq1\), differentiation yields
\[
\begin{aligned}
M_{\chi_1}\partial_r^q u_{-r}
={}&M_{\chi_1}(\partial_r^qR_{\mc F_X}(-r))
M_{\chi_2}(\Delta_X+r)H\\
&+qM_{\chi_1}(\partial_r^{q-1}R_{\mc F_X}(-r))M_{\chi_2}H.
\end{aligned}
\]
Apply Proposition~\ref{prop:resolvent-power-decay} with decay
order \(M+1\) to the first term and with decay order \(M\) to the
second. Since
\(
\|(\Delta_X+r)H\|_{L^2}\leq(1+r)\|H\|_{\Delta_X},
\)
the first term is \(O(r^{-M-1})(1+r)=O(r^{-M})\), whereas the
second is \(O(r^{-M})\). The same argument applied to the
undifferentiated identity proves the case \(q=0\).
\end{proof}


\section{The Spherical Conic Model}\label{sec:spherical-model}
The purpose of this section is to compute the model Weyl matrix in
the two boundary bases used later. We first recall the spherical
covering model and its Weyl function. We then compare the spherical
and distinguished conic coordinates and determine the resulting
change of Weyl matrix.

We use the explicit conic model and its spectral data developed in
\cite{LiouModel2026}. The special-function identities underlying these
formulas are recorded in \cite[Chapters~5 and~14]{DLMF}.

Let \(n\geq2\) be an integer. On the source \(\mb P^1\), let
\(\xi\) and \(\eta=\xi^{-1}\) denote the standard affine
coordinates. Consider the map
\(\Phi_n\colon\mb P^1\to\mb P^1\) defined by
\(\Phi_n(\xi)=\xi^n\). If
\[
    ds_{\mathrm{rd}}^2
    :=
    \frac{4|dx|^2}{(1+|x|^2)^2}
\]
denotes the round metric on the target \(\mb P^1\), then
\[
    ds_n^2
    :=
    \Phi_n^*ds_{\mathrm{rd}}^2
    =
    \frac{4n^2|\xi|^{2n-2}}
         {(1+|\xi|^{2n})^2}|d\xi|^2.
\]
In the coordinate \(\eta=\xi^{-1}\) near infinity, the metric has
the same form:
\[
    ds_n^2
    =
    \frac{4n^2|\eta|^{2n-2}}
         {(1+|\eta|^{2n})^2}|d\eta|^2.
\]
Thus,
\(\widehat X_n=(\mb P^1,\{0,\infty\},ds_n^2)\) is a compact conic
Riemann surface.

Let \(V_{\widehat X_n,0}\) and \(V_{\widehat X_n,\infty}\) be the
critical asymptotic spaces of \(\widehat X_n\) at \(0\) and
\(\infty\), respectively, and set
\(V_{\widehat X_n,\{0,\infty\}}
:=V_{\widehat X_n,0}\oplus V_{\widehat X_n,\infty}\).
For \(j\in J_n\), write \(g_j\) and \(g_j^\#\) for the
normalized regular and singular terms of
Section~\ref{sec:conic-analysis}, viewed as functions of a local
coordinate; in particular,
\(g_0=1/\sqrt{2\pi}\) and
\(g_0^\#(\xi)=-\log|\xi|/\sqrt{2\pi}\).
Using the coordinate \(\xi\) at \(0\), define
\(g_{0,j}:=g_j(\xi)\) and \(g_{0,j}^\#:=g_j^\#(\xi)\).
Similarly, using \(\eta\) at \(\infty\), define
\(g_{\infty,j}:=g_j(\eta)\) and
\(g_{\infty,j}^\#:=g_j^\#(\eta)\).

Then \(V_{\widehat X_n,0}\) and \(V_{\widehat X_n,\infty}\) have
the ordered bases \((\beta_0,\beta_0^\#)\) and
\((\beta_\infty,\beta_\infty^\#)\), respectively, where
\(\beta_0=(g_{0,j})_j\), \(\beta_0^\#=(g_{0,j}^\#)_j\),
\(\beta_\infty=(g_{\infty,j})_j\), and
\(\beta_\infty^\#=(g_{\infty,j}^\#)_j\).

Let \(\mc F_{\widehat X_n}\) be the Friedrichs subspace of
\(V_{\widehat X_n,\{0,\infty\}}\), and let
\(\mc F_{\widehat X_n}^\#\) be the Lagrangian complement spanned
by \(\beta_0^\#\) and \(\beta_\infty^\#\). This complement
is determined by the chosen coordinates \(\xi\) and \(\eta\). For
\(\lambda\in\rho(\Delta_{\widehat X_n,\mc F_{\widehat X_n}})\),
the Weyl function
\[
    M_{\mc F_{\widehat X_n},\mc F_{\widehat X_n}^\#}^{\widehat X_n}
    (\lambda)
    \colon
    \mc F_{\widehat X_n}^\#
    \longrightarrow
    \mc F_{\widehat X_n}
\]
has the block decomposition
\[
    M_{\mc F_{\widehat X_n},\mc F_{\widehat X_n}^\#}^{\widehat X_n}
    (\lambda)
    =
    \begin{bmatrix}
        M_{00}^{\widehat X_n}(\lambda)
        &
        M_{0\infty}^{\widehat X_n}(\lambda)
        \\
        M_{\infty0}^{\widehat X_n}(\lambda)
        &
        M_{\infty\infty}^{\widehat X_n}(\lambda)
    \end{bmatrix}
\]
with respect to the decompositions
\[
    \mc F_{\widehat X_n}^\#
    =
    \mc F_{\widehat X_n,0}^\#
    \oplus
    \mc F_{\widehat X_n,\infty}^\#
\]
and
\[
    \mc F_{\widehat X_n}
    =
    \mc F_{\widehat X_n,0}
    \oplus
    \mc F_{\widehat X_n,\infty}.
\]
We abbreviate \(M_{00}^{\widehat X_n}(\lambda)\) and
\(M_{\infty\infty}^{\widehat X_n}(\lambda)\) as
\(M_0^{\widehat X_n}(\lambda)\) and
\(M_\infty^{\widehat X_n}(\lambda)\), respectively.

Let
\(J_n=\{j\in\mb Z:|j|\leq n-1\}\), ordered as
\(0,1,-1,\ldots,n-1,-(n-1)\). Define the local \(S\)-matrix of
\(\widehat X_n\) at \(0\), with respect to the domain basis
\(\beta_0^\#\) and the codomain basis \(\beta_0\), by
\[
    S_0^{\widehat X_n,(\beta_0,\beta_0^\#)}(\lambda)
    :=
    \bigl[M_0^{\widehat X_n}(\lambda)\bigr]_{\beta_0^\#}^{\beta_0}
    =
    \left[
        S_{0,\ell j}^{\widehat X_n,(\beta_0,\beta_0^\#)}(\lambda)
    \right]_{\ell,j\in J_n}.
\]
Equivalently,
\[
    M_0^{\widehat X_n}(\lambda)g_{0,j}^\#
    =
    \sum_{\ell\in J_n}
    S_{0,\ell j}^{\widehat X_n,(\beta_0,\beta_0^\#)}(\lambda)
    g_{0,\ell}.
\]

Choose \(\nu\in\mb C\) such that
\(\lambda=\nu(\nu+1)\), and set \(\alpha_j:=|j|/n\). By the
explicit computation for the conic model on \(\mb P^1\)
\cite{LiouModel2026},
\[
    S_{0,\ell j}^{\widehat X_n,(\beta_0,\beta_0^\#)}(\lambda)
    =
    \delta_{\ell j}
    s_j^{\widehat X_n,(\beta_0,\beta_0^\#)}(\lambda),
\]
where
\[
    s_j^{\widehat X_n,(\beta_0,\beta_0^\#)}(\lambda)
    =
    \begin{cases}
    \displaystyle
    -\frac{1}{n}
    \left(
        \gamma_E+\psi(\nu+1)
        +\frac{\pi}{2}\cot(\pi\nu)
    \right),
    & j=0,\\[1.2em]
    \displaystyle
    -\frac{\Gamma(1-\alpha_j)}
           {\Gamma(1+\alpha_j)}
     \frac{\Gamma(\nu+\alpha_j+1)}
          {\Gamma(\nu-\alpha_j+1)}
     \frac{\sin(\pi\nu)}
          {\sin\bigl(\pi(\nu-\alpha_j)\bigr)},
    & 0<|j|\leq n-1.
    \end{cases}
\]
Here \(\gamma_E\) is Euler's constant, \(\Gamma\) is the gamma
function, and \(\psi=\Gamma'/\Gamma\) is the digamma function.
The second expression is understood by meromorphic continuation at
its removable singularities.

The preceding formulas are written in the coordinate naturally
centered over the target value \(0\).  In the Hurwitz problem, however,
the corresponding branch value is \(z_k\).  Passing between these two
coordinates changes not only the normalization of the critical modes
but also the singular Lagrangian complement used to represent the Weyl
curve as a graph.  The resulting transformation law is therefore
affine rather than purely multiplicative.  Its off-diagonal term is
the source of the \(z_k\)- and \(\overline{z_k}\)-dependent entries that
eventually produce the ramification correction in the determinant
formula.

We next introduce another local coordinate at
\(0\in\widehat X_n\). Fix \(z_k\in\mb C\subset\mb P^1\), and
consider the M\"obius transformation
\(T_{z_k}(x):=(x-z_k)/(1+\overline{z_k}x)\), whose inverse is
\(T_{z_k}^{-1}(x)=(x+z_k)/(1-\overline{z_k}x)\).
Both \(T_{z_k}\) and its inverse are round-metric isometries.

Define \(\Phi_{n,k}:=T_{z_k}^{-1}\circ\Phi_n\). Then
\[
    \Phi_{n,k}(\xi)
    =
    \frac{\xi^n+z_k}
         {1-\overline{z_k}\xi^n},
    \qquad
    \Phi_{n,k}(0)=z_k,
\]
with \(\Phi_{n,k}^*ds_{\mathrm{rd}}^2=ds_n^2\).
Set \(A_k:=1+|z_k|^2\). A coordinate satisfying
\(\Phi_{n,k}=z_k+\zeta^n\) is
\[
    \zeta
    =
    A_k^{1/n}\xi
    \bigl(1-\overline{z_k}\xi^n\bigr)^{-1/n},
\]
where the local holomorphic branch is chosen so that
\(\zeta(0)=0\) and \(\zeta'(0)=A_k^{1/n}\). Its inverse is
\[
    \xi
    =
    \zeta
    \bigl(A_k+\overline{z_k}\zeta^n\bigr)^{-1/n}.
\]

The two expressions for the same metric are
\[
    \frac{4n^2|\xi|^{2n-2}}
         {(1+|\xi|^{2n})^2}|d\xi|^2
    =
    \frac{4n^2|\zeta|^{2n-2}}
         {\bigl(1+|z_k+\zeta^n|^2\bigr)^2}|d\zeta|^2.
\]
Define
\[
    h_{0,j}(\xi)
    :=
    g_j\bigl(\zeta(\xi)\bigr),
    \qquad
    h_{0,j}^\#(\xi)
    :=
    g_j^\#\bigl(\zeta(\xi)\bigr).
\]
Set \(\gamma_0:=(h_{0,j})_j\) and
\(\gamma_0^\#:=(h_{0,j}^\#)_j\). Then
\((\gamma_0,\gamma_0^\#)\) is the ordered basis induced by the
coordinate \(\zeta\), expressed in terms of the standard coordinate
\(\xi\), whereas \((\beta_0,\beta_0^\#)\) is defined directly by
\(\xi\).

\begin{lem}\label{lem:model-coordinate-basis-change}
The ordered bases \((\beta_0,\beta_0^\#)\) and
\((\gamma_0,\gamma_0^\#)\) of \(V_{\widehat X_n,0}\) are related
as follows. For the zero mode,
\[
    \begin{pmatrix}
        g_{0,0}\\
        g_{0,0}^\#
    \end{pmatrix}
    =
    \begin{pmatrix}
        1 & 0\\
        \dfrac{\log A_k}{n} & 1
    \end{pmatrix}
    \begin{pmatrix}
        h_{0,0}\\
        h_{0,0}^\#
    \end{pmatrix}.
\]
For \(1\leq j\leq n-1\),
\[
    \begin{pmatrix}
        g_{0,j}\\
        g_{0,-j}
    \end{pmatrix}
    =
    A_k^{-j/n}
    \begin{pmatrix}
        h_{0,j}\\
        h_{0,-j}
    \end{pmatrix},
\]
and
\[
    \begin{pmatrix}
        g_{0,j}^\#\\
        g_{0,-j}^\#
    \end{pmatrix}
    =
    A_k^{j/n}
    \begin{pmatrix}
        h_{0,j}^\#\\
        h_{0,-j}^\#
    \end{pmatrix}
    +
    \frac{\sqrt{j(n-j)}}{n}A_k^{j/n-1}
    \begin{pmatrix}
        z_k h_{0,j-n}\\
        \overline{z_k}h_{0,n-j}
    \end{pmatrix}.
\]
\end{lem}
\begin{proof}
We compute the critical asymptotic terms one by one. All identities
below are understood in \(V_{\widehat X_n,0}\), that is, modulo
noncritical asymptotic terms.

Recall that
\[
    \xi
    =
    \zeta
    \bigl(A_k+\overline{z_k}\zeta^n\bigr)^{-1/n}.
\]
For the zero mode, \(g_{0,0}=h_{0,0}\). Moreover,
\[
    \log|\xi|
    =
    \log|\zeta|
    -
    \frac{1}{n}\log A_k
    -
    \frac{1}{n}
    \log\left|
        1+\frac{\overline{z_k}}{A_k}\zeta^n
    \right|.
\]
The last term has no critical asymptotic contribution. It follows
that
\[
    g_{0,0}^\#
    =
    h_{0,0}^\#
    +
    \frac{\log A_k}{n}h_{0,0}.
\]

For \(1\leq j\leq n-1\), expansion of the coordinate transformation
at \(\zeta=0\) gives
\[
    g_{0,j}=A_k^{-j/n}h_{0,j},
    \qquad
    g_{0,-j}=A_k^{-j/n}h_{0,-j}.
\]
Similarly, retaining only the critical asymptotic terms yields
\[
    g_{0,j}^\#
    =
    A_k^{j/n}h_{0,j}^\#
    +
    \frac{\sqrt{j(n-j)}}{n}
    A_k^{j/n-1}z_kh_{0,j-n},
\]
and
\[
    g_{0,-j}^\#
    =
    A_k^{j/n}h_{0,-j}^\#
    +
    \frac{\sqrt{j(n-j)}}{n}
    A_k^{j/n-1}\overline{z_k}h_{0,n-j}.
\]
This proves the change-of-basis formulas.
\end{proof}

\begin{lem}
Regard the ordered bases as row vectors indexed by \(J_n\). Set
\(\mathsf D_k:=\operatorname{diag}(A_k^{-|j|/n})_{j\in J_n}\), and define
\[
    \mathsf C_k
    :=
    \frac{\log A_k}{n}E_{00}
    +
    \sum_{j=1}^{n-1}
    \frac{\sqrt{j(n-j)}}{nA_k}
    \left(
        z_kE_{j-n,j}
        +
        \overline{z_k}E_{n-j,-j}
    \right),
\]
where \(E_{\ell j}\) denotes the matrix unit indexed by
\(\ell,j\in J_n\). Then
\[
    \bigl[\id_{V_{\widehat X_n,0}}\bigr]
    _{(\beta_0,\beta_0^\#)}^{(\gamma_0,\gamma_0^\#)}
    =
    \begin{bmatrix}
        \mathsf D_k & \mathsf C_k\mathsf D_k^{-1}\\
        0   & \mathsf D_k^{-1}
    \end{bmatrix}.
\]
\end{lem}

\begin{proof}
This follows immediately by collecting the coefficients in the
preceding componentwise change-of-basis formulas.
\end{proof}

Write \(S_0^{\widehat X_n,(\gamma_0,\gamma_0^\#)}\) for the local
Weyl matrix obtained by using \(\operatorname{span}_{\mb C}\gamma_0^\#\)
as the singular complement at \(0\) and \(\gamma_0\) as the regular
basis, retaining the coordinate \(\eta\) at infinity.

\begin{lem}
The local \(S\)-matrices of \(\widehat X_n\) at \(0\), associated
with these two choices of regular and singular bases, satisfy
\[
    S_0^{\widehat X_n,(\gamma_0,\gamma_0^\#)}(\lambda)
    =
    \mathsf D_k
    S_0^{\widehat X_n,(\beta_0,\beta_0^\#)}(\lambda)
    \mathsf D_k
    +
    \mathsf C_k.
\]
\end{lem}

\begin{proof}
Apply \eqref{eq:weyl-chart-transition} using the preceding block
matrix, with the bases at infinity unchanged. Since its lower-left
block is zero, the local block at \(0\) transforms as
\[
\begin{aligned}
S_0^{\widehat X_n,(\gamma_0,\gamma_0^\#)}(\lambda)
&=
\left(
    \mathsf D_k
    S_0^{\widehat X_n,(\beta_0,\beta_0^\#)}(\lambda)
    +
    \mathsf C_k\mathsf D_k^{-1}
\right)
\left(\mathsf D_k^{-1}\right)^{-1}  \\
&=
\mathsf D_k
S_0^{\widehat X_n,(\beta_0,\beta_0^\#)}(\lambda)
\mathsf D_k
+
\mathsf C_k.
\end{aligned}
\]
\end{proof}

\begin{cor}
The only nonzero entries of
\(S_0^{\widehat X_n,(\gamma_0,\gamma_0^\#)}(\lambda)\) are the
following. Its zero-mode entry is
\[
    S_{0,00}^{\widehat X_n,(\gamma_0,\gamma_0^\#)}(\lambda)
    =
    -\frac{1}{n}
    \left(
        \gamma_E+\psi(\nu+1)
        +\frac{\pi}{2}\cot(\pi\nu)
        -\log A_k
    \right).
\]
For \(0<|j|\leq n-1\), its remaining diagonal entries are
\[
    S_{0,jj}^{\widehat X_n,(\gamma_0,\gamma_0^\#)}(\lambda)
    =
    A_k^{-2|j|/n}
    s_j^{\widehat X_n,(\beta_0,\beta_0^\#)}(\lambda).
\]
Moreover, for \(1\leq j\leq n-1\), the nonzero off-diagonal
entries are
\[
    S_{0,j-n,j}^{\widehat X_n,(\gamma_0,\gamma_0^\#)}(\lambda)
    =
    \frac{\sqrt{j(n-j)}}{nA_k}z_k,
    \qquad
    S_{0,n-j,-j}^{\widehat X_n,(\gamma_0,\gamma_0^\#)}(\lambda)
    =
    \frac{\sqrt{j(n-j)}}{nA_k}\overline{z_k}.
\]
All other entries vanish.
\end{cor}

\begin{proof}
By the preceding change-of-basis formula,
\[
    S_0^{\widehat X_n,(\gamma_0,\gamma_0^\#)}(\lambda)
    =
    \mathsf D_k
    S_0^{\widehat X_n,(\beta_0,\beta_0^\#)}(\lambda)
    \mathsf D_k
    +
    \mathsf C_k.
\]
Since
\(S_0^{\widehat X_n,(\beta_0,\beta_0^\#)}(\lambda)\) is diagonal,
the term
\(\mathsf D_kS_0^{\widehat X_n,(\beta_0,\beta_0^\#)}(\lambda)\mathsf D_k\)
has diagonal entries
\(A_k^{-2|j|/n}
s_j^{\widehat X_n,(\beta_0,\beta_0^\#)}(\lambda)\).
The remaining entries are precisely those of \(\mathsf C_k\), which gives
the stated formulas.
\end{proof}

\section{Local Model Comparison of Weyl Solutions}\label{sec:local-model-comparison}

Inspired by the scalar local-model comparison used by Kalvin and
Kokotov \cite{KalvinKokotov2019} in the case of simple ramification,
we develop in this section
a matrix-valued extension for higher-order ramification. More
precisely, we compare restrictions of global Weyl solutions on \(X\)
with pullbacks of the corresponding global solutions on the model
surface. This comparison identifies the difference of their critical
asymptotic data with the difference of the corresponding diagonal Weyl
blocks and, consequently, of the associated \(S\)-matrices.

Let \(P_k\) be a ramification point of
\(\varphi\colon X\to\mb P^1\), with ramification index \(n_k\).
If \(\varphi(P_k)\neq\infty\), we use the affine chart
\((U_0,x)\) on \(\mb P^1\) and set
\(z_k=x\bigl(\varphi(P_k)\bigr)\).
If \(\varphi(P_k)=\infty\), we instead use the affine chart
\((U_\infty,y)\), where \(y(\infty)=0\).
After shrinking a neighborhood of \(P_k\), if necessary, there exists
a local holomorphic coordinate \(z\) on a neighborhood \(U_k\) of
\(P_k\), centered at \(P_k\), such that
\[
    x\circ\varphi=z_k+z^{n_k}
\]
when \(\varphi(P_k)\neq\infty\), whereas
\(y\circ\varphi=z^{n_k}\) when \(\varphi(P_k)=\infty\).
It suffices to consider the case \(\varphi(P_k)\neq\infty\), since the
case \(\varphi(P_k)=\infty\) is treated analogously using the affine
coordinate \(y\). For simplicity, write \(n=n_k\). Then, on \(U_k\),
the pullback of the round metric is
\[
    g
    =
    \varphi^*ds_{\mathrm{rd}}^2
    =
    \frac{4n^2|z|^{2n-2}}
         {\bigl(1+|z_k+z^n|^2\bigr)^2}
    \,|dz|^2.
\]
We denote by \(V_{P_k}\) the critical asymptotic space at \(P_k\).
Let \(\gamma_k=(f_{k,j})_{j\in J_n}\) and
\(\gamma_k^\#=(f_{k,j}^\#)_{j\in J_n}\), where
\(f_{k,j}=f_{P_k,j}(z)\) and
\(f_{k,j}^\#=f_{P_k,j}^\#(z)\).
Then \((\gamma_k,\gamma_k^\#)\) is the ordered basis of \(V_{P_k}\)
obtained by concatenating the two ordered families
\(\gamma_k\) and \(\gamma_k^\#\).

Consider the Möbius transformation
\(T\colon\mb P^1\to\mb P^1\) given in the affine coordinate \(x\) by
\(
    T(x)=\frac{x-z_k}{1+\overline{z_k}x}.
\)
Since \(T\) is a round-metric isometry taking \(z_k\) to \(0\),
the map \(\widehat\varphi:=T\circ\varphi\) induces the same metric
and has ramification index \(n\) at \(P_k\).
Choose a centered coordinate \(w\) on a smaller neighborhood
\(\widetilde U_k\) such that \(x\circ\widehat\varphi=w^n\). Then
\[
    g
    =
    \widehat\varphi^*ds_{\mathrm{rd}}^2
    =
    \frac{4n^2|w|^{2n-2}}
         {\bigl(1+|w|^{2n}\bigr)^2}
    \,|dw|^2.
\]
The critical asymptotic space \(V_{P_k}\) has another ordered basis
\(\bigl(\beta_k,\beta_k^\#\bigr)\), where
\(\beta_k=(\widehat f_{k,j})_{j\in J_n}\), with
\(\widehat f_{k,j}=f_{P_k,j}(w)\), and
\(\beta_k^\#=(\widehat f_{k,j}^\#)_{j\in J_n}\), with
\(\widehat f_{k,j}^\#=f_{P_k,j}^\#(w)\).

Set \(A_k=1+|z_k|^2\), and choose the branch satisfying
\(
    \lim_{z\to 0}\frac{w}{z}=A_k^{-1/n}.
\)
If \(z_k\neq 0\), then, for
\(
    |z|<\left(\frac{A_k}{|z_k|}\right)^{1/n},
\)
we have
\[
\begin{aligned}
    w
    &=
    z\bigl(A_k+\overline{z_k}z^n\bigr)^{-1/n} \\
    &=
    A_k^{-1/n}z
    \left(1+\frac{\overline{z_k}}{A_k}z^n\right)^{-1/n} \\
    &=
    A_k^{-1/n}z
    \sum_{\ell\geq 0}
    \binom{-1/n}{\ell}
    \left(\frac{\overline{z_k}}{A_k}\right)^\ell
    z^{n\ell}.
\end{aligned}
\]
When \(z_k=0\), we have \(A_k=1\) and \(w=z\).
\begin{lem}
The ordered bases
\((\beta_k,\beta_k^\#)\) and
\((\gamma_k,\gamma_k^\#)\) of \(V_{P_k}\) are related as follows.
For the zero mode,
\[
    \begin{pmatrix}
        \widehat f_{k,0}\\
        \widehat f_{k,0}^\#
    \end{pmatrix}
    =
    \begin{pmatrix}
        1 & 0\\
        \dfrac{\log A_k}{n} & 1
    \end{pmatrix}
    \begin{pmatrix}
        f_{k,0}\\
        f_{k,0}^\#
    \end{pmatrix}.
\]
For \(1\leq j\leq n-1\),
\[
    \begin{pmatrix}
        \widehat f_{k,j}\\
        \widehat f_{k,-j}
    \end{pmatrix}
    =
    A_k^{-j/n}
    \begin{pmatrix}
        f_{k,j}\\
        f_{k,-j}
    \end{pmatrix},
\]
and
\[
    \begin{pmatrix}
        \widehat f_{k,j}^\#\\
        \widehat f_{k,-j}^\#
    \end{pmatrix}
    =
    A_k^{j/n}
    \begin{pmatrix}
        f_{k,j}^\#\\
        f_{k,-j}^\#
    \end{pmatrix}
    +
    \frac{\sqrt{j(n-j)}}{n}A_k^{j/n-1}
    \begin{pmatrix}
        z_k f_{k,j-n}\\
        \overline{z_k}f_{k,n-j}
    \end{pmatrix}.
\]
\end{lem}
\begin{proof}
The coordinate change \(w=z(A_k+\overline{z_k}z^n)^{-1/n}\)
has the same form as the model change from \(\zeta\) to \(\xi\) in
Lemma~\ref{lem:model-coordinate-basis-change}. Substituting
\((w,z,\widehat f_{k,j},f_{k,j})\) for
\((\xi,\zeta,g_{0,j},h_{0,j})\) in that lemma gives the stated
identities in \(V_{P_k}\).
\end{proof}
For the comparison below, choose the singular complement at
\(P_k\) to be \(\mc F_{X,P_k}^{\#}:=\operatorname{span}_{\mb C}\beta_k^{\#}\).
The regular space is
\(\mc F_{X,P_k}=\operatorname{span}_{\mb C}\beta_k
=\operatorname{span}_{\mb C}\gamma_k\).
Fix conic coordinates at the other conic points, and let
\(\mc F_X^{\#}\) be the direct sum of the resulting singular spaces.
We abbreviate
\(M^X:=M_{\mc F_X,\mc F_X^{\#}}^X\).
The notation \(S_k^{X,(\gamma_k,\gamma_k^{\#})}\), used below,
instead refers to the singular complement spanned by
\(\gamma_k^{\#}\) at \(P_k\), with the other choices fixed.
Thus the transformation between these two matrices includes a change
of singular complement, in addition to a change of bases.
Set \(\widetilde U_k':=\widetilde U_k\setminus\{P_k\}\).

For \(\lambda\in\rho(\Delta_{X,\mc F_X})\) and
\(h^\#\in\mc F_{X,P_k}^\#\), set
\(u_\lambda:=\Gamma_{\mc F_X,\mc F_X^\#}^X(\lambda)h^\#\).
Its restriction
\(u_\lambda^{\widetilde U_k'}:=u_\lambda|_{\widetilde U_k'}\)
solves \((\Delta_{\widetilde U_k'}-\lambda)u=0\).
Define
\[
    \mc K_{\widetilde U_k}^{X}(\lambda)
    :=
    \left\{
        u_\lambda^{\widetilde U_k'}
        :
        u_\lambda
        =
        \Gamma_{\mc F_X,\mc F_X^{\#}}^{X}(\lambda)h^{\#},
        \quad
        h^{\#}\in\mc F_{X,P_k}^{\#}
    \right\}.
\]

We use the same notation \(\pi_{P_k}\) for the asymptotic map induced
on germs at \(P_k\). Since the critical asymptotic data at \(P_k\)
depend only on the restriction to \(\widetilde U_k'\), for
\(u_\lambda^{\widetilde U_k'}\in
\mc K_{\widetilde U_k}^{X}(\lambda)\) corresponding to
\(h^{\#}\in\mc F_{X,P_k}^{\#}\), we have
\[
    \pi_{P_k}u_\lambda^{\widetilde U_k'}
    =
    \pi_{P_k}u_\lambda
    =
    M_k^{X}(\lambda)h^{\#}+h^{\#},
\]
where
\(M_k^{X}(\lambda):=M_{P_k,P_k}^{X}(\lambda)\) denotes the
\((P_k,P_k)\)-diagonal block of the global Weyl map \(M^X(\lambda)\).
As \(h^{\#}\) ranges over \(\mc F_{X,P_k}^{\#}\), it follows that
\[
    \pi_{P_k}\bigl(\mc K_{\widetilde U_k}^{X}(\lambda)\bigr)
    =
    \operatorname{graph}_{(
        \mc F_{X,P_k},\mc F_{X,P_k}^{\#})}
    \bigl(M_k^{X}(\lambda)\bigr).
\]
Define the \((P_k,P_k)\)-diagonal block of the global \(S\)-matrix,
with respect to the ordered bases \(\beta_k^\#\) and \(\beta_k\), by
\[
    S_k^{X,(\beta_k,\beta_k^\#)}(\lambda)
    :=
    \bigl[M_k^X(\lambda)\bigr]_{\beta_k^\#}^{\beta_k}.
\]
For each basis element
\(\widehat f_{k,j}^\#\in\mc F_{X,P_k}^\#\), regarded naturally as
an element of \(\mc F_X^\#\), let
\[
    u_{k,j}(\lambda)
    :=
    \Gamma_{\mc F_X,\mc F_X^\#}^{X}(\lambda)
    \widehat f_{k,j}^\#,
\]
and set
\(u_{k,j}^{\widetilde U_k'}(\lambda)
:=u_{k,j}(\lambda)|_{\widetilde U_k'}\).
Writing
\[
    S_k^{X,(\beta_k,\beta_k^\#)}(\lambda)
    =
    \left(
        S_{k,\ell j}^{X,(\beta_k,\beta_k^\#)}(\lambda)
    \right)_{\ell,j},
\]
we have
\[
    \pi_{P_k}u_{k,j}^{\widetilde U_k'}(\lambda)
    =
    \sum_{\ell}
    S_{k,\ell j}^{X,(\beta_k,\beta_k^\#)}(\lambda)
    \widehat f_{k,\ell}
    +
    \widehat f_{k,j}^\#.
\]

Write \(D_a(r):=\{\zeta\in\mb C:|\zeta-a|<r\}\) for a planar disc.
Let \(\epsilon>0\) be sufficiently small, and set
\(\widehat U_0:=\{Q\in\widehat X_n:|\xi(Q)|<\epsilon\}\).
Then \(\xi\colon\widehat U_0\to D_0(\epsilon)\) is a coordinate
chart centered at \(0\).
After shrinking the coordinate neighborhood of \(P_k\), if
necessary, set
\(\widetilde U_k:=\{P\in X:|w(P)|<\epsilon\}\).
Then \(w\colon\widetilde U_k\to D_0(\epsilon)\) is a coordinate
chart centered at \(P_k\).
Define
\(J\colon\widetilde U_k\to\widehat U_0\) by
\(J:=(\xi|_{\widehat U_0})^{-1}\circ w\).
Then \(J\) is biholomorphic and satisfies \(\xi\circ J=w\).
Moreover,
\[
    \begin{aligned}
    J^*\bigl(ds_n^2|_{\widehat U_0}\bigr)
    &=
    J^*\left(
        \frac{4n^2|\xi|^{2n-2}}
             {(1+|\xi|^{2n})^2}|d\xi|^2
    \right) \\
    &=
    \frac{4n^2|w|^{2n-2}}
         {(1+|w|^{2n})^2}|dw|^2
    =
    g|_{\widetilde U_k}.
    \end{aligned}
\]
Thus, \(J\) is a biholomorphic isometry, and its pullback induces
a linear symplectic isomorphism
\(J_0^*\colon V_{\widehat X_n,0}\to V_{P_k}\). It restricts to
linear isomorphisms
\[
    J_0^*\colon
    \mc F_{\widehat X_n,0}\longrightarrow\mc F_{X,P_k},
    \qquad
    J_0^*\colon
    \mc F_{\widehat X_n,0}^{\#}
    \longrightarrow
    \mc F_{X,P_k}^{\#}.
\]
Since \(w=\xi\circ J=J^*\xi\), we have
\(J_0^*g_{0,j}=\widehat f_{k,j}\) and
\(J_0^*g_{0,j}^{\#}=\widehat f_{k,j}^{\#}\). Hence,
\(J_0^*\beta_0=\beta_k\) and \(J_0^*\beta_0^{\#}=\beta_k^{\#}\).
The same invertible transformation relates the model bases
\((\beta_0,\beta_0^{\#})\), \((\gamma_0,\gamma_0^{\#})\) and
the corresponding bases at \(P_k\). Consequently,
\(J_0^*\gamma_0=\gamma_k\) and
\(J_0^*\gamma_0^{\#}=\gamma_k^{\#}\).
Set \(\widehat U_0':=\widehat U_0\setminus\{0\}\).

For \(\lambda\in\rho(\Delta_{\widehat X_n,\mc F_{\widehat X_n}})\)
and \(h_0^\#\in\mc F_{\widehat X_n,0}^\#\), set
\[
    v_\lambda
    =
    \Gamma_{\mc F_{\widehat X_n},
    \mc F_{\widehat X_n}^{\#}}^{\widehat X_n}(\lambda)
    h_0^{\#}.
\]

Let
\(v_\lambda^{\widehat U_0'}
:=v_\lambda|_{\widehat U_0'}\), and set
\(\widetilde v_\lambda:=J^*v_\lambda^{\widehat U_0'}\). Define
\[
    \mc K_{\widetilde U_k}^{\mathrm{mod}}(\lambda)
    :=
    \left\{
        J^*\left(
            \left.
            \Gamma_{\mc F_{\widehat X_n},
            \mc F_{\widehat X_n}^{\#}}^{\widehat X_n}(\lambda)
            h_0^{\#}
            \right|_{\widehat U_0'}
        \right)
        :
        h_0^{\#}\in\mc F_{\widehat X_n,0}^{\#}
    \right\}.
\]

The local isometry \(J\) intertwines the Laplacians, so every
element of \(\mc K_{\widetilde U_k}^{\mathrm{mod}}(\lambda)\)
solves \((\Delta_{\widetilde U_k'}-\lambda)u=0\).

The critical asymptotic data of
\(\widetilde v_\lambda\) at \(P_k\) satisfy
\[
    \pi_{P_k}\widetilde v_\lambda
    =
    J_0^*\pi_0v_\lambda
    =
    J_0^*\left(
        M_0^{\widehat X_n}(\lambda)h_0^{\#}+h_0^{\#}
    \right).
\]
Writing \(h^\#=J_0^*h_0^\#\), which ranges over
\(\mc F_{X,P_k}^\#\), identifies the projected model solution
space with the graph of
\[
    M_{P_k}^{\widetilde U_k}(\lambda)
    :=
    J_0^*M_0^{\widehat X_n}(\lambda)(J_0^*)^{-1}
    \colon
    \mc F_{X,P_k}^{\#}
    \longrightarrow
    \mc F_{X,P_k}.
\]
Equivalently,
\[
    \pi_{P_k}\bigl(
        \mc K_{\widetilde U_k}^{\mathrm{mod}}(\lambda)
    \bigr)
    =
    \operatorname{graph}_{(
        \mc F_{X,P_k},\mc F_{X,P_k}^{\#})}
    \bigl(M_{P_k}^{\widetilde U_k}(\lambda)\bigr).
\]
We call
\(M_{P_k}^{\widetilde U_k}(\lambda)\)
the model-induced local Weyl map at \(P_k\).

Define the corresponding local \(S\)-matrix, with respect to the
ordered basis \(\beta_k^\#\) of \(\mc F_{X,P_k}^\#\) and the
ordered basis \(\beta_k\) of \(\mc F_{X,P_k}\), by
\[
    S_{P_k}^{\widetilde U_k,(\beta_k,\beta_k^\#)}(\lambda)
    :=
    \bigl[M_{P_k}^{\widetilde U_k}(\lambda)\bigr]
        _{\beta_k^\#}^{\beta_k}.
\]
Since \(J_0^*\beta_0=\beta_k\) and
\(J_0^*\beta_0^\#=\beta_k^\#\), the operators
\(M_{P_k}^{\widetilde U_k}(\lambda)\) and
\(M_0^{\widehat X_n}(\lambda)\) have the same matrix representation
with respect to the corresponding ordered bases. Therefore,
\[
    S_{P_k}^{\widetilde U_k,(\beta_k,\beta_k^\#)}(\lambda)
    =
    \bigl[M_0^{\widehat X_n}(\lambda)\bigr]
        _{\beta_0^\#}^{\beta_0}
    =
    S_0^{\widehat X_n,(\beta_0,\beta_0^\#)}(\lambda).
\]

This identity isolates the local part of the comparison.  The global
and model Weyl solutions have the same singular critical data at
\(P_k\), so their difference has only regular critical data there.
After multiplication by a cutoff, Green's identity expresses this
regular component through a commutator supported in a fixed annulus
away from the cone point.  The Davies--Gaffney estimate then makes that
annular contribution rapidly decreasing as \(\lambda\to-\infty\).
Thus the explicit model determines the global diagonal Weyl block to
all algebraic orders; the argument below implements precisely these
three steps.

Let
\(\lambda\in
\rho(\Delta_{X,\mc F_X})
\cap
\rho(\Delta_{\widehat X_n,\mc F_{\widehat X_n}})\),
and fix \(h^\#\in\mc F_{X,P_k}^\#\). The corresponding global
Weyl solution on \(X'\) is
\(u_\lambda
:=\Gamma_{\mc F_X,\mc F_X^\#}^X(\lambda)h^\#\), which
satisfies \((\Delta_X-\lambda)u_\lambda=0\).
Set \(h_0^\#:=(J_0^*)^{-1}h^\#\). The model Weyl solution on
\(\widehat X_n\setminus\{0,\infty\}\) with singular boundary
component \(h_0^\#\) at \(0\) is
\(v_\lambda
:=\Gamma_{\mc F_{\widehat X_n},
\mc F_{\widehat X_n}^\#}^{\widehat X_n}(\lambda)h_0^\#\), which
satisfies \((\Delta_{\widehat X_n}-\lambda)v_\lambda=0\).
Restricting the model solution to \(\widehat U_0'\) and transporting
it by the local isometry \(J\), set
\(
\widetilde v_\lambda:=J^*(v_\lambda|_{\widehat U_0'}).
\)
This is a solution on \(\widetilde U_k'\). Since
\(J_0^*h_0^\#=h^\#\), it has the same singular boundary component
at \(P_k\) as \(u_\lambda|_{\widetilde U_k'}\).
Consequently, the two functions solve
\((\Delta_{\widetilde U_k'}-\lambda)u=0\), and we may define
\(w_\lambda
:=u_\lambda^{\widetilde U_k'}-\widetilde v_\lambda\).
Then
\((\Delta_{\widetilde U_k'}-\lambda)w_\lambda=0\), and the
\(\mc F_{X,P_k}^{\#}\)-component of \(\pi_{P_k}w_\lambda\) vanishes.
Indeed,
\[
\begin{aligned}
    \pi_{P_k}w_\lambda
    &=
    \pi_{P_k}u_\lambda^{\widetilde U_k'}
    -
    \pi_{P_k}\widetilde v_\lambda \\
    &=
    \left(M_k^X(\lambda)h^{\#}+h^{\#}\right)
    -
    \left(
        M_{P_k}^{\widetilde U_k}(\lambda)h^{\#}+h^{\#}
    \right) \\
    &=
    \left(
        M_k^X(\lambda)
        -
        M_{P_k}^{\widetilde U_k}(\lambda)
    \right)h^{\#}
    \in\mc F_{X,P_k}.
\end{aligned}
\]
Thus, the difference between the diagonal Weyl block on \(X\) and
the model-induced local Weyl map, evaluated at \(h^{\#}\), is
precisely the regular component of the critical asymptotic data of
\(w_\lambda\).

For \(h^\#=\widehat f_{k,j}^\#=J_0^*g_{0,j}^\#\), the solution
on \(X\) is \(u_{k,j}(\lambda)\). Set
\[
    v_{0,j}(\lambda)
    :=
    \Gamma_{\mc F_{\widehat X_n},
    \mc F_{\widehat X_n}^{\#}}^{\widehat X_n}(\lambda)
    g_{0,j}^{\#},
    \qquad
    \widetilde v_{k,j}(\lambda)
    :=
    J^*\bigl(v_{0,j}(\lambda)|_{\widehat U_0'}\bigr).
\]
Define
\(w_{k,j}(\lambda)
:=
u_{k,j}^{\widetilde U_k'}(\lambda)
-
\widetilde v_{k,j}(\lambda)\).
Then
\[
    \pi_{P_k}w_{k,j}(\lambda)
    =
    \sum_{\ell}
    \left(
        S_{k,\ell j}^{X,(\beta_k,\beta_k^\#)}(\lambda)
        -
        S_{0,\ell j}^{\widehat X_n,(\beta_0,\beta_0^\#)}(\lambda)
    \right)
    \widehat f_{k,\ell}.
\]
Consequently, comparing the diagonal \(S\)-matrix block on \(X\)
with the model \(S\)-matrix reduces to estimating the critical
asymptotic data of \(w_{k,j}(\lambda)\).

Let
\(\chi\in C_c^\infty(\widetilde U_k)\) be a cutoff function such that
\(\chi=1\) near \(P_k\). Then
\(\widetilde w_{k,j}(\lambda):=\chi w_{k,j}(\lambda)\), initially
defined on \(\widetilde U_k'\), vanishes near
\(\partial\widetilde U_k\) because
\(\operatorname{supp}\chi\Subset\widetilde U_k\). It therefore
extends by zero across \(X\setminus\widetilde U_k\) to an element of
\(\mc D_{\max}(\Delta_X)\). Indeed, on \(\widetilde U_k'\),
\((\Delta_X-\lambda)(\chi w_{k,j}(\lambda))
=[\Delta_X,\chi]w_{k,j}(\lambda)\).
The right-hand side is smooth and supported in a fixed annulus
contained in the smooth locus, and hence belongs to \(L^2(X')\).
The extension vanishes near \(\partial\widetilde U_k\), so the
displayed distributional identity continues to hold after extension
by zero. Since \(\chi w_{k,j}(\lambda)\in L^2(X')\), this proves the
claimed maximal-domain membership. Since \(\chi=1\) near \(P_k\), the
extended function has the same critical asymptotic data at \(P_k\) as
\(w_{k,j}(\lambda)\), and it has no critical asymptotic data at the
other conic points. Thus
\(\widetilde\pi_{\mc F_X^{\#}}
\widetilde w_{k,j}(\lambda)=0\), and consequently
\(\widetilde w_{k,j}(\lambda)\in\mc D_{X,\mc F_X}\).

The boundary pairing therefore reduces to the contribution at \(P_k\):
\begin{align*}
    \mk q_X\bigl(
        \widetilde w_{k,j}(\lambda),
        u_{k,s}(\lambda)
    \bigr)
    &=
    \omega_{P_k}\bigl(
        \pi_{P_k}\widetilde w_{k,j}(\lambda),
        \pi_{P_k}u_{k,s}(\lambda)
    \bigr) \\
    &=
    \sum_{\ell}
    \left(
        S_{k,\ell j}^{X,(\beta_k,\beta_k^\#)}(\lambda)
        -
        S_{0,\ell j}^{\widehat X_n,(\beta_0,\beta_0^\#)}(\lambda)
    \right)
    \omega_{P_k}\bigl(
        \widehat f_{k,\ell},
        \widehat f_{k,s}^\#
    \bigr) \\
    &=
    S_{k,sj}^{X,(\beta_k,\beta_k^\#)}(\lambda)
    -
    S_{0,sj}^{\widehat X_n,(\beta_0,\beta_0^\#)}(\lambda).
\end{align*}

Assume that \(\lambda<0\). On the other hand, by the definition of
the Green form and since
\((\Delta_X-\lambda)u_{k,s}(\lambda)=0\), we have
\begin{align*}
    \mk q_X\bigl(
        \widetilde w_{k,j}(\lambda),
        u_{k,s}(\lambda)
    \bigr)
    &=
    \left\langle
        \Delta_X\widetilde w_{k,j}(\lambda),
        u_{k,s}(\lambda)
    \right\rangle_{L^2(X')}
    -
    \left\langle
        \widetilde w_{k,j}(\lambda),
        \Delta_Xu_{k,s}(\lambda)
    \right\rangle_{L^2(X')}
    \\
    &=
    \left\langle
        (\Delta_X-\lambda)\widetilde w_{k,j}(\lambda),
        u_{k,s}(\lambda)
    \right\rangle_{L^2(X')}
    \\
    &=
    \left\langle
        [\Delta_X,\chi]w_{k,j}(\lambda),
        u_{k,s}(\lambda)
    \right\rangle_{L^2(X')}.
\end{align*}

It follows that
\[
    S_{k,sj}^{X,(\beta_k,\beta_k^\#)}(\lambda)
    -
    S_{0,sj}^{\widehat X_n,(\beta_0,\beta_0^\#)}(\lambda)
    =
    \left\langle
        [\Delta_X,\chi]w_{k,j}(\lambda),
        u_{k,s}(\lambda)
    \right\rangle_{L^2(X')}.
\]

We now estimate the comparison pairing on a fixed annulus away
from the conic point, using the conic coordinate \(w=\xi\circ J\)
introduced above. Choose \(0<r_1<r_2<\epsilon\), and take
the cutoff above so that \(\chi=1\) on
\(\{P\in\widetilde U_k:|w(P)|\leq r_1\}\) and \(\chi=0\) on
\(\{P\in\widetilde U_k:|w(P)|\geq r_2\}\).
Set
\(A:=\{P\in\widetilde U_k:r_1\leq|w(P)|\leq r_2\}\),
and choose an open annulus \(A_1\) with smooth boundary such that
\(A\Subset A_1\Subset\widetilde U_k'\).
The cutoff and both annuli are fixed independently of \(\lambda\).
All norms below use the metric \(g\), which is smooth and
nondegenerate on a neighborhood of \(\overline{A_1}\).
For the closed annulus \(A\), Sobolev norms are taken over its
interior.

\begin{lem}\label{lem:annular-commutator-estimate}
There is a constant \(C>0\), depending only on \(g\), \(\chi\),
\(A\), and \(A_1\), such that, for every \(u\in H^2(A_1)\),
\[
\|[\Delta_X,\chi]u\|_{L^2(A)}
\leq C\bigl(
\|\Delta_Xu\|_{L^2(A_1)}+\|u\|_{L^2(A_1)}
\bigr).
\]
The commutator is supported in \(A\).

Suppose, in addition, that \(u\) satisfies
\((\Delta_X-\lambda)u=0\) on \(A_1\). Then the preceding estimate
implies
\(\|[\Delta_X,\chi]u\|_{L^2(A)}
\leq C(1+|\lambda|)\|u\|_{L^2(A_1)}\),
where \(C\) is independent of \(\lambda\).
\end{lem}

\begin{proof}
The product rule gives
\([\Delta_X,\chi]u=(\Delta_X\chi)u
-2\langle\nabla\chi,\nabla u\rangle_g\),
where the pointwise metric pairing is extended complex bilinearly.
Both \(\Delta_X\chi\) and \(\nabla\chi\) are supported in \(A\)
and are bounded there. Thus
\(\|[\Delta_X,\chi]u\|_{L^2(A)}
\leq C_\chi\|u\|_{H^1(A)}\),
where \(C_\chi\) depends only on the fixed cutoff and metric.

To control the \(H^1\)-norm, apply the interior elliptic estimate
on the fixed annuli:
\[
\|u\|_{H^1(A)}
\leq C\bigl(
\|\Delta_Xu\|_{H^{-1}(A_1)}+\|u\|_{L^2(A_1)}
\bigr).
\]
Since
\(L^2(A_1)\hookrightarrow H^{-1}(A_1)=(H_0^1(A_1))^*\)
continuously, the right-hand side is bounded by
\(C(\|\Delta_Xu\|_{L^2(A_1)}+\|u\|_{L^2(A_1)})\).
Combining the two bounds and enlarging \(C\) proves the first
estimate.

Now assume, in addition, that
\((\Delta_X-\lambda)u=0\) on \(A_1\). Then
\(\Delta_Xu=\lambda u\) in \(L^2(A_1)\), and hence
\(\|\Delta_Xu\|_{L^2(A_1)}
=|\lambda|\|u\|_{L^2(A_1)}\).
Applying the first estimate gives
\[
\|[\Delta_X,\chi]u\|_{L^2(A)}
\leq C(1+|\lambda|)\|u\|_{L^2(A_1)},
\]
which proves the second assertion.
\end{proof}

We now combine this estimate with the local rapid decay of the
Weyl solutions.

\begin{thm}\label{thm:S-local-comparison}
For all \(s,j\in J_n\),
\[
S_{k,sj}^{X,(\beta_k,\beta_k^\#)}(\lambda)
-
S_{0,sj}^{\widehat X_n,(\beta_0,\beta_0^\#)}(\lambda)
=
O(|\lambda|^{-\infty})
\]
as \(\lambda\to-\infty\).
\end{thm}

\begin{proof}
Lemma~\ref{lem:weyl-solution-decay} applies on
\(\overline{A_1}\) and on its image under the isometry \(J\).
Since
\(w_{k,j}(\lambda)
=u_{k,j}(\lambda)|_{\widetilde U_k'}
-\widetilde v_{k,j}(\lambda)\),
the triangle inequality and invariance of the \(L^2\)-norm under
\(J\) give
\[
\begin{aligned}
\|w_{k,j}(\lambda)\|_{L^2(A_1)}
&\leq
\|u_{k,j}(\lambda)\|_{L^2(A_1)}
+
\|v_{0,j}(\lambda)\|_{L^2(J(A_1))}\\
&=O(|\lambda|^{-\infty}).
\end{aligned}
\]
The same lemma gives
\(\|u_{k,s}(\lambda)\|_{L^2(A)}
=O(|\lambda|^{-\infty})\).

By the preceding comparison identity, the difference of the
two \(S\)-matrix entries is
\(\langle[\Delta_X,\chi]w_{k,j}(\lambda),
u_{k,s}(\lambda)\rangle_{L^2(X')}\).
The commutator is supported in \(A\), so this pairing can be taken
over \(A\). Since \(w_{k,j}(\lambda)\) is smooth on
\(\widetilde U_k'\) and
\(\overline{A_1}\Subset\widetilde U_k'\), its restriction to
\(A_1\) belongs to \(H^2(A_1)\). Moreover,
\((\Delta_X-\lambda)w_{k,j}(\lambda)=0\) there.
The Cauchy--Schwarz inequality and
Lemma~\ref{lem:annular-commutator-estimate}, applied with
\(u=w_{k,j}(\lambda)|_{A_1}\), therefore yield
\[
\begin{aligned}
\left|
\left\langle
[\Delta_X,\chi]w_{k,j}(\lambda),u_{k,s}(\lambda)
\right\rangle_{L^2(X')}
\right|
&\leq
\|[\Delta_X,\chi]w_{k,j}(\lambda)\|_{L^2(A)}
\|u_{k,s}(\lambda)\|_{L^2(A)}\\
&\leq
C(1+|\lambda|)
\|w_{k,j}(\lambda)\|_{L^2(A_1)}
\|u_{k,s}(\lambda)\|_{L^2(A)}.
\end{aligned}
\]
Both \(L^2\)-norms decay faster than every negative power of
\(|\lambda|\). Multiplication by \(1+|\lambda|\) preserves this
property, so the right-hand side is
\(O(|\lambda|^{-\infty})\). This proves the assertion.
\end{proof}

\begin{cor}
For all \(s,j\in J_n\),
\[
    S_{k,sj}^{X,(\gamma_k,\gamma_k^\#)}(\lambda)
    -
    S_{0,sj}^{\widehat X_n,(\gamma_0,\gamma_0^\#)}(\lambda)
    =
    O(|\lambda|^{-\infty})
\]
as \(\lambda\to-\infty\).
\end{cor}

\begin{proof}
The same change-of-basis matrices \(\mathsf D_k\) and \(\mathsf C_k\) apply on
\(V_{P_k}\) and \(V_{\widehat X_n,0}\). Hence, by the corresponding
affine transformations of the Weyl matrices,
\[
    S_k^{X,(\gamma_k,\gamma_k^\#)}(\lambda)
    =
    \mathsf D_k
    S_k^{X,(\beta_k,\beta_k^\#)}(\lambda)
    \mathsf D_k
    +
    \mathsf C_k
\]
and
\[
    S_0^{\widehat X_n,(\gamma_0,\gamma_0^\#)}(\lambda)
    =
    \mathsf D_k
    S_0^{\widehat X_n,(\beta_0,\beta_0^\#)}(\lambda)
    \mathsf D_k
    +
    \mathsf C_k.
\]
Therefore,
\[
\begin{aligned}
&S_k^{X,(\gamma_k,\gamma_k^\#)}(\lambda)
-
S_0^{\widehat X_n,(\gamma_0,\gamma_0^\#)}(\lambda)\\
&\qquad=
\mathsf D_k
\left(
    S_k^{X,(\beta_k,\beta_k^\#)}(\lambda)
    -
    S_0^{\widehat X_n,(\beta_0,\beta_0^\#)}(\lambda)
\right)
\mathsf D_k.
\end{aligned}
\]
Theorem~\ref{thm:S-local-comparison} and the fact that \(\mathsf D_k\) is independent of
\(\lambda\) imply the assertion.
\end{proof}

\begin{cor}\label{cor:S-high-energy-derivatives}
For every \(s,j\in J_n\), every integer \(q\geq0\), and every
\(N>0\),
\[
\frac{d^q}{dr^q}
\left(
S_{k,sj}^{X,(\gamma_k,\gamma_k^\#)}(-r)
-
S_{0,sj}^{\widehat X_n,(\gamma_0,\gamma_0^\#)}(-r)
\right)
=
O(r^{-N})
\]
as \(r\to+\infty\). In particular, for the off-diagonal entries used
in the determinant variation,
\[
\frac{d^q}{dr^q}
\left(
S_{k,n-p,-p}^{X,(\gamma_k,\gamma_k^\#)}(-r)
-
S_{k,p}^{(\infty)}
\right)
=
O(r^{-N}),
\qquad q\geq0,
\]
where
\(S_{k,p}^{(\infty)}
=\sqrt{p(n-p)}\,\overline{z_k}/(nA_k)\).
Only the cases \(q=0,1,2\) will be used below.
\end{cor}

\begin{proof}
Write \(u_r=u_{k,s}(-r)\) and \(w_r=w_{k,j}(-r)\).
In the \((\beta_k,\beta_k^\#)\)-basis, set
\[
F_{sj}(r)
:=
S_{k,sj}^{X,(\beta_k,\beta_k^\#)}(-r)
-S_{0,sj}^{\widehat X_n,(\beta_0,\beta_0^\#)}(-r).
\]
Since \([\Delta_X,\chi]w_r\) is supported in \(A\), the comparison
identity gives
\[
F_{sj}(r)
=\left\langle[\Delta_X,\chi]w_r,u_r\right\rangle_{L^2(A)}.
\]

We first justify differentiation of this pairing. For a fixed
cutoff representative \(H\) of the prescribed singular data, the
Weyl solution is \(H+R_{\mc F}(-r)(\Delta_X+r)H\).
The resolvent is holomorphic with values in the Friedrichs graph
domain, so this expression is smooth in \(r>0\) with values in
the maximal graph domain. Interior elliptic estimates make
restriction to \(A_1\) continuous from that graph domain to
\(H^2(A_1)\); the same holds on the model annulus. Since the
local isometry \(J\) is fixed, \(w_r\) is therefore smooth with
values in \(H^2(A_1)\). The commutator is continuous from
\(H^2(A_1)\) to \(L^2(A)\), so the pairing may be differentiated
to every order by the Leibniz rule.

By Lemma~\ref{lem:weyl-solution-decay}, the Weyl solutions on both
surfaces and all their \(r\)-derivatives are rapidly
decreasing in \(L^2\) on the fixed comparison annuli. The local
isometry \(J\) is independent of \(r\), so the same conclusion holds
for \(\partial_r^q w_r\) on \(A_1\), for every integer \(q\geq0\).
In particular, for every integer \(q\geq0\) and every \(M>0\),
\[
\|\partial_r^q w_r\|_{L^2(A_1)}=O(r^{-M}),
\qquad
\|\partial_r^q u_r\|_{L^2(A)}=O(r^{-M}).
\]

We next control the commutator terms. Since
\((\Delta_X+r)w_r=0\) on \(A_1\), differentiation with respect to
\(r\) gives
\[
(\Delta_X+r)\partial_r^q w_r
=-q\,\partial_r^{q-1}w_r,
\qquad q\geq1.
\]
Together with \(\Delta_Xw_r=-rw_r\), this yields, for every
\(q\geq0\),
\[
\Delta_X\partial_r^q w_r
=-r\partial_r^q w_r-q\partial_r^{q-1}w_r,
\]
where the second term is omitted when \(q=0\). Given any \(N>0\),
apply the preceding rapid-decay bounds with exponent \(N+1\).
The additional factor \(r\) is then absorbed, and hence
\(\|\Delta_X\partial_r^q w_r\|_{L^2(A_1)}=O(r^{-N})\).
The first estimate in
Lemma~\ref{lem:annular-commutator-estimate} now shows that
\[
\|[\Delta_X,\chi]\partial_r^q w_r\|_{L^2(A)}
=O(r^{-N}),
\qquad q\geq0.
\]

Because the commutator is independent of \(r\), differentiation of
the comparison pairing gives
\[
\begin{aligned}
F_{sj}'(r)
&=\left\langle[\Delta_X,\chi]\partial_rw_r,u_r\right\rangle_{L^2(A)}
+\left\langle[\Delta_X,\chi]w_r,\partial_ru_r\right\rangle_{L^2(A)},\\
F_{sj}''(r)
&=\left\langle[\Delta_X,\chi]\partial_r^2w_r,u_r\right\rangle_{L^2(A)}
+2\left\langle[\Delta_X,\chi]\partial_rw_r,
\partial_ru_r\right\rangle_{L^2(A)}\\
&\quad
+\left\langle[\Delta_X,\chi]w_r,
\partial_r^2u_r\right\rangle_{L^2(A)}.
\end{aligned}
\]
The same pairing formula gives \(F_{sj}(r)\) itself. More generally,
repeated differentiation and the fact that \([\Delta_X,\chi]\) is
independent of \(r\) give, inductively,
\[
F_{sj}^{(q)}(r)
=\sum_{a=0}^q\binom qa
\left\langle
[\Delta_X,\chi]\partial_r^a w_r,
\partial_r^{q-a}u_r
\right\rangle_{L^2(A)}.
\]
For each \(0\leq a\leq q\), both factors on the right are
rapidly decreasing in their respective \(L^2\)-norms. Thus
Cauchy--Schwarz gives
\(F_{sj}^{(q)}(r)=O(r^{-N})\) for every \(q\geq0\) and every
\(N>0\).

Finally, the affine change-of-basis formulas show that the difference
of the two Weyl matrices in the \((\gamma,\gamma^\#)\)-bases is
\(\mathsf D_k\) times their difference in the \((\beta,\beta^\#)\)-bases
times \(\mathsf D_k\). Since \(\mathsf D_k\) is independent of \(r\), the same
estimates hold entrywise in the \((\gamma_k,\gamma_k^\#)\)-basis.
For \(1\leq p\leq n-1\), the explicit spherical-model formula gives
\(S_{0,n-p,-p}^{\widehat X_n,(\gamma_0,\gamma_0^\#)}(-r)
=S_{k,p}^{(\infty)}\), independently of \(r\). Taking
\((s,j)=(n-p,-p)\) proves the final assertion.
\end{proof}

\begin{cor}
As \(\lambda\to-\infty\), the entries of
\(S_k^{X,(\gamma_k,\gamma_k^\#)}(\lambda)\) have the following
asymptotic behavior. The zero-mode entry satisfies
\[\displaystyle
S_{k,00}^{X,(\gamma_k,\gamma_k^\#)}(\lambda)
=
-\frac{1}{n}
\left(
\gamma_E+\psi(\nu+1)
+\frac{\pi}{2}\cot(\pi\nu)
-\log A_k
\right)
+
O(|\lambda|^{-\infty}).
\]
For \(0<|j|\leq n-1\),
\[\displaystyle
S_{k,jj}^{X,(\gamma_k,\gamma_k^\#)}(\lambda)
=
A_k^{-2|j|/n}
s_j^{\widehat X_n,(\beta_0,\beta_0^\#)}(\lambda)
+
O(|\lambda|^{-\infty})
.\]
Moreover, for \(1\leq j\leq n-1\),
\[\displaystyle
S_{k,j-n,j}^{X,(\gamma_k,\gamma_k^\#)}(\lambda)
=
\frac{\sqrt{j(n-j)}}{nA_k}z_k
+
O(|\lambda|^{-\infty})
,\]
while
\[\displaystyle
S_{k,n-j,-j}^{X,(\gamma_k,\gamma_k^\#)}(\lambda)
=
\frac{\sqrt{j(n-j)}}{nA_k}\overline{z_k}
+
O(|\lambda|^{-\infty})
.\]

All remaining entries satisfy
\[\displaystyle
S_{k,\ell j}^{X,(\gamma_k,\gamma_k^\#)}(\lambda)
=
O(|\lambda|^{-\infty})
.\]
\end{cor}

\begin{proof}
This follows from
\(\displaystyle
S_{k,\ell j}^{X,(\gamma_k,\gamma_k^\#)}(\lambda)
-
S_{0,\ell j}^{\widehat X_n,(\gamma_0,\gamma_0^\#)}(\lambda)
=
O(|\lambda|^{-\infty})
\)
and the explicit formulas for the entries of
\(S_0^{\widehat X_n,(\gamma_0,\gamma_0^\#)}(\lambda)\).
\end{proof}

The off-diagonal limits in the preceding corollary also explain the
coefficient attached to a ramification point in the final determinant
formula.  If \(P_k\) has ramification index \(n\), then for
\(1\leq p\leq n-1\),
\[
S_{k,n-p,-p}^{X,(\gamma_k,\gamma_k^\#)}(-\infty)
:=
\lim_{r\to\infty}
S_{k,n-p,-p}^{X,(\gamma_k,\gamma_k^\#)}(-r)
=
\frac{\sqrt{p(n-p)}}{n(1+|z_k|^2)}\,\overline{z_k}.
\]
The weights occurring in the variational trace formula give
\[
\sum_{p=1}^{n-1}\frac{\sqrt{p(n-p)}}{n}
S_{k,n-p,-p}^{X,(\gamma_k,\gamma_k^\#)}(-\infty)
=
\frac16\left(n-\frac1n\right)
\frac{\overline{z_k}}{1+|z_k|^2}.
\]
Since
\(\partial_{z_k}\log\rho(z_k,\overline{z_k})
=-2\overline{z_k}/(1+|z_k|^2)\), this weighted sum is exactly what
produces the exponent
\(\frac1{12}(n-n^{-1})\) of the round conformal factor.  Hence that
exponent is not an external correction: it is already encoded in the
affine coordinate-change term of the local \(S\)-matrix.

\section{Hurwitz Coordinates and the Gluing Construction}\label{sec:hurwitz-spaces}
In order to study the variation of the conic metric
\(g_\varphi:=\varphi^*ds_{\mathrm{rd}}^2\) and of the associated spectral
quantities, we need a natural parameter space for deformations of the
branched covering \(\varphi:X\to\mb P^1\) with fixed Hurwitz data.
This parameter space is the Hurwitz space
\(H_{g,d,N}^{\mb P^1}\), whose points are equivalence classes
\([\varphi]\) of branched coverings.
In this section, we review the classical Hurwitz-space construction of
Fulton \cite{Fulton1969} in the present setting with fixed ramification
data; for Hurwitz spaces with arbitrary ramification indices, see also
\cite{KokotovKorotkin2004}.
The basic feature of the construction is
that, after equipping \(H_{g,d,N}^{\mb P^1}\) with the Hurwitz
topology, the branch-value map
\(\operatorname{br}:H_{g,d,N}^{\mb P^1}\to C_N(\mb P^1)\)
is a topological covering map onto the configuration space of \(N\)
distinct branch values. We therefore first describe the complex
manifold structure and local holomorphic coordinates on
\(C_N(\mb P^1)\). These coordinates lift through
\(\operatorname{br}\) to local holomorphic coordinates on
\(H_{g,d,N}^{\mb P^1}\), making the branch-value map locally
biholomorphic.
After establishing these Hurwitz coordinates, we fix a point
\([\varphi]\in H_{g,d,N}^{\mb P^1}\) and give an explicit geometric
realization of deformation in one coordinate direction. If \(Q_k\)
is a branch value and \(z_k\) is its corresponding local coordinate,
we vary \(z_k\) to \(z_k+t\) while keeping the remaining Hurwitz
coordinates fixed. We then construct a representative
\(\varphi_t:X_t\to\mb P^1\) of the corresponding point of
\(H_{g,d,N}^{\mb P^1}\). This gives a family of conic metrics
\(g_t:=\varphi_t^*ds_{\mathrm{rd}}^2\), whose spectral variation will be
studied in the subsequent sections.

For a positive integer \(N\), let
\(F_N(\mb P^1)\) be the space of ordered \(N\)-tuples of distinct
points, and let \(C_N(\mb P^1)=F_N(\mb P^1)/S_N\) be the space
of unordered configurations. The permutation action is free and
holomorphic; write \(\pi_N\) for the quotient map.
For pairwise disjoint coordinate discs \(D_1,\ldots,D_N\), set
\[
\mathcal N(D_1,\ldots,D_N)
:=\{A\in C_N(\mb P^1):\#(A\cap D_i)=1\text{ for every }i\}.
\]
The restriction of \(\pi_N\) to \(D_1\times\cdots\times D_N\)
is a homeomorphism onto this neighborhood. If \(\psi_i\) is a
holomorphic coordinate on \(D_i\), its inverse followed by
\(\psi_1\times\cdots\times\psi_N\) gives a chart on
\(C_N(\mb P^1)\). The transitions are locally permutations
followed by holomorphic coordinate changes. The finite quotient
is Hausdorff and second countable, so these charts make
\(C_N(\mb P^1)\) a complex manifold of dimension \(N\).

Let $\varphi:X\to\mb P^{1}$ and $\psi:Y\to\mb P^{1}$ be meromorphic
functions. We say that $\varphi$ and $\psi$ are equivalent if there
exists a biholomorphism $\Theta:X\to Y$ such that the following diagram
commutes:
\[
\begin{tikzcd}
X \arrow[r, "\Theta", "\sim"'] \arrow[dr, "\varphi"']
&
Y \arrow[d, "\psi"]
\\
&
\mb P^{1}.
\end{tikzcd}
\]
Equivalently, $\psi\circ\Theta=\varphi$. Suppose that $\varphi$ and $\psi$ are equivalent through a
biholomorphism $\Theta:X\to Y$. Then $\Theta$ maps the ramification points of
$\varphi$ bijectively onto those of $\psi$, and
$e_P(\varphi)=e_{\Theta(P)}(\psi)$ for every $P\in X$. In particular,
$\varphi$ and $\psi$ have the same set of branch values in
$\mb P^{1}$.

For each $A\in C_N(\mb P^{1})$, let
$H_{g,d}^{\mb P^{1}}(A)$ denote the set of all equivalence classes
$[\varphi]$, where $\varphi:X\to\mb P^{1}$ is a degree-$d$
meromorphic function on a compact connected Riemann surface $X$ of
genus $g$ whose set of branch values is exactly $A$. This definition
is well-defined because equivalence preserves the genus and degree of
the covering, as well as its branch values and ramification indices.
Let
\[
H_{g,d,N}^{\mb P^{1}}
=
\coprod_{A\in C_N(\mb P^{1})}
H_{g,d}^{\mb P^{1}}(A).
\]
Thus, $H_{g,d,N}^{\mb P^{1}}$ is the moduli space of equivalence
classes of degree-$d$ meromorphic functions
$\varphi:X\to\mb P^{1}$, where $X$ is a compact connected Riemann
surface of genus $g$ and $\varphi$ has exactly $N$ distinct branch
values. The branch-value map
\[
\operatorname{br}:
H_{g,d,N}^{\mb P^1}
\longrightarrow
C_N(\mb P^1),
\qquad
[\varphi]\longmapsto\operatorname{Br}(\varphi),
\]
is well-defined. Notice that, for every $A\in C_N(\mb P^1)$,
\(
\operatorname{br}^{-1}(A)
=
H_{g,d}^{\mb P^1}(A).
\)
For a degree-$d$ meromorphic function
\(\varphi:X\to\mb P^1\) and a point \(Q\in\mb P^1\), write
\(\varphi^{-1}(Q)=\{P_1,\ldots,P_s\}\). The ramification profile of
\(\varphi\) over \(Q\) is the partition of \(d\) defined by
\[
\mu_\varphi(Q)
=
\bigl(e_{P_1}(\varphi),\ldots,e_{P_s}(\varphi)\bigr)\vdash d,
\]
where the ramification indices are arranged in nonincreasing order.
The following standard transport construction goes back to Fulton
\cite{Fulton1969}; its use on a stratum with fixed arbitrary
ramification data is consistent with the Hurwitz-space framework of
\cite{KokotovKorotkin2004}.

\begin{prop}
Let \(A=\{Q_1,\ldots,Q_N\}\in C_N(\mb P^1)\), and let
\(D_1,\ldots,D_N\) be pairwise disjoint open discs such that
\(Q_i\in D_i\) for every \(i\). For every
\(A'\in\mathcal N(D_1,\ldots,D_N)\), let \(Q_i'\) denote the unique
point of \(A'\cap D_i\). Then there is a natural bijection
\[
\tau_{A,A'}:
H_{g,d}^{\mb P^1}(A)
\longrightarrow
H_{g,d}^{\mb P^1}(A').
\]
More precisely, if
\(\tau_{A,A'}([\varphi])=[\varphi_{A'}]\), then
\(\mu_{\varphi_{A'}}(Q_i')=\mu_\varphi(Q_i)\) for every
\(1\leq i\leq N\).
\end{prop}

\begin{proof}
The construction is the standard transport of a branched covering
under a deformation of its branch values. Keep the covering fixed over
\(\mb P^1\setminus\bigcup_{k=1}^N D_k\).
For each \(k\), the connected components of
\(\varphi^{-1}(D_k)\) are discs \(U_{k1},\ldots,U_{ks_k}\), each
containing a unique point \(P_{kj}\in\varphi^{-1}(Q_k)\). The
restriction of \(\varphi\) to \(U_{kj}\) is a degree-\(n_{kj}\)
branched covering of \(D_k\), where
\(n_{kj}=e_{P_{kj}}(\varphi)\).
Given \(A'=\{Q_1',\ldots,Q_N'\}\), replace each \(U_{kj}\) by a disc
carrying a degree-\(n_{kj}\) branched covering of \(D_k\) with branch
value \(Q_k'\), and glue it to the unchanged exterior covering along
the corresponding boundary component. Such a local replacement is
unique up to equivalence once the boundary covering is fixed.
Indeed, after deleting the branch value, the local covering is
classified by its monodromy around the puncture. A connected
degree-\(n_{kj}\) local piece has monodromy an \(n_{kj}\)-cycle,
and the prescribed boundary covering fixes this monodromy together
with the attaching data. The
resulting covering \(\varphi_{A'}:X_{A'}\to\mb P^1\) has branch values
\(A'\) and satisfies
\(\mu_{\varphi_{A'}}(Q_k')=\mu_\varphi(Q_k)\) for
\(1\leq k\leq N\).
Thus the construction determines a well-defined map
\(\tau_{A,A'}([\varphi]):=[\varphi_{A'}]\).
Interchanging \(A\) and \(A'\) gives the inverse construction, so
\(\tau_{A,A'}\) is a bijection.
\end{proof}

For \([\varphi]\in H_{g,d}^{\mb P^1}(A)\), define
\[
\mathcal U([\varphi];D_1,\ldots,D_N)
:=\{\tau_{A,A'}([\varphi]):A'\in\mathcal N(D_1,\ldots,D_N)\}.
\]
Transport within these discs satisfies
\(\tau_{A',A''}\circ\tau_{A,A'}=\tau_{A,A''}\), and is compatible
with restriction to smaller discs. Consequently, the sets
\(\mathcal U([\varphi];D_1,\ldots,D_N)\) form a basis for the
Hurwitz topology, and
\[
\operatorname{br}^{-1}\mathcal N(D_1,\ldots,D_N)
=\coprod_{[\varphi]\in H_{g,d}^{\mb P^1}(A)}
\mathcal U([\varphi];D_1,\ldots,D_N).
\]
On each summand, \(\operatorname{br}\) is a homeomorphism with
inverse \(s_{[\varphi];D_1,\ldots,D_N}(A'):=\tau_{A,A'}([\varphi])\).
Thus the branch-value map is a covering map. Its fibers are finite:
the monodromy of a degree-\(d\) covering is determined by finitely
many permutations in \(S_d\), up to simultaneous conjugation.
The Hurwitz space is therefore Hausdorff and second countable.

Lift the configuration-space charts by setting
\[
\Phi_{[\varphi];D_1,\ldots,D_N}
:=(\psi_1\times\cdots\times\psi_N)
\circ(\pi_N|_{D_1\times\cdots\times D_N})^{-1}
\circ\operatorname{br}|_{\mathcal U([\varphi];D_1,\ldots,D_N)}.
\]
Their biholomorphic transitions define the unique complex structure
for which \(\operatorname{br}\) is locally biholomorphic.
In particular, \(H_{g,d,N}^{\mb P^1}\) has complex dimension \(N\),
and the ramification profiles remain fixed on each such chart.

Let \([\varphi]\in H_{g,d,N}^{\mb P^1}\), and write
\(\operatorname{br}([\varphi])=\{Q_1,\ldots,Q_N\}\).
Choose pairwise disjoint coordinate discs
\(D_1,\ldots,D_N\subset\mb P^1\) such that \(Q_i\in D_i\) for every
\(1\leq i\leq N\). Let \((U_0,x)\) and \((U_\infty,y)\) be the
standard affine coordinate charts on \(\mb P^1\). After shrinking the
discs if necessary, assume that \(D_i\subset U_0\) when
\(Q_i\neq\infty\) and \(D_i\subset U_\infty\) when \(Q_i=\infty\).
For each \(1\leq i\leq N\), define
\[
\xi_i:=
\begin{cases}
x|_{D_i}, & Q_i\neq\infty,\\[4pt]
y|_{D_i}, & Q_i=\infty,
\end{cases}
\qquad
z_i:=\xi_i(Q_i).
\]
Shrink \(D_i\) so that \(\xi_i(D_i)=D_{z_i}(\epsilon_i)\)
for some \(\epsilon_i>0\). From now on, take \(\psi_i=\xi_i\)
in the definition of \(\Phi_{[\varphi];D_1,\ldots,D_N}\).
Lifting these branch-value coordinates gives the local Hurwitz
coordinates \((z_1,\ldots,z_N)\) near \([\varphi]\). In particular, \(z_i=0\) whenever
\(Q_i=\infty\). Fix \(1\leq k\leq N\). Write
\(\varphi^{-1}(Q_k)=\{P_{k1},\ldots,P_{ks_k}\}\), and let
\(n_{kj}=e_{P_{kj}}(\varphi)\) be the ramification index of \(\varphi\)
at \(P_{kj}\). Thus, \(n_{kj}=1\) if \(P_{kj}\) is unramified, and
\(\sum_{j=1}^{s_k}n_{kj}=d\).
After shrinking \(D_k\) around \(Q_k\), if necessary, assume that
\[
\varphi^{-1}(D_k)
=
\bigsqcup_{j=1}^{s_k} U_{kj},
\]
where, for each \(j\), \(U_{kj}\) is a coordinate disc centered at
\(P_{kj}\), and the discs \(U_{k1},\ldots,U_{ks_k}\) are pairwise
disjoint. The restriction
\(\varphi|_{U_{kj}}:U_{kj}\to D_k\) has degree \(n_{kj}\), with
\(P_{kj}\) as its unique ramification point when \(n_{kj}>1\).
Define
\(w_k:=\xi_k-z_k:D_k\to D_0(\epsilon_k)\), so that \(w_k(Q_k)=0\).
On \(\varphi^{-1}(D_k)\), set
\(d_k\varphi:=d(\xi_k\circ\varphi)=d(w_k\circ\varphi)\).
For each \(j\), choose a centered holomorphic coordinate
\(z_{kj}:U_{kj}\to D_0(\epsilon_k^{1/n_{kj}})\) satisfying
\(w_k\circ\varphi=z_{kj}^{\,n_{kj}}\).
Thus the local representation
\(\widetilde\varphi_{kj}:=w_k\circ\varphi\circ z_{kj}^{-1}\)
is \(\widetilde\varphi_{kj}(\zeta)=\zeta^{n_{kj}}\).

Let \(t\in D_0(\epsilon_k)\). For each \(1\leq j\leq s_k\), define
\[
\widetilde{\varphi}_{t;kj}:\mb C\longrightarrow\mb C,
\qquad
\widetilde{\varphi}_{t;kj}(\zeta)
=
t+\zeta^{n_{kj}},
\]
and set
\[
\Omega_{kj}(t)
:=
\widetilde{\varphi}_{t;kj}^{-1}\bigl(D_0(\epsilon_k)\bigr)
=
\left\{
\zeta\in\mb C:
|t+\zeta^{n_{kj}}|<\epsilon_k
\right\}.
\]
Let \(z_{kj,t}:\Omega_{kj}(t)\to\mb C\) denote the standard
coordinate, given by \(z_{kj,t}(\zeta)=\zeta\), and define
\[
\varphi_{t;kj}
:=
w_k^{-1}\circ\widetilde{\varphi}_{t;kj}
:
\Omega_{kj}(t)\longrightarrow D_k.
\]
Then \(w_k\circ\varphi_{t;kj}=t+z_{kj,t}^{\,n_{kj}}\).

At \(t=0\), we have
\(\Omega_{kj}(0)=D_0(\epsilon_k^{1/n_{kj}})\).
The biholomorphism \(z_{kj}:U_{kj}\to\Omega_{kj}(0)\) identifies
the original local covering with \(\varphi_{0;kj}\), since
\(\varphi_{0;kj}\circ z_{kj}=\varphi\) on \(U_{kj}\).
This observation suggests constructing the global deformation by
replacing each \(U_{kj}\) with \(\Omega_{kj}(t)\), carrying the
local map \(\varphi_{t;kj}\), while keeping the covering unchanged
away from these neighborhoods. To realize this replacement
holomorphically, we remove only a smaller closed disc from each
\(U_{kj}\), leaving an open annulus along which the old and new
pieces can be identified.


We now turn from the abstract Hurwitz coordinates to an explicit
representative of a one-coordinate deformation. For the gluing
construction, assume henceforth that \(|t|<\epsilon_k/4\).
Each \(\Omega_{kj}(t)\) contains the origin and is star-shaped
with respect to it; in particular, it is simply connected.
Define the fixed exterior part
\[
X_k^o
:=
X\setminus
\bigcup_{j=1}^{s_k}
\left\{
P\in U_{kj}:
|z_{kj}(P)|^{n_{kj}}\leq\frac34\epsilon_k
\right\},
\]
which is open in \(X\), and let
\[
A_{kj}
:=
X_k^o\cap U_{kj}
=
\left\{
P\in U_{kj}:
\frac34\epsilon_k
<
|z_{kj}(P)|^{n_{kj}}
<
\epsilon_k
\right\}.
\]
The choice of the inner boundary will ensure compatibility
with the cutoff used in the smooth trivializations constructed
in the next section.

The gluing maps must identify points having the same image
in \(D_k\), so that the local covering maps descend to a
well-defined map on the adjunction space.
We therefore seek holomorphic embeddings
\(\mathscr G_{kj,t}:A_{kj}\to\Omega_{kj}(t)\)
for which the following diagram commutes:
\[
\begin{CD}
A_{kj} @>{\mathscr G_{kj,t}}>> \Omega_{kj}(t) \\
@V{\varphi}VV @VV{\varphi_{t;kj}}V \\
D_k @= D_k.
\end{CD}
\]
In local coordinates, this condition becomes
\(t+(z_{kj,t}\circ\mathscr G_{kj,t})^{n_{kj}}
=z_{kj}^{\,n_{kj}}\).

For \(P\in A_{kj}\), we have
\(\lvert t/z_{kj}(P)^{n_{kj}}\rvert<1/3\).
Using the holomorphic branch of
\(u\mapsto(1-u)^{1/n_{kj}}\) on \(D_0(1)\)
whose value at \(0\) is \(1\), define
\[
\mathscr G_{kj,t}(P)
:=
z_{kj}(P)
\left(
1-\frac{t}{z_{kj}(P)^{n_{kj}}}
\right)^{1/n_{kj}}.
\]
Here the right-hand side is regarded as a point of the planar
piece \(\Omega_{kj}(t)\).
This map is biholomorphic onto its image, which is the
open annulus
\[
\mathscr G_{kj,t}(A_{kj})
=
\left\{
\zeta\in\Omega_{kj}(t):
\frac34\epsilon_k
<
|t+\zeta^{n_{kj}}|
<
\epsilon_k
\right\}.
\]
Indeed, on this annulus,
\(|\zeta|^{n_{kj}}
\geq|t+\zeta^{n_{kj}}|-|t|>\epsilon_k/2\),
and the inverse is given by
\[
\mathscr G_{kj,t}^{-1}(\zeta)
=
z_{kj}^{-1}
\left(
\zeta
\left(
1+\frac{t}{\zeta^{n_{kj}}}
\right)^{1/n_{kj}}
\right),
\]
where the root is again normalized to have value \(1\) at \(1\).
By construction,
\(\varphi_{t;kj}\circ\mathscr G_{kj,t}=\varphi\)
on \(A_{kj}\).

To form the adjunction space, set
\[
Y_t:=\bigsqcup_{j=1}^{s_k}\Omega_{kj}(t),
\qquad
\mathcal A_t
:=
\bigsqcup_{j=1}^{s_k}\mathscr G_{kj,t}(A_{kj})
\subset Y_t.
\]
Define the attaching map \(f_t:\mathcal A_t\to X_k^o\) by
\(f_t(\mathscr G_{kj,t}(P)):=P\) for
\(P\in A_{kj}\) and \(1\leq j\leq s_k\).
The required adjunction space is
\[
X_t
:=
X_k^o\cup_{f_t}Y_t
=
\left(
X_k^o
\sqcup
\bigsqcup_{j=1}^{s_k}\Omega_{kj}(t)
\right)\big/\sim_t,
\]
equipped with the quotient topology, where \(\sim_t\)
is generated by \(q\sim_t f_t(q)\) for \(q\in\mathcal A_t\).
Equivalently, the identifications are
\(P\sim_t\mathscr G_{kj,t}(P)\) for \(P\in A_{kj}\).

\definecolor{oldblue}{RGB}{214,231,248}
\definecolor{newred}{RGB}{249,226,226}
\definecolor{newgreen}{RGB}{225,241,211}
\definecolor{newpurple}{RGB}{232,225,245}

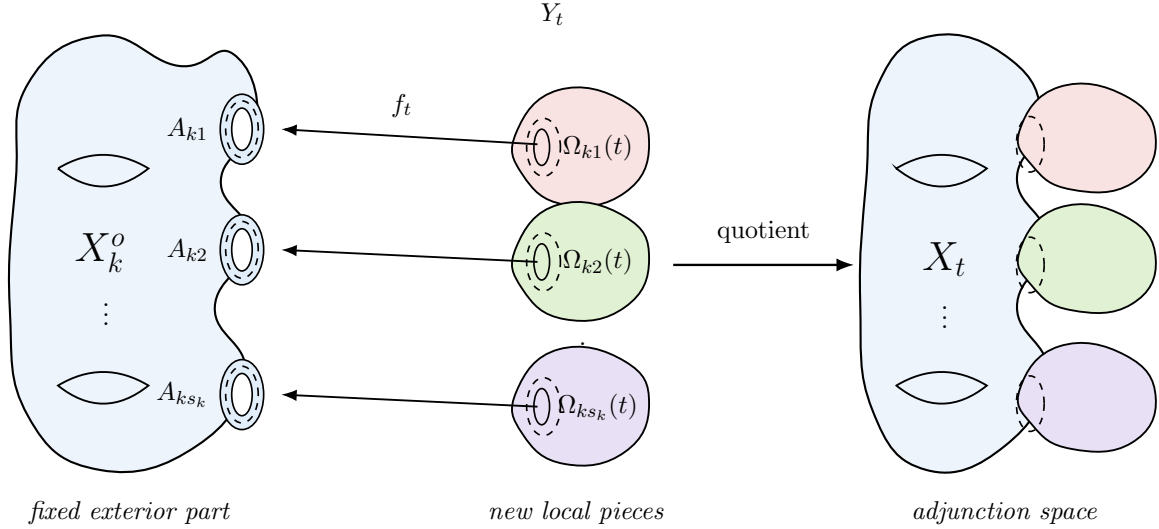
\begin{figure}[htbp]
\centering

\resizebox{\linewidth}{!}{%
\begin{tikzpicture}[
    scale=0.90,
    >=Latex,
    every node/.style={font=\small},
    surface/.style={
        draw=black,
        line width=0.9pt,
        fill=oldblue!45
    },
    localpiece/.style={
        draw=black,
        line width=0.8pt
    },
    annulus/.style={
        draw=black,
        dashed,
        line width=0.7pt
    }
]


\begin{scope}[shift={(0,0)}]

\draw[surface]
plot[smooth cycle,tension=0.8] coordinates {
    (-4.0, 3.0)
    (-2.8, 3.8)
    (-1.5, 3.3)
    (-0.6, 3.5)
    (-0.2, 2.4)
    (-0.8, 1.4)
    (-0.3, 0.4)
    (-0.9,-0.7)
    (-0.4,-1.8)
    (-1.2,-3.0)
    (-2.8,-3.4)
    (-4.0,-2.5)
    (-4.3,-0.7)
    (-4.2, 1.4)
};

\node at (-2.8,0.2) {\Large \(X_k^o\)};


\draw[line width=0.8pt]
    (-3.5,1.6)
    .. controls (-3.0,1.2) and (-2.5,1.2) .. (-2.0,1.6)
    .. controls (-2.5,1.9) and (-3.0,1.9) .. cycle;

\draw[line width=0.8pt]
    (-3.5,-2.0)
    .. controls (-3.0,-2.4) and (-2.5,-2.4) .. (-2.0,-2.0)
    .. controls (-2.5,-1.7) and (-3.0,-1.7) .. cycle;

\node at (-2.7,-0.7) {\(\vdots\)};


\foreach \y/\lab in {
    2.25/{A_{k1}},
    0.25/{A_{k2}},
    -2.15/{A_{ks_k}}
} {
    \draw[fill=oldblue!70,line width=0.7pt]
        (-0.45,\y) ellipse (0.36 and 0.58);
    \draw[fill=white,line width=0.7pt]
        (-0.45,\y) ellipse (0.17 and 0.35);
    \draw[annulus]
        (-0.45,\y) ellipse (0.27 and 0.47);
    \node[left] at (-0.85,\y) {\(\lab\)};
}

\node at (-2.3,-4.1) {\textit{fixed exterior part}};

\end{scope}


\begin{scope}[shift={(4.0,0)}]

\node at (0.7,4.15) {
    \(Y_t\)
};


\draw[localpiece,fill=newred]
plot[smooth cycle,tension=0.8] coordinates {
    (0.1,2.2)
    (0.8,2.9)
    (1.9,2.7)
    (2.3,2.0)
    (1.9,1.2)
    (0.9,1.0)
    (0.2,1.5)
};

\draw[annulus]
    (0.55,1.95) ellipse (0.28 and 0.48);
\draw[fill=newred,line width=0.7pt]
    (0.52,1.95) ellipse (0.13 and 0.30);

\node at (1.45,1.95) {\(\Omega_{k1}(t)\)};


\draw[localpiece,fill=newgreen]
plot[smooth cycle,tension=0.8] coordinates {
    (0.1,0.3)
    (0.8,1.0)
    (1.9,0.8)
    (2.3,0.1)
    (1.9,-0.7)
    (0.9,-0.9)
    (0.2,-0.4)
};

\draw[annulus]
    (0.55,0.05) ellipse (0.28 and 0.48);
\draw[fill=newgreen,line width=0.7pt]
    (0.52,0.05) ellipse (0.13 and 0.30);

\node at (1.45,0.05) {\(\Omega_{k2}(t)\)};

\node at (1.2,-1.35) {\(\vdots\)};


\draw[localpiece,fill=newpurple]
plot[smooth cycle,tension=0.8] coordinates {
    (0.1,-2.1)
    (0.8,-1.4)
    (1.9,-1.6)
    (2.3,-2.3)
    (1.9,-3.1)
    (0.9,-3.3)
    (0.2,-2.8)
};

\draw[annulus]
    (0.55,-2.35) ellipse (0.28 and 0.48);
\draw[fill=newpurple,line width=0.7pt]
    (0.52,-2.35) ellipse (0.13 and 0.30);

\node at (1.45,-2.35) {\(\Omega_{ks_k}(t)\)};

\node at (1.1,-4.1) {\textit{new local pieces}};

\end{scope}


\draw[->,line width=0.8pt]
    (4.45,2.0) -- (0.2,2.25);
\draw[->,line width=0.8pt]
    (4.45,0.05) -- (0.2,0.25);
\draw[->,line width=0.8pt]
    (4.45,-2.35) -- (0.2,-2.15);

\node at (2.2,2.65) {\(f_t\)};

\draw[->,line width=1pt]
    (6.7,0) -- (9.7,0);
\node at (8.2,0.55) {quotient};


\begin{scope}[shift={(12.2,0)}]


\draw[surface]
plot[smooth cycle,tension=0.8] coordinates {
    (-2.1, 3.0)
    (-1.2, 3.8)
    (0.1, 3.4)
    (0.6, 2.4)
    (0.2, 1.4)
    (0.7, 0.4)
    (0.2,-0.7)
    (0.7,-1.8)
    (0.0,-3.0)
    (-1.3,-3.4)
    (-2.2,-2.4)
    (-2.4,-0.6)
    (-2.3, 1.4)
};


\draw[localpiece,fill=newred]
plot[smooth cycle,tension=0.8] coordinates {
    (0.25,2.25)
    (0.9,2.95)
    (2.0,2.8)
    (2.5,2.1)
    (2.1,1.35)
    (1.1,1.25)
    (0.45,1.65)
};


\draw[localpiece,fill=newgreen]
plot[smooth cycle,tension=0.8] coordinates {
    (0.25,0.25)
    (0.9,0.95)
    (2.0,0.8)
    (2.5,0.1)
    (2.1,-0.65)
    (1.1,-0.75)
    (0.45,-0.35)
};


\draw[localpiece,fill=newpurple]
plot[smooth cycle,tension=0.8] coordinates {
    (0.25,-2.05)
    (0.9,-1.35)
    (2.0,-1.5)
    (2.5,-2.2)
    (2.1,-2.95)
    (1.1,-3.05)
    (0.45,-2.65)
};


\foreach \y in {2.0,0.0,-2.3} {
    \draw[annulus]
        (0.42,\y) ellipse (0.23 and 0.46);
}


\draw[line width=0.8pt]
    (-1.8,1.6)
    .. controls (-1.3,1.2) and (-0.8,1.2) .. (-0.3,1.6)
    .. controls (-0.8,1.9) and (-1.3,1.9) .. cycle;

\draw[line width=0.8pt]
    (-1.8,-2.0)
    .. controls (-1.3,-2.4) and (-0.8,-2.4) .. (-0.3,-2.0)
    .. controls (-0.8,-1.7) and (-1.3,-1.7) .. cycle;

\node at (-1.0,0.1) {\Large \(X_t\)};
\node at (-1.0,-0.75) {\(\vdots\)};

\node at (0.0,-4.1) {\textit{adjunction space}};
\end{scope}

\end{tikzpicture}%
}

\caption{Adjunction construction of the deformed surface
\(X_t=X_k^o\cup_{f_t}Y_t\).
The dashed curves indicate the open annuli identified by
\(f_t=\mathscr G_{kj,t}^{-1}\) on each
\(\mathscr G_{kj,t}(A_{kj})\subset\mathcal A_t\).
The quotient identifies \(P\in A_{kj}\) with
\(\mathscr G_{kj,t}(P)\in\Omega_{kj}(t)\).
The three local pieces shown schematically represent the first,
second, and last components.}
\label{fig:adjunction-Xt}
\end{figure}

Since the attaching maps are biholomorphisms between open
annuli, the quotient map
\(q_t:X_k^o\sqcup Y_t\to X_t\) is open.
We next verify that the equivalence relation is closed.

Fix \(j\), write \(n=n_{kj}\), and suppose that
\(P_\nu\in A_{kj}\),
\(P_\nu\to P\in X_k^o\), and
\(\mathscr G_{kj,t}(P_\nu)\to\zeta\in\Omega_{kj}(t)\).
Choose \(\rho\) such that
\(|t+\zeta^n|<\rho<\epsilon_k\).
For all sufficiently large \(\nu\), we have
\[
|z_{kj}(P_\nu)|^n
=
\left|t+\mathscr G_{kj,t}(P_\nu)^n\right|
<
\rho.
\]
Thus \(P_\nu\) eventually lies in the compact subset
\(z_{kj}^{-1}(\overline{D_0(\rho^{1/n})})\) of \(U_{kj}\),
so \(P\in U_{kj}\).
Since \(P\in X_k^o\), it follows that
\(P\in X_k^o\cap U_{kj}=A_{kj}\).
Continuity then gives
\(\zeta=\mathscr G_{kj,t}(P)\).
Hence the graph of \(\mathscr G_{kj,t}\) is closed in
\(X_k^o\times\Omega_{kj}(t)\).

The annuli \(A_{kj}\) are pairwise disjoint, and each
\(\mathscr G_{kj,t}\) is injective.
Consequently, the equivalence relation is the union of the
diagonal, the finitely many graphs of the gluing maps,
and their transposes.
It is therefore closed.
Since \(q_t\) is open, \(X_t\) is Hausdorff.
The biholomorphic identifications then endow \(X_t\) with
a complex structure.
Topologically, the construction replaces finitely many
coordinate discs by discs, so \(X_t\) is a compact connected
Riemann surface of genus \(g\).

Identifying the pieces with their images in \(X_t\), define
\[
\varphi_t:X_t\longrightarrow\mb P^1,
\qquad
\varphi_t
=
\begin{cases}
\varphi,
& \text{on }X_k^o,\\
\varphi_{t;kj},
& \text{on }\Omega_{kj}(t),
\quad 1\leq j\leq s_k.
\end{cases}
\]
The compatibility condition
\(\varphi_{t;kj}\circ\mathscr G_{kj,t}=\varphi\)
shows that \(\varphi_t\) is well defined and holomorphic.
Since the covering is unchanged outside \(D_k\),
its degree is \(d\).

For \(1\leq j\leq s_k\), set
\(D_{kj}(t):=
D_0((\epsilon_k/2)^{1/n_{kj}})
\subset\Omega_{kj}(t)\).
Also define \(Q_k(t):=w_k^{-1}(t)\) and
\(D_k(t):=w_k^{-1}(D_t(\epsilon_k/2))\).
These inclusions are valid because
\(|t|+\epsilon_k/2<3\epsilon_k/4<\epsilon_k\).
The local representation of \(\varphi_t\) is given by
\[
\begin{CD}
D_{kj}(t) @>{\varphi_t}>> D_k(t) \\
@V{z_{kj,t}}VV @VV{w_k}V \\
D_0\bigl((\epsilon_k/2)^{1/n_{kj}}\bigr)
@>{\widetilde{\varphi}_{t;kj}}>>
D_t(\epsilon_k/2),
\end{CD}
\]
where
\(\widetilde{\varphi}_{t;kj}(\zeta)=t+\zeta^{n_{kj}}\).
Let \(P_{kj}(t)\) be the point corresponding to
\(z_{kj,t}=0\).
Then
\[
\varphi_t(P_{kj}(t))=Q_k(t),
\qquad
e_{P_{kj}(t)}(\varphi_t)=n_{kj}.
\]
Thus the ramification profile over \(Q_k(t)\) is
\((n_{k1},\ldots,n_{ks_k})\), while all other branch values
and ramification data remain unchanged.
Hence
\[
\operatorname{Br}(\varphi_t)
=
A(t)
:=
\{Q_1,\ldots,Q_{k-1},Q_k(t),Q_{k+1},\ldots,Q_N\},
\]
and \([\varphi_t]\in H_{g,d,N}^{\mb P^1}\).

At \(t=0\), we have
\(\mathscr G_{kj,0}=z_{kj}\) on \(A_{kj}\).
The identity on \(X_k^o\), together with \(z_{kj}^{-1}\)
on \(\Omega_{kj}(0)\), therefore identifies
\((X_0,\varphi_0)\) biholomorphically with \((X,\varphi)\).
We henceforth use this identification and write
\(X_0=X\) and \(\varphi_0=\varphi\).

Let \(A:=\operatorname{Br}(\varphi)\).
By construction,
\([\varphi_t]=\tau_{A,A(t)}([\varphi])\).
In particular, \([\varphi_t]\) lies in the Hurwitz coordinate
neighborhood
\(\mathcal U([\varphi];D_1,\ldots,D_N)\).
If \(\gamma(t):=[\varphi_t]\), then
\[
\Phi_{[\varphi];D_1,\ldots,D_N}(\gamma(t))
=
(z_1,\ldots,z_{k-1},z_k+t,z_{k+1},\ldots,z_N).
\]
Thus \(\gamma\) is holomorphic,
\(\gamma(0)=[\varphi]\), and
\(\gamma'(0)
=\left.\frac{\partial}{\partial z_k}\right|_{[\varphi]}\).
The construction therefore realizes the \(k\)-th Hurwitz
coordinate deformation explicitly.

More generally, write
\(\mf z:=
\Phi_{[\varphi];D_1,\ldots,D_N}([\varphi])\).
Let \(B_{\mf z}(\delta)\) denote the open Euclidean ball of radius
\(\delta\) centered at \(\mf z\). Choose \(\delta>0\) sufficiently
small that this ball is contained in the image of the
Hurwitz chart and
\(\delta<\frac14\min_{1\leq i\leq N}\epsilon_i\).
For every \(\mf t\in\mb C^N\) with
\(\mf z+\mf t\in B_{\mf z}(\delta)\), there is a unique
Hurwitz point \([\varphi_{\mf t}]\) satisfying
\[
\Phi_{[\varphi];D_1,\ldots,D_N}([\varphi_{\mf t}])
=
\mf z+\mf t.
\]
Applying the construction simultaneously over the pairwise
disjoint discs \(D_1,\ldots,D_N\) gives a representative
\(\varphi_{\mf t}:X_{\mf t}\to\mb P^1\).

For variation in the \(k\)-th coordinate alone, we retain the
notation \(\varphi_t:X_t\to\mb P^1\) used above.
All complex coordinate derivatives are Wirtinger derivatives;
in particular, \(\partial_t|_{t=0}=\partial_{z_k}|_{[\varphi]}\).

\section{Hurwitz Deformations and Smooth Trivializations}\label{sec:smooth-trivializations}

The gluing construction produces varying Riemann surfaces. For the
operator-theoretic arguments, we now identify them with a fixed
underlying smooth surface. The use of cutoff diffeomorphisms for this
purpose is modeled on the moving-cone constructions in
\cite{HillairetKalvinKokotov2018,KalvinKokotov2019}. In the present
setting, however, one moving branch value may have several
ramification points above it, with possibly different indices. We
therefore construct a local interpolation map for each such point,
verify its compatibility with the corresponding gluing map, and
then combine the local maps into a smooth family of
orientation-preserving diffeomorphisms. The estimates ensuring global
invertibility are supplied in Appendix~\ref{app:local-interpolation}.

Choose \(\delta>0\) sufficiently small that the family is defined
for \(|t|<\delta\) and \(\delta<\epsilon_k/4\).
Choose a smooth function \(\beta:[0,\infty)\to[0,1]\) such that
\(\beta(s)=0\) for \(0\leq s\leq\epsilon_k/2\) and
\(\beta(s)=1\) for \(s\geq3\epsilon_k/4\).
For \(1\leq j\leq s_k\), write \(n=n_{kj}\).
For \(|t|<\delta\), define \(H_{kj,t}:\mb C\to\mb C\) by
\[
H_{kj,t}(\zeta)
:=
\begin{cases}
\displaystyle
\zeta
\left(
1-\beta(|\zeta|^n)\frac{t}{\zeta^n}
\right)^{1/n},
& \zeta\neq0,\\[8pt]
0,
& \zeta=0.
\end{cases}
\]
Whenever \(\beta(|\zeta|^n)\neq0\), we have
\(|\zeta|^n>\epsilon_k/2\), and therefore
\(\left|\beta(|\zeta|^n)t/\zeta^n\right|
\leq2|t|/\epsilon_k<1/2\).
Thus the expression inside the root lies in \(D_1(1/2)\).
We use the holomorphic branch of the \(n\)-th root on this disc
whose value at \(1\) is \(1\).

Since \(\beta(|\zeta|^n)\) vanishes near the origin,
\(H_{kj,t}(\zeta)=\zeta\) there. Consequently, \(H_{kj,t}\) is
smooth on \(\mb C\) and depends smoothly on
\((\Re t,\Im t,\Re\zeta,\Im\zeta)\).
It remains to verify that these local interpolation maps are
globally invertible and map the planar gluing domains onto one
another. The derivative estimates and global arguments are
collected in Appendix~\ref{app:local-interpolation}. In particular,
after shrinking \(\delta\) if necessary,
Proposition~\ref{prop:H-global-diffeomorphism} shows that each
\(H_{kj,t}\) is an orientation-preserving smooth diffeomorphism of
\(\mb C\), while
Proposition~\ref{prop:H-domain-diffeomorphism} gives
\(H_{kj,t}(\Omega_{kj}(0))=\Omega_{kj}(t)\).

Let \(U_{kj}(t)\subset X_t\) denote the image of the planar
piece \(\Omega_{kj}(t)\) under the quotient map.
We use the same symbol
\(z_{kj,t}:U_{kj}(t)\to\Omega_{kj}(t)\) for the induced
holomorphic coordinate, centered at \(P_{kj}(t)\).
At \(t=0\), we have
\(U_{kj}(0)=U_{kj}\), \(P_{kj}(0)=P_{kj}\), and
\(z_{kj,0}=z_{kj}\).
In particular,
\(z_{kj}(U_{kj})=\Omega_{kj}(0)
=D_0(\epsilon_k^{1/n_{kj}})\).

Define
\[
\widehat H_{kj,t}
:=
z_{kj,t}^{-1}\circ H_{kj,t}\circ z_{kj}
:
U_{kj}\longrightarrow U_{kj}(t).
\]
Then \(\widehat H_{kj,t}\) is an orientation-preserving smooth
diffeomorphism, and
\(z_{kj,t}\circ\widehat H_{kj,t}
=H_{kj,t}\circ z_{kj}\).

If \(P\in U_{kj}\) satisfies
\(|z_{kj}(P)|^{n_{kj}}\leq\epsilon_k/2\), then
\(\beta(|z_{kj}(P)|^{n_{kj}})=0\).
Consequently,
\(z_{kj,t}(\widehat H_{kj,t}(P))=z_{kj}(P)\).
In particular,
\(\widehat H_{kj,t}(P_{kj})=P_{kj}(t)\).

We now verify compatibility with the gluing maps used to
construct \(X_t\).
Recall that
\(A_{kj}=X_k^o\cap U_{kj}
=\{P\in U_{kj}:3\epsilon_k/4
<|z_{kj}(P)|^{n_{kj}}<\epsilon_k\}\).
For \(P\in A_{kj}\), we therefore have
\(\beta(|z_{kj}(P)|^{n_{kj}})=1\), and hence
\[
H_{kj,t}(z_{kj}(P))
=
z_{kj}(P)
\left(
1-\frac{t}{z_{kj}(P)^{n_{kj}}}
\right)^{1/n_{kj}}
=
\mathscr G_{kj,t}(P).
\]
Here \(\mathscr G_{kj,t}(P)\) is regarded as a point of
the planar piece \(\Omega_{kj}(t)\).
The branches of the root agree by their normalization at \(1\).
Since \(\mathscr G_{kj,0}=z_{kj}\) on \(A_{kj}\), we obtain
the compatibility relation
\[
H_{kj,t}\circ\mathscr G_{kj,0}
=
\mathscr G_{kj,t}
\quad\text{on }A_{kj}.
\]

\begin{thm}
Under the quotient descriptions of \(X\) and \(X_t\), define
\(\Theta_t:X\to X_t\) by
\(\Theta_t([P]):=[P]\) for \(P\in X_k^o\), and by
\(\Theta_t([\zeta]):=[H_{kj,t}(\zeta)]\) for
\(\zeta\in\Omega_{kj}(0)\).
Then, for every \(|t|<\delta\), the map \(\Theta_t\) is a
well-defined orientation-preserving smooth diffeomorphism.
Moreover, \(\Theta_0=\id_X\), and the family depends smoothly
on the real coordinates \(\Re t\) and \(\Im t\).
\end{thm}

\begin{proof}
We verify, in order, that the local formulas descend through the
gluing relations, define a smooth diffeomorphism, preserve
orientation, and depend smoothly on \(t\).
Recall that \(Y_t=\bigsqcup_{j=1}^{s_k}\Omega_{kj}(t)\),
and let \(q_t:X_k^o\sqcup Y_t\to X_t\) be the quotient map.
At \(t=0\), the quotient map is
\(q_0:X_k^o\sqcup Y_0\to X\).
The gluing relations are generated by
\(P\sim_t\mathscr G_{kj,t}(P)\) for \(P\in A_{kj}\).

Define
\(\mathcal T_t:X_k^o\sqcup Y_0\to X_k^o\sqcup Y_t\)
by \(\mathcal T_t(P):=P\) on \(X_k^o\) and
\(\mathcal T_t(\zeta):=H_{kj,t}(\zeta)\) on
\(\Omega_{kj}(0)\).
For \(P\in A_{kj}\), the compatibility relation gives
\[
\mathcal T_t(\mathscr G_{kj,0}(P))
=
H_{kj,t}(\mathscr G_{kj,0}(P))
=
\mathscr G_{kj,t}(P).
\]
Since \(P\sim_t\mathscr G_{kj,t}(P)\), the map
\(\mathcal T_t\) preserves the gluing relations.
Thus \(\mathcal T_t\) descends to the continuous map
\(\Theta_t:X\to X_t\) in the statement, satisfying
\(\Theta_t\circ q_0=q_t\circ\mathcal T_t\).

We verify smoothness directly from its local-coordinate definition.
Let \((V,x)\) be a smooth chart of \(X\) with
\(V\subset X_k^o\). The exterior identification determines a
corresponding chart \((\Theta_t(V),x_t)\) of \(X_t\), and in these
charts the coordinate representation of \(\Theta_t\) is
\(x_t\circ\Theta_t\circ x^{-1}=\id_{x(V)}\). On the other hand,
for each local piece \(U_{kj}\), the source and target charts are
\(z_{kj}\) and \(z_{kj,t}\), and
\[
z_{kj,t}\circ\Theta_t\circ z_{kj}^{-1}
=H_{kj,t}
\quad\text{on }\Omega_{kj}(0).
\]
The first coordinate representation is the identity and the second
is smooth. These charts cover \(X\), so the local-coordinate
criterion for smoothness shows that \(\Theta_t:X\to X_t\) is smooth.

To construct the inverse, take the identity on \(X_k^o\)
and \(H_{kj,t}^{-1}\) on each \(\Omega_{kj}(t)\).
The compatibility relation implies
\(H_{kj,t}^{-1}\circ\mathscr G_{kj,t}=\mathscr G_{kj,0}\) on
\(A_{kj}\).
Thus these maps also preserve the gluing relations and
descend to a map \(\Xi_t:X_t\to X\). In the corresponding exterior
charts its coordinate representation is the identity, while in the
local charts it is
\(z_{kj}\circ\Xi_t\circ z_{kj,t}^{-1}=H_{kj,t}^{-1}\).
It is therefore smooth by the same local-coordinate criterion.
Both compositions are the identity on every piece, so
\(\Xi_t=\Theta_t^{-1}\) and \(\Theta_t\) is a diffeomorphism.

The exterior identification preserves orientation, as do
the local diffeomorphisms \(H_{kj,t}\).
Consequently, \(\Theta_t\) preserves orientation.

It remains to formulate smooth dependence on \(t\) in the same
local-coordinate manner. Consider the total map
\[
\boldsymbol\Theta:
D_0(\delta)\times X
\longrightarrow
\coprod_{|t|<\delta}\{t\}\times X_t,
\qquad
\boldsymbol\Theta(t,P):=(t,\Theta_t(P)).
\]
Give the target the quotient topology obtained by gluing
\(D_0(\delta)\times X_k^o\) and the open sets
\[
\left\{
(t,\zeta)\in D_0(\delta)\times\mb C:
|t+\zeta^{n_{kj}}|<\epsilon_k
\right\}.
\]
The parameter-dependent gluing maps
\((t,P)\mapsto(t,\mathscr G_{kj,t}(P))\) are smooth.
The maps \((t,\zeta)\mapsto(t,H_{kj,t}(\zeta))\) have invertible
real differentials, with identity parameter block and invertible
surface block. The inverse function theorem and the bijectivity
of each \(H_{kj,t}\) give jointly smooth local inverses
\((t,\eta)\mapsto(t,H_{kj,t}^{-1}(\eta))\). These inverses
respect the gluing relations, so they descend, together with the
exterior identity, to a continuous inverse of
\(\boldsymbol\Theta\). The quotient topology therefore makes
\(\boldsymbol\Theta\) a homeomorphism. In particular, the total
space is Hausdorff and second countable, and the induced family
charts define a smooth manifold structure.
In the corresponding source and target charts, the coordinate
representations of \(\boldsymbol\Theta\) are respectively
\((t,x)\mapsto(t,x)\) and
\((t,\zeta)\mapsto(t,H_{kj,t}(\zeta))\). Both are smooth with
respect to the real coordinates
\((\Re t,\Im t)\) and the real coordinates on the surface.
The local-coordinate definition of smoothness therefore shows that
\(\boldsymbol\Theta\) is smooth; equivalently, the family
\(\Theta_t\) depends smoothly on \(t\).
The local inverse formulas above also show that
\(\Theta_t^{-1}\) depends smoothly on \(t\).

Finally, \(H_{kj,0}=\id_{\mb C}\), so
\(\mathcal T_0=\id_{X_k^o\sqcup Y_0}\).
It follows that \(\Theta_0=\id_X\).
\end{proof}

More generally, consider the family
\(\varphi_{\mf t}:X_{\mf t}\to\mb P^1\),
\(|\mf t|<\delta\), constructed in the preceding section.
Apply the same construction over each of the pairwise
disjoint discs \(D_1,\ldots,D_N\), using a cutoff adapted
to the corresponding \(\epsilon_i\).
Since only finitely many local pieces are involved,
we may shrink \(\delta\) so that all the required estimates
hold simultaneously.
The resulting local maps then glue to a smooth family of
orientation-preserving diffeomorphisms
\(\Theta_{\mf t}:X\to X_{\mf t}\), with
\(\Theta_{\mf 0}=\id_X\).
Here smoothness is understood with respect to the real
and imaginary parts of all components of \(\mf t\).

\section{Rauch Variational Formulas for Arbitrary Ramification}\label{sec:rauch-formulas}

This section develops the complex-analytic side of the determinant
variation. Starting from normalized differentials and canonical
bidifferentials, we derive their variations in a Hurwitz coordinate,
obtain the corresponding period-matrix formula, and then pass to the
Bergman and Schiffer tau-functions. These formulas will identify the
zero-energy term arising from the spectral calculation.

The variational approach originates in the work of Rauch
\cite{Rauch1959}. Residue formulas on Hurwitz spaces and the associated
tau-functions, including arbitrary ramification indices, are developed
in \cite{KokotovKorotkin2004}. We use their tau-function construction
and variational formulas, and record them below in the Hurwitz
coordinates fixed above, with the residue contributions grouped when
several ramification points lie over the same branch value. For
background on the canonical
bidifferential and projective connections, see \cite{Fay1992}.

Let \(X\) be the compact Riemann surface fixed above, with its
complex orientation, and choose a base point \(P_*\) away from the
ramification points. A homological marking is a symplectic basis
\(\{\alpha_i,\beta_i\}_{i=1}^g\) of \(H_1(X,\mb Z)\), with
\((\alpha_i,\alpha_j)=(\beta_i,\beta_j)=0\) and
\((\alpha_i,\beta_j)=\delta_{ij}\). We choose a canonical system
of based loops representing this basis and shrink the deformation
neighborhoods to avoid these loops. When a fundamental polygon is
used below, its boundary word is
\(\alpha_1\beta_1\alpha_1^{-1}\beta_1^{-1}\cdots
\alpha_g\beta_g\alpha_g^{-1}\beta_g^{-1}\).

The marking determines a unique basis
\(\{\omega_1,\ldots,\omega_g\}\) of \(H^0(X,K_X)\) satisfying the
normalization conditions
\(\int_{\alpha_j}\omega_i=\delta_{ij}\) for \(1\leq i,j\leq g\).
The period matrix of \(X\) with respect to the fixed homological
marking is the \(g\times g\) complex matrix
\(B=(B_{ij})_{i,j=1}^g\), where
\(B_{ij}=\int_{\beta_j}\omega_i\).
For \(g\geq1\), the Riemann bilinear relations give
\(B\in\mathfrak H_g:=\{Z\in M_g(\mb C):Z=Z^T,\ \operatorname{Im}Z>0\}\),
the Siegel upper half-space.
When \(g=0\), the homological marking and normalized basis are empty,
and \(B\) is the empty matrix. Throughout the paper, we take
\(\mathfrak H_0\) to be a point, set \(\det\operatorname{Im}B=1\),
and interpret all sums indexed by \(1\leq i,j\leq g\) as empty.

Let \(\pi_1,\pi_2:X\times X\to X\) be the projections and let
\(\Delta=\{(P,P):P\in X\}\) be the diagonal. We write
\[
(K_X\boxtimes K_X)(2\Delta)
=
\pi_1^*K_X\otimes_{\mc O_{X\times X}}
\pi_2^*K_X\otimes_{\mc O_{X\times X}}
\mc O_{X\times X}(2\Delta).
\]
The canonical meromorphic bidifferential associated with the fixed
homological marking is the unique section
\(W\in H^0(X\times X,(K_X\boxtimes K_X)(2\Delta))\)
satisfying the following properties:
\begin{enumerate}
\item
\(W\) is symmetric, that is, \(W(P,Q)=W(Q,P)\).

\item
For every \(P\in X\),
\(\int_{\alpha_j}W(P,\cdot)=0\) for \(1\le j\le g\).
By symmetry, equivalently,
\(\int_{\alpha_j}W(\cdot,P)=0\).

\item
In a local coordinate \(z\) near \(P\), as \(Q\to P\),
\[
W(P,Q)
=
\left(
\frac{1}{(z(P)-z(Q))^2}
+
\mathcal H(z(P),z(Q))
\right)
dz(P)\otimes dz(Q),
\]
where \(\mathcal H\) is holomorphic near the diagonal and
\(\mathcal H(z,z)=S_B(z)/6\).
Here \(S_B\) is the Bergman projective connection associated with
the fixed homological marking.
\end{enumerate}
Periods of meromorphic differentials are taken over representatives
avoiding their poles. Since the poles considered here have zero residue,
these periods depend only on the homology classes.
Moreover, by the Riemann bilinear relations,
\(\int_{\beta_i}W(P,\cdot)
=
2\pi\sqrt{-1}\,\omega_i(P)\).
Equivalently,
\(\omega_i(P)
=
(2\pi\sqrt{-1})^{-1}
\int_{\beta_i}W(P,\cdot)\).
The Schiffer bidifferential
\(\mc S\in H^0(X\times X,(K_X\boxtimes K_X)(2\Delta))\)
is defined by
\[
\mc S(P,Q)
=
W(P,Q)
-
\pi
\sum_{i,j=1}^{g}
\bigl((\operatorname{Im}B)^{-1}\bigr)_{ij}
\omega_i(P)\otimes\omega_j(Q).
\]
The bidifferential \(\mc S\) is independent of the choice of
homological marking; see \cite[Section~4]{HillairetKalvinKokotov2018}.
In a local coordinate \(z\) near \(P\), the Schiffer bidifferential
has the expansion
\[
\mc S(P,Q)
=
\left(
\frac{1}{(z(P)-z(Q))^2}
+
H_{\mc S}(z(P),z(Q))
\right)
dz(P)\otimes dz(Q),
\]
where \(H_{\mc S}\) is holomorphic near the diagonal and
\(H_{\mc S}(z,z)=S_{\mathrm{Sch}}(z)/6\).
Here \(S_{\mathrm{Sch}}\) is the Schiffer projective connection.
Writing \(\omega_i(z)/dz\) for the coefficient of \(\omega_i\)
in the coordinate \(z\), we obtain
\[
S_{\mathrm{Sch}}(z)
=
S_B(z)
-
6\pi
\sum_{i,j=1}^{g}
\bigl((\operatorname{Im}B)^{-1}\bigr)_{ij}
\frac{\omega_i(z)\omega_j(z)}{(dz)^2}.
\]
Let \(\varphi:X\to\mb P^1\) be a meromorphic function.
Write \(\operatorname{Br}(\varphi)=\{Q_1,\ldots,Q_N\}\). Fix
\(1\leq k\leq N\), and write
\(\varphi^{-1}(Q_k)=\{P_{kj}:1\leq j\leq s_k\}\).
Let \(\varphi_t:X_t\to\mb P^1\), with \(\varphi_0=\varphi\), be the
deformation of \(\varphi\) corresponding to the branch value \(Q_k\)
constructed above, and let
\(\Theta_t:X\to X_t\), \(|t|<\delta\), be the corresponding smooth
family of orientation-preserving diffeomorphisms, with
\(\Theta_0=\id_X\).
Transport the fixed homological marking of \(X\) by setting
\(\alpha_i(t)=(\Theta_t)_*\alpha_i\) and
\(\beta_i(t)=(\Theta_t)_*\beta_i\).
These cycles form a homological marking of \(X_t\).
Let \(\{\omega_i^{X_t}\}_{i=1}^g\) be the unique basis of
\(H^0(X_t,K_{X_t})\) satisfying
\(\int_{\alpha_j(t)}\omega_i^{X_t}=\delta_{ij}\) for
\(1\leq i,j\leq g\). At \(t=0\), this is the fixed normalized
basis \(\{\omega_i\}_{i=1}^g\).
We justify the smooth dependence needed to differentiate these
forms. Use \(\Theta_t\) to transport the complex structures to
the fixed smooth surface \(X\), and choose a smooth family of
positive Riemannian metrics on all of \(X\) compatible with them.
For these auxiliary metrics, the Hodge Laplacians on real one-forms
are a smooth elliptic family, and their kernels have constant
dimension \(2g\). The spectral projections onto these kernels
depend smoothly on \(t\): locally they are given by a resolvent
integral around zero, with the positive spectrum separated from
the contour. Interior elliptic regularity on the compact surface
gives this dependence in every Sobolev norm.

Take the harmonic representatives \(a_i(t)\) of the fixed real
cohomology classes \([\operatorname{Re}\omega_i]\). They are
obtained by applying these projections to fixed closed
representatives, and hence depend smoothly on \(t\). If
\(*_t\) is the Hodge star of the auxiliary metric, then
\(a_i(t)+\sqrt{-1}\,*_ta_i(t)\) is a holomorphic one-form for
the transported complex structure. At \(t=0\), these forms are
\(\omega_i\), so they remain a basis for small \(t\). Their
\(\alpha\)-period matrix is smooth and equals the identity at
zero. Inverting this matrix gives the normalized basis, proving
that \(\Theta_t^*\omega_i^{X_t}\) depends smoothly on the real
coordinates of the family. The same argument applies to the
simultaneous Hurwitz deformation. All complex coordinate
derivatives below are understood in the Wirtinger sense.

Define the period matrix
\(B(t)=(B_{ij}(t))_{1\leq i,j\leq g}\) by
\(B_{ij}(t)=\int_{\beta_j(t)}\omega_i^{X_t}\). 
We now compute the variation
\(\left. \frac{\partial}{\partial t}B_{ij}(t) \right|_{t=0}\).

Set \(\widetilde\omega_i(t):=\Theta_t^*\omega_i^{X_t}\). Since
\(\beta_j(t)=(\Theta_t)_*\beta_j\), we have
\(B_{ij}(t)=\int_{\beta_j(t)}\omega_i^{X_t}
=\int_{\beta_j}\widetilde\omega_i(t)\).
Here \(\beta_j\) is a fixed smooth curve contained in the fixed
exterior, and \(\widetilde\omega_i(t)\) depends smoothly on the real
parameters. Differentiation may therefore be passed through the
integral over this fixed curve. Hence
\[
\left.
\frac{\partial}{\partial t}B_{ij}(t)
\right|_{t=0}
=
\int_{\beta_j}
\left.
\frac{\partial}{\partial t}\widetilde\omega_i(t)
\right|_{t=0}.
\]
To compute the variation, we introduce some notation following the
traditional convention. If \(\omega\) is a one-form on a coordinate
chart \((U,z)\) and \(\omega=f(z)\,dz\), we write
\(\omega(z)/dz:=f(z)\).
Let \(\omega_1,\omega_2,\omega_3\) be meromorphic one-forms on \(X\),
with \(\omega_3\) not identically zero. If
\(\omega_1=a(z)\,dz\), \(\omega_2=b(z)\,dz\), and
\(\omega_3=h(z)\,dz\) in a local coordinate \(z\), we write
\[
\frac{\omega_1\omega_2}{\omega_3}
:=
\frac{a(z)b(z)}{h(z)}\,dz.
\]
This expression is independent of the choice of local coordinate and
defines a meromorphic one-form on \(X\).
Similarly, let \(\mathcal K\) be a bidifferential on \(X\times X\), and let
\((U,z)\) and \((V,w)\) be coordinate charts on \(X\). If
\(\mathcal K=g(z,w)\,dz\otimes dw\) on \(U\times V\), we write
\[
\frac{\mathcal K(z,w)}{dw}:=g(z,w)\,dz,
\qquad
\frac{\mathcal K(z,w)}{dz}:=g(z,w)\,dw.
\]
Thus, division by the differential in one variable leaves a one-form
in the other variable.

Choose a local chart \((U_{kj}(t),z_{kj,t})\) centered at
\(P_{kj}(t)\), as before. By the choice of the target coordinate
\(\xi_k\), we have
\(\xi_k\circ\varphi_t=z_k+t+z_{kj,t}^{n_{kj}}\) on
\(U_{kj}(t)\).

Let \(W_t\) be the canonical meromorphic bidifferential on \(X_t\)
associated with the transported homological marking. Then, for every
\(P\in U_{kj}(t)\),
\[
\omega_i^{X_t}(P)
=
\frac{1}{2\pi\sqrt{-1}}
\int_{\beta_i(t)}W_t(P,\cdot).
\]
For each \(t\), write
\(\omega_i^{X_t}
=
f_{i,kj,t}(z_{kj,t})\,dz_{kj,t}\)
on \(U_{kj}(t)\). Thus,
\[
f_{i,kj,t}(z_{kj,t})
=
\frac{\omega_i^{X_t}(z_{kj,t})}{dz_{kj,t}}
=
\frac{1}{2\pi\sqrt{-1}}
\frac{1}{dz_{kj,t}}
\int_{\beta_i(t)}W_t(z_{kj,t},\cdot).
\]
Recall that, on \(\Omega_{kj}(0)\), the map \(\Theta_t\) is represented
in the local coordinates \(z_{kj}\) and \(z_{kj,t}\) by \(H_{kj,t}\);
that is,
\(
z_{kj,t}\circ\Theta_t\circ z_{kj}^{-1}=H_{kj,t}.
\)
Consequently,
\[
\widetilde\omega_i(t)
=
\Theta_t^*\omega_i^{X_t}
=
f_{i,kj,t}\bigl(H_{kj,t}(z_{kj})\bigr)\,
dH_{kj,t}(z_{kj})
\]
on \(U_{kj}\). Since \(H_{kj,t}\) is, in general, not
holomorphic with respect to \(z_{kj}\), we may write
\[
\widetilde\omega_i(t)
=
f_{i,kj,t}\bigl(H_{kj,t}(z_{kj})\bigr)
\left(
\frac{\partial H_{kj,t}}{\partial z_{kj}}\,dz_{kj}
+
\frac{\partial H_{kj,t}}{\partial\overline z_{kj}}\,
d\overline z_{kj}
\right).
\]
We now compute the variation locally. Fix \(1\leq \ell\leq s_k\), and
write
\(z=z_{k\ell}\), \(n=n_{k\ell}\), \(f_t=f_{i,k\ell,t}\), and \(H_t=H_{k\ell,t}\).
For \(z\neq0\), the definition of \(H_t\) gives
\(
H_t(z)
=
z\left(
1-\beta(|z|^n)\frac{t}{z^n}
\right)^{1/n}.
\)
At \(z=0\), we have \(H_t(0)=0\).
Set
\(v_{k\ell}(z) := \left. \frac{\partial H_t(z)}{\partial t} \right|_{t=0}\).
Then, for \(z\neq0\),
\(v_{k\ell}(z) = -\frac{\beta(|z|^n)}{nz^{n-1}}\),
and we set \(v_{k\ell}(0)=0\). Since
\(\beta(|z|^n)=0\) for \(|z|^n\leq\epsilon_k/2\),
\(v_{k\ell}\) is smooth on \(\Omega_{k\ell}(0)\).

Since \(H_0=\id\), we have
\(\partial H_0/\partial z=1\),
\(\partial H_0/\partial\bar z=0\), and
\(f_0(z)\,dz=\omega_i\). Differentiating with respect to \(t\) at
\(t=0\), we obtain
\begin{align*}
\left.
\frac{\partial}{\partial t}\widetilde\omega_i(t)
\right|_{t=0}
={}&
\left.
\frac{\partial f_t}{\partial t}
\right|_{t=0}dz
+
\frac{\partial f_0}{\partial z}\,
v_{k\ell}(z)\,dz \\
&+
f_0(z)\frac{\partial v_{k\ell}}{\partial z}\,dz
+
f_0(z)\frac{\partial v_{k\ell}}{\partial\bar z}\,d\bar z.
\end{align*}
Since \(f_0\) is holomorphic, the variation, regarded as a one-form on
\(X\), has the local expression
\[
\left.
\frac{\partial}{\partial t}\widetilde\omega_i(t)
\right|_{t=0}
=
\left.
\frac{\partial f_t}{\partial t}
\right|_{t=0}dz
+
d\bigl(f_0v_{k\ell}\bigr).
\]
Since \(f_0(z)=\omega_i(z)/dz\), the variation can be written as
\[
\left.
\frac{\partial}{\partial t}\widetilde\omega_i(t)
\right|_{t=0}
=
\left.
\frac{\partial f_t}{\partial t}
\right|_{t=0}dz
+
d\left(
v_{k\ell}(z)\frac{\omega_i(z)}{dz}
\right).
\]
Set \(
\eta_i
:=
\left.
\frac{\partial}{\partial t}\widetilde\omega_i(t)
\right|_{t=0}.
\)
Since
\(\int_{\alpha_j}\Theta_t^*\omega_i^{X_t}=\delta_{ij}\), we have
\[
\int_{\alpha_j}\eta_i
=
\left.
\frac{\partial}{\partial t}
\int_{\alpha_j}\Theta_t^*\omega_i^{X_t}
\right|_{t=0}
=
0.
\]
Since
\(d\widetilde{\omega}_i(t)
=
\Theta_t^*d\omega_i^{X_t}=0\),
we have \(d\eta_i=0\). Applying the Riemann bilinear relation,
\[
\int_X\eta_i\wedge\omega_j
=
\sum_{m=1}^g
\left(
\int_{\alpha_m}\eta_i\int_{\beta_m}\omega_j
-
\int_{\beta_m}\eta_i\int_{\alpha_m}\omega_j
\right),
\]
and using
\(\int_{\alpha_m}\eta_i=0\) and
\(\int_{\alpha_m}\omega_j=\delta_{jm}\), we obtain
\(\int_X\eta_i\wedge\omega_j = -\int_{\beta_j}\eta_i\).
On each \(U_{k\ell}\), using the coordinate \(z=z_{k\ell}\), we have
\[
\eta_i\wedge\omega_j
=
-\frac{\partial v_{k\ell}}{\partial\bar z}\,
\frac{\omega_i(z)\omega_j(z)}{(dz)^2}\,
dz\wedge d\bar z.
\]
Moreover, since \(\omega_j\) is of type \((1,0)\), only the
\((0,1)\)-part of \(\eta_i\) contributes to
\(\eta_i\wedge\omega_j\).
Since \(\Theta_t\) is the identity on the fixed exterior \(X_k^o\),
this part is supported in the union of the \(U_{k\ell}\).
In the following integrals over \(\Omega_{k\ell}(0)\), the forms are
expressed in the coordinate \(z=z_{k\ell}\). Since
\[
\left.
\frac{\partial}{\partial t}B_{ij}(t)
\right|_{t=0}
=
\int_{\beta_j}\eta_i
=
-\int_X\eta_i\wedge\omega_j,
\]
we obtain
\[
\left.
\frac{\partial}{\partial t}B_{ij}(t)
\right|_{t=0}
=
\sum_{\ell=1}^{s_k}
\int_{\Omega_{k\ell}(0)}
\frac{\partial v_{k\ell}}{\partial\bar z}\,
\frac{\omega_i(z)\omega_j(z)}{(dz)^2}\,
dz\wedge d\bar z.
\]
Since \(\beta(|z_{k\ell}|^{n_{k\ell}})=1\) for
\(|z_{k\ell}|^{n_{k\ell}}\geq 3\epsilon_k/4\), we have
\(\partial_{\overline z_{k\ell}}v_{k\ell}=0\) on this region. Thus,
\[
\int_{\Omega_{k\ell}(0)}
\frac{\partial v_{k\ell}}{\partial\overline z_{k\ell}}\,
\frac{\omega_i(z_{k\ell})\omega_j(z_{k\ell})}
{(dz_{k\ell})^2}\,
dz_{k\ell}\wedge d\overline z_{k\ell}
=
\int_{|z_{k\ell}|^{n_{k\ell}}\leq 3\epsilon_k/4}
\frac{\partial v_{k\ell}}{\partial\overline z_{k\ell}}\,
\frac{\omega_i(z_{k\ell})\omega_j(z_{k\ell})}
{(dz_{k\ell})^2}\,
dz_{k\ell}\wedge d\overline z_{k\ell}.
\]
Since \(\omega_i\) and \(\omega_j\) are holomorphic,
\[
d\left(
v_{k\ell}
\frac{\omega_i(z_{k\ell})\omega_j(z_{k\ell})}
{(dz_{k\ell})^2}\,dz_{k\ell}
\right)
=
-
\frac{\partial v_{k\ell}}{\partial\overline z_{k\ell}}\,
\frac{\omega_i(z_{k\ell})\omega_j(z_{k\ell})}
{(dz_{k\ell})^2}\,
dz_{k\ell}\wedge d\overline z_{k\ell}.
\]
Hence, by Stokes' theorem,
\[
\int_{|z_{k\ell}|^{n_{k\ell}}\leq 3\epsilon_k/4}
\frac{\partial v_{k\ell}}{\partial\overline z_{k\ell}}\,
\frac{\omega_i(z_{k\ell})\omega_j(z_{k\ell})}
{(dz_{k\ell})^2}\,
dz_{k\ell}\wedge d\overline z_{k\ell}
=
-
\int_{|z_{k\ell}|^{n_{k\ell}}=3\epsilon_k/4}
v_{k\ell}
\frac{\omega_i(z_{k\ell})\omega_j(z_{k\ell})}
{(dz_{k\ell})^2}\,dz_{k\ell}.
\]
On
\(|z_{k\ell}|^{n_{k\ell}}=3\epsilon_k/4\), we have
\(v_{k\ell} = -\frac{1}{n_{k\ell}z_{k\ell}^{n_{k\ell}-1}}\).
Therefore,
\[
\begin{aligned}
&\int_{|z_{k\ell}|^{n_{k\ell}}\leq 3\epsilon_k/4}
\frac{\partial v_{k\ell}}{\partial\overline z_{k\ell}}\,
\frac{\omega_i(z_{k\ell})\omega_j(z_{k\ell})}
{(dz_{k\ell})^2}\,
dz_{k\ell}\wedge d\overline z_{k\ell}
\\
&\qquad=
\int_{|z_{k\ell}|^{n_{k\ell}}=3\epsilon_k/4}
\frac{1}{n_{k\ell}z_{k\ell}^{n_{k\ell}-1}}
\frac{\omega_i(z_{k\ell})\omega_j(z_{k\ell})}
{(dz_{k\ell})^2}\,dz_{k\ell}.
\end{aligned}
\]
Recall that, near \(P_{k\ell}\), the local coordinate \(z_{k\ell}\)
is chosen so that
\[
w_k\circ\varphi
=
z_{k\ell}^{n_{k\ell}},
\qquad
z_{k\ell}(P_{k\ell})=0.
\]
Hence,
\(d_k\varphi=d(w_k\circ\varphi)
=
n_{k\ell}z_{k\ell}^{n_{k\ell}-1}dz_{k\ell}\)
in this local coordinate. Therefore,
\[
\frac{1}{n_{k\ell}z_{k\ell}^{n_{k\ell}-1}}
\frac{\omega_i(z_{k\ell})\omega_j(z_{k\ell})}
{(dz_{k\ell})^2}\,dz_{k\ell}
=
\frac{\omega_i\omega_j}{d_k\varphi}.
\]
Consequently, since the curve
\(|z_{k\ell}|^{n_{k\ell}}=3\epsilon_k/4\) is positively oriented
around \(P_{k\ell}\),
\[
\int_{|z_{k\ell}|^{n_{k\ell}}=3\epsilon_k/4}
\frac{1}{n_{k\ell}z_{k\ell}^{n_{k\ell}-1}}
\frac{\omega_i(z_{k\ell})\omega_j(z_{k\ell})}
{(dz_{k\ell})^2}\,dz_{k\ell}
=
2\pi\sqrt{-1}\,
\operatorname{Res}_{P_{k\ell}}
\frac{\omega_i\omega_j}{d_k\varphi}.
\]
Thus, we obtain the following variation formula.

\begin{thm}
Let \(\varphi:X\to\mb P^1\) be a meromorphic function.
Write \(\operatorname{Br}(\varphi)=\{Q_1,\ldots,Q_N\}\). Fix
\(1\leq k\leq N\), and write
\[
\varphi^{-1}(Q_k)
=
\{P_{k1},\ldots,P_{ks_k}\},
\qquad
n_{k\ell}=e_{P_{k\ell}}(\varphi).
\]
Let \(\varphi_t:X_t\to\mb P^1\), with \(\varphi_0=\varphi\), be the
deformation corresponding to the branch value \(Q_k\) constructed
above, and let \(\Theta_t:X\to X_t\) be the corresponding smooth
family of orientation-preserving diffeomorphisms with
\(\Theta_0=\id_X\).
Fix a homological marking
\(\{\alpha_1,\ldots,\alpha_g,\beta_1,\ldots,\beta_g\}\) of \(X\), and
define
\(\alpha_m(t)=(\Theta_t)_*\alpha_m\) and
\(\beta_m(t)=(\Theta_t)_*\beta_m\).
Let \(\{\omega_i^{X_t}\}_{i=1}^g\) be the normalized basis of
\(H^0(X_t,K_{X_t})\) satisfying
\(\int_{\alpha_m(t)}\omega_i^{X_t}=\delta_{im}\), and define
\(B(t)=(B_{ij}(t))\) by
\(B_{ij}(t)=\int_{\beta_j(t)}\omega_i^{X_t}\).
At the base point, use the normalized differentials \(\omega_i\).
Then, for \(1\leq i,j\leq g\),
\[
\left.
\frac{\partial}{\partial t}B_{ij}(t)
\right|_{t=0}
=
2\pi\sqrt{-1}
\sum_{\substack{1\leq\ell\leq s_k\\ n_{k\ell}>1}}
\operatorname{Res}_{P_{k\ell}}
\frac{\omega_i\omega_j}{d_k\varphi}.
\]
\end{thm}

In general, let
\([\varphi]\in H_{g,d,N}^{\mb P^1}\), and choose a representative
\(\varphi:X\to\mb P^1\). Write
\(\mf z=\Phi_{[\varphi];D_1,\ldots,D_N}([\varphi])\), and fix a
homological marking
\(\{\alpha_1,\ldots,\alpha_g,\beta_1,\ldots,\beta_g\}\) of \(X\).
Choose \(\delta>0\) sufficiently small that \(B_{\mf z}(\delta)\)
is contained in the image of the Hurwitz chart and all the local
constructions above are defined. Set
\[
\mc B([\varphi],\delta)
:=
\Phi_{[\varphi];D_1,\ldots,D_N}^{-1}
\bigl(B_{\mf z}(\delta)\bigr),
\]
and write
\(\Phi:=\Phi_{[\varphi];D_1,\ldots,D_N}
|_{\mc B([\varphi],\delta)}\).

For \(\mf t\in\mb C^N\) with
\(\mf z+\mf t\in B_{\mf z}(\delta)\), the gluing construction above
provides a representative
\(\varphi_{\mf t}:X_{\mf t}\to\mb P^1\) of the unique Hurwitz point
\([\varphi_{\mf t}]
=\Phi_{[\varphi];D_1,\ldots,D_N}^{-1}(\mf z+\mf t)\),
together with a smooth family of orientation-preserving
diffeomorphisms
\(\Theta_{\mf t}:X\to X_{\mf t}\), with
\(\Theta_{\mf 0}=\id_X\).
The fixed homological marking of \(X\) is transported to \(X_{\mf t}\)
by setting
\(\alpha_m(\mf t)=(\Theta_{\mf t})_*\alpha_m\) and
\(\beta_m(\mf t)=(\Theta_{\mf t})_*\beta_m\).

Let \(\{\omega_i^{X_{\mf t}}\}_{i=1}^g\) be the normalized basis of
\(H^0(X_{\mf t},K_{X_{\mf t}})\) satisfying
\(\int_{\alpha_m(\mf t)}\omega_i^{X_{\mf t}}=\delta_{im}\), and let
\(B(\mf t)=(B_{ij}(\mf t))\) be the corresponding period matrix,
where \(B_{ij}(\mf t)
=\int_{\beta_j(\mf t)}\omega_i^{X_{\mf t}}\).
Thus the family of period matrices defines the local period map
\(B:B_{\mf z}(\delta)\to\mathfrak H_g\),
\(\mf z+\mf t\mapsto B(\mf t)\).
Equivalently, through the local coordinate
\(\Phi_{[\varphi];D_1,\ldots,D_N}\), we regard \(B\) as a locally
defined \(\mathfrak H_g\)-valued function on
\(\mc B([\varphi],\delta)\), with the genus-zero conventions stated above.

\begin{thm}
Let \([\varphi]\in H_{g,d,N}^{\mb P^1}\), and choose a
representative \(\varphi:X\to\mb P^1\). Fix a homological marking
\(\{\alpha_1,\ldots,\alpha_g,\beta_1,\ldots,\beta_g\}\) of \(X\).
Choose \(\delta>0\) as above, and let
\(\mc B([\varphi],\delta)\) be the corresponding coordinate
neighborhood of \([\varphi]\).
For each \([\varphi_{\mf t}]\in\mc B([\varphi],\delta)\), choose the
representative \(\varphi_{\mf t}:X_{\mf t}\to\mb P^1\) provided by the
gluing construction above, and transport the fixed homological marking
of \(X\) to \(X_{\mf t}\) by the corresponding diffeomorphism
\(\Theta_{\mf t}:X\to X_{\mf t}\). Let \(B(\mf t)\) be the period
matrix of \(X_{\mf t}\) with respect to this transported marking.
Then the assignment
\([\varphi_{\mf t}]\mapsto B(\mf t)\) defines a holomorphic local
function
\(B:\mc B([\varphi],\delta)\to\mathfrak H_g\).
Equivalently, if
\(\mf z=\Phi_{[\varphi];D_1,\ldots,D_N}([\varphi])\), then in the
local coordinates \(\Phi_{[\varphi];D_1,\ldots,D_N}\), this function
is represented by \(\mf z+\mf t\mapsto B(\mf t)\).

Moreover, if
\(\varphi^{-1}(Q_k)=\{P_{k1},\ldots,P_{ks_k}\}\) and
\(n_{k\ell}=e_{P_{k\ell}}(\varphi)\), then, for
\(1\leq i,j\leq g\) and \(1\leq k\leq N\),
\[
\left.
\frac{\partial B_{ij}}{\partial z_k}
\right|_{[\varphi]}
=
2\pi\sqrt{-1}
\sum_{\substack{1\leq\ell\leq s_k\\ n_{k\ell}>1}}
\operatorname{Res}_{P_{k\ell}}
\frac{\omega_i\omega_j}{d_k\varphi},
\]
where \(\{\omega_1,\ldots,\omega_g\}\) is the normalized basis of
\(H^0(X,K_X)\) associated with the fixed homological marking.
\end{thm}

\begin{proof}
The normalized differentials, and hence the period matrix, depend
smoothly on the real coordinates. To prove holomorphicity, fix a base
point of the Hurwitz chart and consider the deformation in the
\(z_k\)-direction based there. Write
\(\eta_i^{\bar t}:=
\left.\partial_{\bar t}(\Theta_t^*\omega_i^{X_t})\right|_{t=0}\).
For fixed \(z\), the map \(t\mapsto H_{kj,t}(z)\) is holomorphic.
Consequently, on \(U_{kj}\),
\[
\eta_i^{\bar t}
=
\left.\frac{\partial f_{i,kj,t}}{\partial\bar t}\right|_{t=0}
(z_{kj})\,dz_{kj}.
\]
This is holomorphic in \(z_{kj}\). On the fixed exterior piece,
\(\eta_i^{\bar t}\) is also holomorphic. Thus it is a holomorphic
one-form on \(X\). Differentiating the normalization conditions
shows that all its \(\alpha\)-periods vanish, so
\(\eta_i^{\bar t}=0\). Hence
\(\left.\partial_{\bar t}B_{ij}(t)\right|_{t=0}=0\).
Since the base point and \(k\) were arbitrary, the period map is
holomorphic. The stated derivative formula follows from the
preceding theorem.
\end{proof}

\begin{cor}
Under the same assumptions, one has
\[
\frac{\partial}{\partial z_k}\log\det\operatorname{Im}B
=
\pi
\sum_{\substack{1\leq\ell\leq s_k\\ n_{k\ell}>1}}
\operatorname{Res}_{P_{k\ell}}
\frac{
\displaystyle
\sum_{i,j=1}^g
\bigl((\operatorname{Im}B)^{-1}\bigr)_{ij}
\omega_i\omega_j
}
{d_k\varphi}.
\]
\end{cor}
\begin{proof}
Let \(C=\operatorname{Im}B\). Then
\(C=(B-\overline B)/(2\sqrt{-1})\). Since \(B\) is holomorphic,
\(\frac{\partial C}{\partial z_k} = \frac{1}{2\sqrt{-1}} \frac{\partial B}{\partial z_k}\).
By Jacobi's formula,
\(\partial_{z_k}\log\det C
=
\Tr(C^{-1}\partial_{z_k}C)\).
Using the variation formula above, we obtain
\[
\frac{\partial C}{\partial z_k}
=
\pi
\sum_{\substack{1\leq\ell\leq s_k\\ n_{k\ell}>1}}
\left[
\operatorname{Res}_{P_{k\ell}}
\frac{\omega_i\omega_j}{d_k\varphi}
\right]_{i,j=1}^g.
\]
Writing \(C^{-1}=[C^{ij}]\), and using the symmetry of \(C^{-1}\),
we obtain
\[
\Tr\left(C^{-1}\frac{\partial C}{\partial z_k}\right)
=
\pi
\sum_{\substack{1\leq\ell\leq s_k\\ n_{k\ell}>1}}
\operatorname{Res}_{P_{k\ell}}
\frac{\displaystyle\sum_{i,j=1}^g
C^{ij}\omega_i\omega_j}{d_k\varphi}.
\]
This proves the result.
\end{proof}

\begin{lem}
Let \(\eta\) be a meromorphic one-form on \(X\) with vanishing
\(\alpha\)-periods,
\(\int_{\alpha_j}\eta=0\) for \(1\leq j\leq g\),
and assume that
\(\operatorname{Res}_{Q_0}\eta=0\) for every
\(Q_0\in\operatorname{Pole}(\eta)\).
Assume that \(P_*\notin\operatorname{Pole}(\eta)\), and choose cycle
representatives avoiding the poles. Then, for \(P\notin\operatorname{Pole}(\eta)\),
\[
\eta(P)
=
-
\sum_{Q_0\in\operatorname{Pole}(\eta)}
\operatorname{Res}_{Q=Q_0}
\left[
\left(\int_{P_*}^{Q}\eta\right)W(P,Q)
\right],
\]
where the residues are taken with respect to the \(Q\)-variable.
Near each pole, the integral denotes a local branch of a primitive.
Such a primitive is meromorphic because the residue of \(\eta\)
vanishes. Different branches differ by a constant, which does not
affect the residue since \(W(P,\cdot)\) is holomorphic at that pole
for \(P\notin\operatorname{Pole}(\eta)\).
\end{lem}
\begin{proof}
Define
\[
\widetilde\eta(P)
=
-
\sum_{Q_0\in\operatorname{Pole}(\eta)}
\Res_{Q=Q_0}
\left[
F(Q)W(P,Q)
\right],
\qquad
F(Q)=\int_{P_*}^{Q}\eta.
\]
We prove the identity in two steps. First, we compare the principal
parts of \(\widetilde\eta\) and \(\eta\) at every pole. Their
difference is then holomorphic, and the normalization of \(W\) shows
that all of its \(\alpha\)-periods vanish.

Let \(Q_0\) be a pole of \(\eta\), and choose a local chart
\((U_0,w)\) centered at \(Q_0\) so that \(Q_0\) is the only pole of
\(\eta\) on \(U_0\). For \(P\in U_0'=U_0\setminus\{Q_0\}\), write
\(z=w(P)\). By the definition of \(\widetilde\eta\),
\[
\widetilde\eta(z)
=
-\Res_{w=0}F(w)W(z,w)
+
\text{a holomorphic differential in \(z\)}.
\]
Indeed, for every \(Q_1\in\operatorname{Pole}(\eta)\setminus\{Q_0\}\),
the differential
\(-\Res_{Q=Q_1}F(Q)W(P,Q)\) is holomorphic in \(P\) on \(U_0\).

Write
\(\eta=(\sum_{k=-m}^{\infty}a_kw^k)dw\) on
\(U_0'=U_0\setminus\{Q_0\}\). Since
\(\Res_{Q_0}\eta=0\), we have \(a_{-1}=0\), and hence
\[
F(w)
=
C+
\sum_{\substack{k=-m\\k\neq-1}}^{\infty}
\frac{a_k}{k+1}w^{k+1}
\]
on \(U_0'\).
Fix \(P\in U_0'\) and write \(z=w(P)\). Since \(z\neq0\), the
meromorphic one-form
\[
\frac{F(w)W(z,w)}{dz}
=
\left(
\frac{F(w)}{(z-w)^2}
+
F(w)\mathcal H(z,w)
\right)dw
\]
in the \(w\)-variable has possible poles only at \(w=0\) and
\(w=z\). The possible pole at \(w=0\) comes from \(F(w)\), while the
one at \(w=z\) comes from the diagonal singularity of \(W(z,w)\).

Because \(F\) has a finite principal part and \(\mathcal H\) is
jointly holomorphic, the residue of
\(F(w)\mathcal H(z,w)\,dw\) at \(w=0\) is holomorphic in
\(z\). Thus only the diagonal singularity of
\(W\) contributes to the principal part. Expanding
\((z-w)^{-2}\) in powers of \(w\), for \(0<|w|<|z|\), gives
\[
\Res_{w=0}\frac{F(w)}{(z-w)^2}\,dw
=
-\sum_{k=-m}^{-2}a_kz^k.
\]
It follows that \(-\Res_{w=0}F(w)W(z,w)\), and hence
\(\widetilde\eta\), has the same principal part at \(Q_0\) as
\(\eta\). Since \(Q_0\) was arbitrary,
\(\widetilde\eta-\eta\) is holomorphic on \(X\).

It remains to use the normalization. By the vanishing
\(\alpha\)-periods of \(W\),
\[
\int_{\alpha_i}\widetilde\eta
=
-
\sum_{Q_0\in\operatorname{Pole}(\eta)}
\Res_{Q=Q_0}
\left[
F(Q)\int_{\alpha_i}W(\cdot,Q)
\right]
=
0.
\]
Thus \(\widetilde\eta-\eta\) is holomorphic with vanishing
\(\alpha\)-periods, so it is zero.
\end{proof}

Let \(W_t\) be the canonical meromorphic bidifferential on \(X_t\)
associated with the homological marking
\(\{\alpha_i(t),\beta_i(t)\}_{i=1}^g\).
Choose \(P_*\) outside the fiber \(\varphi^{-1}(Q_k)\), and choose
representatives of \(\alpha_i,\beta_i\), based at \(P_*\), disjoint
from this entire fiber. Shrink the deformation neighborhoods so that
their closures are disjoint from these representatives.
Thus the cycles are contained in \(X_k^o\), and, under the
canonical identification of \(X_k^o\) with its image in
\(X_t\),
\(\alpha_i(t)=\alpha_i\) and \(\beta_i(t)=\beta_i\).

Choose \(Q\in X_k^o\), and choose a local chart
\((U,u)\) centered at \(Q\) such that
\(U\Subset X_k^o\) and \(u(Q)=0\).
By the construction of \(\Theta_t\),
\(\Theta_t(Q)=Q\) for all \(|t|<\delta\).
On \(U_{kj}(t)\times U\), write
\[
W_t
=
g_{kj,t}(z_{kj,t},u)\,
dz_{kj,t}\otimes du.
\]
The coefficient \(g_{kj,t}\) is meromorphic on this product,
with possible poles only along the diagonal.
Since \(Q\in X_k^o\), it lies outside the inner coordinate disc
\(|z_{kj,t}|^{n_{kj}}\leq\epsilon_k/2\).
Indeed, on this disc,
\(|t+z_{kj,t}^{n_{kj}}|<3\epsilon_k/4\), whereas the overlap
with \(X_k^o\) satisfies \(|t+z_{kj,t}^{n_{kj}}|>3\epsilon_k/4\).
Thus \(g_{kj,t}(z,0)\) is holomorphic for
\(|z|^{n_{kj}}<\epsilon_k/2\), and in particular at \(z=0\).

Define
\(\vartheta_t := \left. \frac{W_t(\cdot,u)}{du} \right|_{u=0}\).
Then \(\vartheta_t\) is a meromorphic differential on \(X_t\), with a unique
double pole at \(Q\) and zero residue. Moreover, on \(U_{kj}(t)\),
\(\vartheta_t = g_{kj,t}(z_{kj,t},0)\,dz_{kj,t}\).

On \(U_{kj}(t)\), we have
\(\xi_k\circ\varphi_t=z_k+t+z_{kj,t}^{n_{kj}}\), whereas on
\(U_{kj}\), we have \(\xi_k\circ\varphi=z_k+z_{kj}^{n_{kj}}\).
For \(P\in A_{kj}\), the points \(P\) and
\(\mathscr G_{kj,t}(P)\) represent the same point of \(X_t\).
Under the fixed exterior identification, the coordinate of this point is
\(z_{kj,t}(q_t(P))=\mathscr G_{kj,t}(P)\). The gluing relation gives
\(t+\mathscr G_{kj,t}(P)^{n_{kj}}=z_{kj}(P)^{n_{kj}}\).
In the following calculation, we write
\(z_{kj,t}(P):=z_{kj,t}(q_t(P))\), with \(P\) held fixed in
the exterior piece. Thus
\(z_{kj,t}^{n_{kj}}=z_{kj}^{n_{kj}}-t\) on \(A_{kj}\).

Differentiating with respect to \(t\), we obtain
\(n_{kj}z_{kj,t}^{\,n_{kj}-1}
\frac{\partial z_{kj,t}}{\partial t}=-1\).
Since \(z_{kj,t}=z_{kj}\) at \(t=0\), it follows that, on \(A_{kj}\),
\[
\left.
\frac{\partial z_{kj,t}}{\partial t}
\right|_{t=0}
=
-\frac{1}{n_{kj}z_{kj}^{\,n_{kj}-1}}.
\]
Thus, on \(A_{kj}\setminus\{Q\}\),
\[
\left.
\frac{\partial}{\partial t}\vartheta_t
\right|_{t=0}
=
\left.
\frac{\partial g_{kj,t}}{\partial t}
\right|_{t=0}(z_{kj},0)\,dz_{kj}
-
d\left(
\frac{g_{kj,0}(z_{kj},0)}
{n_{kj}z_{kj}^{\,n_{kj}-1}}
\right).
\]
On \(X_k^o\setminus\{Q\}\), define
\(\eta := \left. \frac{\partial}{\partial t}\vartheta_t \right|_{t=0}\).
Near \(Q\), using the fixed coordinate \(u\), we have
\[
\vartheta_t
=
\left(
\frac{1}{u^2}+h_t(u)
\right)du,
\]
where \(h_t\) is holomorphic. Since the singular term \(du/u^2\) is
independent of \(t\), \(\eta\) extends holomorphically across \(Q\).

The preceding calculation on the overlaps gives a meromorphic
extension of \(\eta\) to each deformation neighborhood.
If \(Q\in U_{kj}\), this extension agrees near \(Q\) with the
holomorphic extension just obtained from the fixed exterior coordinate.
Thus \(\eta\) extends to a meromorphic differential on \(X\). Its possible poles are contained in the points
\(P_{kj}\) with \(n_{kj}>1\), and all its residues vanish.
Furthermore, by the normalization of \(W_t\) and the choice of the
deformation neighborhoods, all \(\alpha\)-periods of \(\eta\) vanish.

We now apply the preceding residue formula. Setting
\(F(R):=\int_{P_*}^{R}\eta\), we obtain
\[
\eta(P)
=
-
\sum_{\substack{1\leq j\leq s_k\\ n_{kj}>1}}
\operatorname{Res}_{R=P_{kj}}
\bigl[F(R)W(P,R)\bigr].
\]
Since, near \(P_{kj}\),
\[
\eta
=
\left.
\frac{\partial g_{kj,t}}{\partial t}
\right|_{t=0}(z_{kj},0)\,dz_{kj}
-
d\left(
\frac{g_{kj,0}(z_{kj},0)}
{n_{kj}z_{kj}^{\,n_{kj}-1}}
\right),
\]
the function
\[
F(R)
+
\frac{g_{kj,0}(z_{kj}(R),0)}
{n_{kj}z_{kj}(R)^{\,n_{kj}-1}}
\]
is holomorphic at \(P_{kj}\). Hence
\[
\operatorname{Res}_{R=P_{kj}}
\bigl[F(R)W(P,R)\bigr]
=
-
\operatorname{Res}_{R=P_{kj}}
\left[
\frac{g_{kj,0}(z_{kj}(R),0)}
{n_{kj}z_{kj}(R)^{\,n_{kj}-1}}
W(P,R)
\right].
\]
Since, on \(U_{kj}\),
\[
\vartheta_0(R)
=
\frac{W(R,Q)}{du(Q)}
=
g_{kj,0}(z_{kj}(R),0)\,dz_{kj}(R),
\]
and
\(d_k\varphi(R) = n_{kj}z_{kj}(R)^{\,n_{kj}-1}\,dz_{kj}(R)\),
we obtain
\[
\eta(P)
=
\frac{1}{du(Q)}
\sum_{\substack{1\leq j\leq s_k\\ n_{kj}>1}}
\operatorname{Res}_{R=P_{kj}}
\frac{W(P,R)W(R,Q)}{d_k\varphi(R)}.
\]
Recalling that, for
\(P\in X_k^o\setminus\{Q\}\),
\[
\eta(P)
=
\left.
\frac{\partial}{\partial t}
\left(
\frac{W_t(P,Q)}{du(Q)}
\right)
\right|_{t=0},
\]
and that \(du(Q)\) is independent of \(t\), we obtain
\[
\left.
\frac{\partial}{\partial t}W_t(P,Q)
\right|_{t=0}
=
\sum_{\substack{1\leq j\leq s_k\\ n_{kj}>1}}
\operatorname{Res}_{R=P_{kj}}
\frac{W(P,R)W(R,Q)}{d_k\varphi(R)}.
\]
Thus the identity holds initially for
\(P,Q\in X_k^o\) with \(P\neq Q\).
Since the singular part of \(W_t(P,Q)\) along the diagonal is
independent of \(t\) in the fixed exterior coordinates, the
left-hand side extends holomorphically across \(P=Q\). The
right-hand side is also holomorphic there. Hence the identity extends
across the diagonal.

Recall that
\(\Phi:\mc B([\varphi],\delta)\to B_{\mf z}(\delta)\)
is the Hurwitz coordinate chart with
\(\Phi([\varphi])=\mf z\). For each
\(\mf z+\mf t\in B_{\mf z}(\delta)\), let
\([\varphi_{\mf t}]=\Phi^{-1}(\mf z+\mf t)\), and let
\(\varphi_{\mf t}:X_{\mf t}\to\mb P^1\) be the representative
constructed above.

Let
\(\pi:\mathscr X\to B_{\mf z}(\delta)\)
be the corresponding local holomorphic family of Riemann surfaces.
Set-theoretically,
\(\mathscr X=\coprod_{\mf z+\mf t\in B_{\mf z}(\delta)}X_{\mf t}\),
and
\(\pi(P)=\mf z+\mf t\) for \(P\in X_{\mf t}\). Thus
\(\pi^{-1}(\mf z+\mf t)=X_{\mf t}\), and in particular
\(\pi^{-1}(\mf z)=X_{\mf 0}=X\).

Consider the fiber product
\(\mathscr X^{[2]}:=\mathscr X\times_{B_{\mf z}(\delta)}\mathscr X\).
Let \(p_1,p_2:\mathscr X^{[2]}\to\mathscr X\) be the projections, and let
\(\pi^{[2]}:=\pi\circ p_1=\pi\circ p_2\) be the projection to
\(B_{\mf z}(\delta)\).
Let
\(
\Delta_\pi
:=
\{(P,P):P\in\mathscr X\}
\subset\mathscr X^{[2]}
\)
be the relative diagonal. Its intersection with the fiber
\(X_{\mf t}\times X_{\mf t}\) is the diagonal
\(\Delta_{\mf t}^{\mathrm{diag}}\subset X_{\mf t}\times X_{\mf t}\).
Denote by \(K_{\mathscr X/B_{\mf z}(\delta)}\) the relative canonical
bundle, and set
\[
\mc L
:=
p_1^*K_{\mathscr X/B_{\mf z}(\delta)}
\otimes
p_2^*K_{\mathscr X/B_{\mf z}(\delta)}
\otimes
\mathcal O(2\Delta_\pi).
\]
Its restriction to the fiber over \(\mf z+\mf t\) is
\[
\mc L|_{X_{\mf t}\times X_{\mf t}}
=
(K_{X_{\mf t}}\boxtimes K_{X_{\mf t}})(2\Delta_{\mf t}^{\mathrm{diag}}),
\]
where \(\Delta_{\mf t}^{\mathrm{diag}}\subset X_{\mf t}\times X_{\mf t}\) is the
diagonal.

The normalized canonical bidifferentials \(W_{\mf t}\) form a
holomorphic section of \(\mc L\); see \cite{Fay1992}.
We denote this family by \(W\). On the fixed exterior,
\(\partial W/\partial z_k\) means differentiation with the two
points held fixed. The preceding calculation gives the following
formula at every point of the Hurwitz coordinate ball.

\begin{thm}\label{thm:rauch-bidifferential}
For the canonical bidifferential associated with the transported
homological marking, the Hurwitz derivative satisfies
\[
\frac{\partial W}{\partial z_k}(P,Q)
=\sum_{\substack{1\leq\ell\leq s_k\\ n_{k\ell}>1}}
\operatorname{Res}_{R=P_{k\ell}}
\frac{W(P,R)W(R,Q)}{d_k\varphi(R)}.
\]
Initially, \(P,Q\in X_k^o\) are held fixed under the exterior
identification. Both sides extend holomorphically across the
diagonal there, and the right-hand side gives their meromorphic
continuation across the deformation neighborhoods in each variable.
\end{thm}

Let
\(R_\varphi=\sum_{P\in\operatorname{Ram}(\varphi)}
(e_P(\varphi)-1)P\)
be the ramification divisor of \(\varphi\), where
\(\operatorname{Ram}(\varphi)=\{P\in X:e_P(\varphi)>1\}\).
Let \(W_{\mb P^1}\) denote the canonical bidifferential on
\(\mb P^1\), and let
\(\iota_\Delta:X\to X\times X\) be the diagonal embedding,
\(\iota_\Delta(P)=(P,P)\). We define
\[
W_\varphi
:=
\iota_\Delta^*
\left(
W-(\varphi\times\varphi)^*W_{\mb P^1}
\right).
\]
Since the principal parts of the two bidifferentials along the
diagonal coincide, their difference is regular along the diagonal
away from the ramification points of \(\varphi\), and its pullback by
\(\iota_\Delta\) defines a meromorphic quadratic differential on
\(X\). Since
\(\iota_\Delta^*(K_X\boxtimes K_X)=K_X^{\otimes2}\), we have
\[
W_\varphi\in
H^0\bigl(X,K_X^{\otimes2}(2R_\varphi)\bigr).
\]

For \(1\leq k\leq N\), define
\[
\mk B_k
:=
-\sum_{\substack{1\leq\ell\leq s_k\\ n_{k\ell}>1}}
\operatorname{Res}_{P_{k\ell}}
\frac{W_\varphi}{d_k\varphi}.
\]
Since \(W_\varphi\) is a meromorphic quadratic differential,
\(W_\varphi/d_k\varphi\) is a meromorphic one-form on \(\varphi^{-1}(D_k)\), and hence
the residues above are well defined. As \([\varphi]\) varies in
\(\mc B([\varphi],\delta)\), each \(\mk B_k\) defines a holomorphic
function on \(\mc B([\varphi],\delta)\).
\begin{lem}
Let \(1\leq k,m\leq N\) with \(k\neq m\). Then, on
\(\bigcup_{\ell=1}^{s_k}U_{k\ell}\),
\[
\frac{\partial W_\varphi(P)}{\partial z_m}
=
\sum_{\substack{1\leq r\leq s_m\\ n_{mr}>1}}
\operatorname{Res}_{R=P_{mr}}
\frac{W(P,R)^2}{d_m\varphi(R)}.
\]
\end{lem}

\begin{proof}
Let \(\varphi_t:X_t\to\mb P^1\) be the deformation in the
\(z_m\)-direction, with \(\varphi_0=\varphi\).
Since \(k\neq m\), the deformation neighborhoods associated with
\(\varphi^{-1}(Q_m)\) are disjoint from
\(\bigcup_{\ell=1}^{s_k}U_{k\ell}\).
Thus \(\varphi_t=\varphi\) on this union under the fixed exterior
identification, and
\((\varphi_t\times\varphi_t)^*W_{\mb P^1}\)
is independent of \(t\) there.
Recalling that
\(W_\varphi
=\iota_\Delta^*\bigl(
W-(\varphi\times\varphi)^*W_{\mb P^1}
\bigr)\),
we obtain
\(\partial_{z_m}W_\varphi
=\iota_\Delta^*(\partial_{z_m}W)\)
on \(\bigcup_{\ell=1}^{s_k}U_{k\ell}\).
The right-hand side of the preceding variational formula for \(W\)
is holomorphic near the diagonal in this region, since the
residue points \(P_{mr}\) lie outside it.
We may therefore restrict that formula to \(Q=P\).
Using the symmetry \(W(R,P)=W(P,R)\) then gives the claimed identity.
\end{proof}

Define the holomorphic one-form
\(\mu:=\sum_{k=1}^N\mk B_k\,dz_k\) on
\(\mc B([\varphi],\delta)\).

\begin{prop}
The one-form \(\mu\) is closed.
\end{prop}

\begin{proof}
It suffices to show that
\(\partial_{z_m}\mk B_k=\partial_{z_k}\mk B_m\)
for \(k\neq m\).
We work at an arbitrary point of the coordinate neighborhood,
using the same notation for the corresponding covering and its
local data.

Fix \(k\neq m\). As in the preceding lemma, the points
\(P_{k\ell}\) and \(d_k\varphi\) remain fixed under the
\(z_m\)-deformation. Differentiating the defining residues of
\(\mk B_k\) therefore gives
\[
\frac{\partial\mk B_k}{\partial z_m}
=
-
\sum_{\substack{1\leq\ell\leq s_k\\ n_{k\ell}>1}}
\sum_{\substack{1\leq r\leq s_m\\ n_{mr}>1}}
\operatorname{Res}_{P=P_{k\ell}}
\operatorname{Res}_{R=P_{mr}}
\frac{W(P,R)^2}
     {d_k\varphi(P)d_m\varphi(R)}.
\]
The neighborhoods of \(P_{k\ell}\) and \(P_{mr}\) are disjoint,
so the two residue operations commute.
Together with the symmetry \(W(P,R)=W(R,P)\), this shows that
the last expression is symmetric under the interchange of
\(k\) and \(m\). Hence
\(\partial_{z_m}\mk B_k=\partial_{z_k}\mk B_m\), and thus
\(d\mu=0\).
\end{proof}

Since \(\mc B([\varphi],\delta)\) is contractible and
\(\mu=\sum_{k=1}^N\mk B_k\,dz_k\) is a closed holomorphic one-form,
there exists a holomorphic function \(\nu\) on
\(\mc B([\varphi],\delta)\) such that \(\mu=d\nu\).
This gives the local residue definition of the Bergman tau-function
used in \cite{KokotovKorotkin2004}, with the contributions over each
branch value grouped together.

\begin{defn}
The local Bergman tau-function associated with the chosen homological
marking is defined by
\(\tau_B:=e^\nu\).
It is a nonvanishing holomorphic function on
\(\mc B([\varphi],\delta)\), uniquely determined up to multiplication
by a nonzero constant, and satisfies
\[
\frac{\partial}{\partial z_k}\log\tau_B
=
\mk B_k,
\qquad
1\leq k\leq N.
\]
Equivalently,
\(d\log\tau_B=\mu\).
\end{defn}

In residues involving the Bergman projective connection, the notation
\(S_B/d_k\varphi\) at \(P_{kj}\) means
\(S_B(u)\,du^2/d_k\varphi\) in the distinguished coordinate
\(u=z_{kj}\). Two such coordinates differ by multiplication by an
\(n_{kj}\)-th root of unity. This linear change has zero Schwarzian
derivative, so the residue is independent of that choice.

\begin{lem}\label{lem:bergman-tau-residue}
For every Hurwitz coordinate \(z_k\), the preceding definition of the
local Bergman tau-function is equivalently expressed by
\[
\frac{\partial}{\partial z_k}\log\tau_B
=
-\frac{1}{6}
\sum_{\substack{1\leq j\leq s_k\\ n_{kj}>1}}
\operatorname{Res}_{P_{kj}}
\frac{S_B}{d_k\varphi}.
\]
\end{lem}

\begin{proof}
Fix a ramification point \(P_{kj}\), put \(n=n_{kj}\), and let
\(u=z_{kj}\) be the distinguished coordinate. Set
\(f=\xi_k\circ\varphi=z_k+u^n\). In the coordinate \(u\), comparison of
the diagonal expansions of \(W\) and
\((\varphi\times\varphi)^*W_{\mb P^1}\) gives
\[
W_\varphi
=
\frac{1}{6}
\left(
S_B(u)-\{f,u\}
\right)du^2,
\]
where
\[
\{f,u\}
:=
\frac{f'''(u)}{f'(u)}
-\frac{3}{2}
\left(\frac{f''(u)}{f'(u)}\right)^2
\]
is the Schwarzian derivative. Since \(f=z_k+u^n\),
\(\{f,u\} = -\frac{n^2-1}{2u^2}\).
Moreover, \(d_k\varphi=n u^{n-1}du\). Hence
\[
\frac{\{f,u\}\,du^2}{d_k\varphi}
=
-\frac{n^2-1}{2n}u^{-n-1}\,du,
\]
which has zero residue at \(u=0\). Therefore
\[
\operatorname{Res}_{P_{kj}}
\frac{W_\varphi}{d_k\varphi}
=
\frac{1}{6}
\operatorname{Res}_{P_{kj}}
\frac{S_B}{d_k\varphi}.
\]
Summing over the ramification points above \(Q_k\) and using
\(\partial_{z_k}\log\tau_B=\mk B_k\) together with the definition of
\(\mk B_k\) proves the formula.
\end{proof}

Similarly, for each
\(\mf z+\mf t\in B_{\mf z}(\delta)\), let
\(\mc S_{\mf t}\) denote the Schiffer bidifferential on
\(X_{\mf t}\times X_{\mf t}\). With respect to the transported
homological marking, let \(B(\mf t)\) be the period matrix and let
\(\{\omega_i^{X_{\mf t}}\}_{i=1}^g\) be the normalized basis of
holomorphic differentials on \(X_{\mf t}\). Then
\[
\mc S_{\mf t}(P,Q)
=
W_{\mf t}(P,Q)
-
\pi
\sum_{i,j=1}^g
\bigl(\operatorname{Im}B(\mf t)\bigr)^{-1}_{ij}
\omega_i^{X_{\mf t}}(P)\omega_j^{X_{\mf t}}(Q).
\]
Thus the family \(\{\mc S_{\mf t}\}\) defines a smooth section of the
relative bidifferential bundle \(\mc L\) introduced above. We denote
this family again by \(\mc S\).

Define
\(\mc S_\varphi
:=\iota_\Delta^*\bigl(
\mc S-(\varphi\times\varphi)^*W_{\mb P^1}
\bigr)\).
As in the case of \(W_\varphi\), this defines a meromorphic
quadratic differential on \(X\). Consequently,
\(\mc S_\varphi/d_k\varphi\) is a meromorphic one-form on
\(\varphi^{-1}(D_k)\) for each \(1\leq k\leq N\).
By the definition of the Schiffer bidifferential,
\(\mc S_\varphi
=W_\varphi-\pi\sum_{i,j=1}^g
\bigl((\operatorname{Im}B)^{-1}\bigr)_{ij}\omega_i\omega_j\).
For each \(1\leq k\leq N\), define
\(\mk S_k
:=-\sum_{\substack{1\leq\ell\leq s_k\\ n_{k\ell}>1}}
\operatorname{Res}_{P_{k\ell}}
(\mc S_\varphi/d_k\varphi)\).
By linearity of residues and the variational formula for
\(\log\det\operatorname{Im}B\), we obtain
\(\mk S_k
=\mk B_k+\partial_{z_k}\log\det\operatorname{Im}B\).
Since \(\partial_{z_k}\log\tau_B=\mk B_k\), it follows that
\[
\mk S_k
=
\frac{\partial}{\partial z_k}
\log\bigl(\det\operatorname{Im}B\cdot\tau_B\bigr).
\]
\begin{defn}
The local Schiffer tau-function is defined by
\[
\tau_{\mathrm{Sch}}
:=\det\operatorname{Im}B\cdot\tau_B.
\]
It is smooth and nowhere vanishing as a function of the Hurwitz
coordinates, but need not be holomorphic. It is determined up to
the same nonzero multiplicative constant as \(\tau_B\).
\end{defn}

The preceding identity therefore gives, for each \(1\leq k\leq N\),
\[
\frac{\partial}{\partial z_k}\log\tau_{\mathrm{Sch}}
=
-\sum_{\substack{1\leq\ell\leq s_k\\ n_{k\ell}>1}}
\operatorname{Res}_{P_{k\ell}}
\frac{\mc S_\varphi}{d_k\varphi}.
\]

\section{Variation of Traces of Resolvent Powers}\label{sec:resolvent-variation}

Let \([\varphi]\in H_{g,d,N}^{\mb P^1}\), and choose a representative
\(
\varphi:X\to\mb P^1
\)
of degree \(d\), with branch values
\(\operatorname{Br}(\varphi)=\{Q_1,\ldots,Q_N\}\).
Fix \(1\le k\le N\), and let
\(\varphi_t:X_t\to\mb P^1\), \(|t|<\delta\), be the deformation
constructed in the previous sections in the \(k\)-th Hurwitz coordinate
direction, with base map \(\varphi_0=\varphi:X\to\mb P^1\).
We take \(\delta<\epsilon_k/4\), shrinking it further as needed.
All derivatives with respect to \(t\) are Wirtinger derivatives:
\(\partial_t=\tfrac12(\partial_s-\sqrt{-1}\,\partial_r)\)
for \(t=s+\sqrt{-1}\,r\). At the base point, this derivative
represents the \(k\)-th Hurwitz coordinate derivative
\(\partial_{z_k}\). Set
\(
g_t:=\varphi_t^*ds_{\mathrm{rd}}^2.
\)

Let \(S_t\) be the conic set and \(X_t':=X_t\setminus S_t\).
Write \(\Delta_{t,c}\) for the initial Laplace--Beltrami operator
on \(C_c^\infty(X_t')\subset L^2(X_t',dA_{g_t})\),
\(\Delta_t:=\Delta_{X_t}\) for its maximal realization, and
\(\mc F_t\subset V_{S_t}\) for the Friedrichs subspace.
The corresponding Friedrichs extension is \(\Delta_{t,\mc F_t}\).
The purpose of this section is to compute, for every integer \(m>1\),
the variation of
\(\Tr R_{\mc F_t}(\lambda)^m\)
and express it in terms of the local \(S\)-matrix at the ramification
points lying over \(Q_k\). The argument has four stages. Since the
operators act on varying surfaces and have a priori varying domains,
we first pull the family back to the fixed underlying smooth surface
\(X\). Second, we identify the maximal, minimal, and Friedrichs
domains. Third, we establish differentiability of the pulled-back
operators and their resolvents in the operator ideals needed below.
Finally, we express the infinitesimal variation as a boundary pairing
and identify it with derivatives of the relevant local \(S\)-matrix
entries.

Let \(\Theta_t:X\to X_t\) be the family of smooth diffeomorphisms
constructed above, with \(\Theta_0=\id_X\). Let \(S\) denote the set of conic points of
\((X,g)\), where \(g=g_0\), so that \(S=S_0\).
By construction, \(\Theta_t(S)=S_t\).
Set \(\widetilde g_t:=\Theta_t^*g_t\). Then
\((X,\widetilde g_t)\) is a conic Riemannian surface with the fixed
conic set \(S\) and the same conic orders as \((X,g)\). Set
\(X':=X\setminus S\), and denote the corresponding initial and
maximal Laplacians by \(\widetilde\Delta_{t,c}\) and
\(\widetilde\Delta_t\). At the base point, the pulled-back metric and maximal Laplacian
are \(g\) and \(\Delta_X\), respectively.

The map \(\Theta_t:(X,\widetilde g_t)\to(X_t,g_t)\) is an
isomorphism of conic Riemannian surfaces. By
Theorem~\ref{thm:conic-isomorphism}, pullback is unitary on the
corresponding \(L^2\)-spaces and identifies their maximal and
minimal graph domains. In particular,
\[
\Theta_t^*\mc D_{\max}(\Delta_t)
=\mc D_{\max}(\widetilde\Delta_t),\qquad
\Theta_t^*\mc D_{\min}(\Delta_t)
=\mc D_{\min}(\widetilde\Delta_t),
\]
and \(\widetilde\Delta_t\Theta_t^*=\Theta_t^*\Delta_t\) on
\(\mc D_{\max}(\Delta_t)\).

Using the conic coordinates transported by \(\Theta_t\), identify
the critical asymptotic space of \((X,\widetilde g_t)\) with \(V_S\),
and write \(\pi_{S,t}:\mc D_{\max}(\widetilde\Delta_t)\to V_S\)
for its asymptotic projection. The induced symplectic isomorphism
\(\widetilde\Theta_t^*:V_{S_t}\to V_S\) satisfies
\(\pi_{S,t}\Theta_t^*=\widetilde\Theta_t^*\pi_{S_t}\).
Let \(\widetilde{\mc F}_t\subset V_S\) be the Friedrichs
Lagrangian for \(\widetilde\Delta_t\). Since
\(\widetilde\Theta_t^*\mc F_t=\widetilde{\mc F}_t\), pullback
maps \(\mc D_{\mc F_t}=\pi_{S_t}^{-1}(\mc F_t)\) onto
\(\mc D_{\widetilde{\mc F}_t}
:=\pi_{S,t}^{-1}(\widetilde{\mc F}_t)\).
Thus the Friedrichs realizations are unitarily equivalent:
\[
\widetilde\Delta_{t,\widetilde{\mc F}_t}
:=\widetilde\Delta_t|_{\mc D_{\widetilde{\mc F}_t}}
=\Theta_t^*\Delta_{t,\mc F_t}(\Theta_t^*)^{-1}.
\]

For \(\lambda\) in their common resolvent set, put
\(R_{\mc F_t}(\lambda):=(\lambda-\Delta_{t,\mc F_t})^{-1}\)
and \(\widetilde R_{\widetilde{\mc F}_t,t}(\lambda)
:=(\lambda-\widetilde\Delta_{t,\widetilde{\mc F}_t})^{-1}\).
These resolvents and their powers are intertwined by \(\Theta_t^*\).
The conic heat-trace asymptotics of
\cite{Cheeger1983,Mooers1999} imply the usual two-dimensional Weyl
bound. Hence their powers of integer order \(m>1\) are trace class.
Unitary invariance of the trace gives
\[
\Tr R_{\mc F_t}(\lambda)^m
=\Tr\widetilde R_{\widetilde{\mc F}_t,t}(\lambda)^m,
\qquad m>1.
\]
We may therefore compute the variation on the fixed smooth surface.

We next study the dependence on \(t\) of the pulled-back operators.
The following metric comparison is the first step toward identifying
their maximal and minimal domains with those of \(\Delta_X\).

\begin{lem}\label{lem:metric-equivalence}
For \(\delta>0\) sufficiently small, the metrics
\(\widetilde g_t\) and \(g\) are uniformly equivalent on \(X'\).
More precisely, there exists \(C\geq1\), independent of \(t\), such that
\[
C^{-1}g\leq\widetilde g_t\leq Cg
\qquad\text{on }X',\quad |t|<\delta.
\]
\end{lem}

\begin{proof}
We compare the metrics successively on the fixed exterior, on the
inner conic neighborhoods, and on the transition annuli. Only the
last region requires quantitative estimates.
On the fixed exterior \(X_k^o\), the construction gives
\(\varphi_t\circ\Theta_t=\varphi\), so \(\widetilde g_t=g\).
It remains to compare the metrics on the local neighborhoods.
Recall that
\(z_{kj,t}:U_{kj}(t)\to\Omega_{kj}(t)\) is the coordinate
centered at \(P_{kj}(t)\), with
\(U_{kj}(0)=U_{kj}\), \(P_{kj}(0)=P_{kj}\), and
\(z_{kj,0}=z_{kj}\).
Set \(U_{kj}(t)':=U_{kj}(t)\cap X_t'\) and
\(U_{kj}':=U_{kj}\cap X'\).
On \(U_{kj}\), write \(\zeta=z_{kj}\) and \(n=n_{kj}\).
Functions and tensors in the following local formulas are expressed
in these coordinates. The construction gives
\(\Theta_t|_{U_{kj}}=\widehat H_{kj,t}
=z_{kj,t}^{-1}\circ H_{kj,t}\circ z_{kj}\).

Since \(\xi_k\circ\varphi_t=z_k+t+z_{kj,t}^n\) on \(U_{kj}(t)\),
the local metric and Laplace--Beltrami differential expression are
\[
g_t
=\frac{4n^2|z_{kj,t}|^{2n-2}}
{(1+|z_k+t+z_{kj,t}^n|^2)^2}|dz_{kj,t}|^2,
\qquad
\Delta_t^{(kj)}
=-\frac{(1+|z_k+t+z_{kj,t}^n|^2)^2}
{n^2|z_{kj,t}|^{2n-2}}
\frac{\partial^2}{\partial z_{kj,t}\partial\overline z_{kj,t}}.
\]
At \(t=0\), denote the corresponding local expression of
\(\Delta_X\) by \(\Delta_X^{(kj)}\).
Using \(H_{kj,t}(\zeta)^n=\zeta^n-\beta(|\zeta|^n)t\), we obtain
\[
\widetilde g_t
=
\frac{4n^2|H_{kj,t}(\zeta)|^{2n-2}}
{\bigl(1+|z_k+\zeta^n+(1-\beta(|\zeta|^n))t|^2\bigr)^2}
\widehat H_{kj,t}^*|dz_{kj,t}|^2.
\]

On the part of \(U_{kj}\) where \(|\zeta|^n<\epsilon_k/2\),
we have \(\beta(|\zeta|^n)=0\) and
\(H_{kj,t}(\zeta)=\zeta\). Consequently,
\[
\widetilde g_t=a_tg,
\qquad
a_t(\zeta)
:=\frac{(1+|z_k+\zeta^n|^2)^2}
{(1+|z_k+t+\zeta^n|^2)^2}.
\]
The function \(a_t\) is smooth and positive, including at \(\zeta=0\).
Thus the conic coordinates and conic orders are unchanged; when
\(n=1\), the metric remains smooth at the center.
Near conic points not lying over \(Q_k\), the metric is fixed, and
we set \(a_t=1\). We use \(a_t\) on these fixed neighborhoods of
\(S\). There, the two-dimensional conformal transformation laws give
\[
\widetilde\Delta_t=a_t^{-1}\Delta_X,
\qquad
dA_{\widetilde g_t}=a_t\,dA_g.
\]
Both \(a_t\) and \(a_t^{-1}\) are uniformly bounded for
\(|t|<\delta\); more precisely,
\((1+\delta)^{-4}\leq a_t\leq(1+\delta)^4\).

The equality \(\widetilde g_t=g\) also holds in each local
neighborhood wherever
\(|\zeta|^n\geq3\epsilon_k/4\), because there \(\beta(|\zeta|^n)=1\).
Thus the metric variation is supported in the fixed compact set
\[
K_k
:=\bigcup_{j=1}^{s_k}
\{P\in U_{kj}:|z_{kj}(P)|^{n_{kj}}\leq3\epsilon_k/4\}.
\]

We now compare the metrics on all of \(U_{kj}'\), including the
transition region. Continue to write \(\zeta=z_{kj}\) and \(n=n_{kj}\).
Define the auxiliary function \(F_t\) by
\(F_t(\zeta)=1-\beta(|\zeta|^n)t/\zeta^n\) for \(\zeta\neq0\)
and \(F_t(0)=1\).
Then \(H_{kj,t}(\zeta)=\zeta F_t(\zeta)^{1/n}\).
The estimate~\eqref{eq:H-uniform-C1} and the choice of \(\delta\)
in Appendix~\ref{app:local-interpolation} give
\(\|dH_{kj,t}-I\|\leq1/2\) and
\(1/2\leq|F_t(\zeta)|\leq3/2\). Hence
\[
\frac14|d\zeta|^2
\leq\widehat H_{kj,t}^*|dz_{kj,t}|^2
\leq\frac94|d\zeta|^2,
\qquad
2^{-(2-2/n)}
\leq\frac{|H_{kj,t}(\zeta)|^{2n-2}}{|\zeta|^{2n-2}}
\leq\left(\frac32\right)^{2-2/n}.
\]
For \(n=1\), the ratio in the second inequality is identically
\(1\), including at \(\zeta=0\).
Since \(|(1-\beta(|\zeta|^n))t|<\delta\), the triangle inequality
also gives
\[
(1+\delta)^{-4}
\leq
\frac{(1+|z_k+\zeta^n|^2)^2}
{\bigl(1+|z_k+\zeta^n+(1-\beta(|\zeta|^n))t|^2\bigr)^2}
\leq(1+\delta)^4.
\]
Combining these estimates with the local metric formula, and using
\(2^{-(2-2/n)}\geq1/4\) and
\((3/2)^{2-2/n}\leq9/4\), we obtain
\[
\frac1{16}(1+\delta)^{-4}g
\leq\widetilde g_t
\leq\frac{81}{16}(1+\delta)^4g
\qquad\text{on }X'.
\]
The particular constants are not essential; the point is that both
bounds are positive and independent of \(t\). This proves the
assertion.
\end{proof}

In particular, the spaces \(L^2(X',dA_{\widetilde g_t})\) and
\(L^2(X',dA_g)\) coincide as sets, with uniformly equivalent norms.
The next step is to prove that the graph norms of \(\Delta_X\)
and \(\widetilde\Delta_t\) are uniformly equivalent. The estimates
establishing this equivalence also identify their maximal domains;
equality of the minimal domains then follows by taking the closure
of \(C_c^\infty(X')\). We obtain these estimates from the conformal
relation near \(S\) and interior elliptic regularity on the smooth
locus.

\begin{thm}\label{thm:common-laplacian-domains}
For \(\delta>0\) sufficiently small, the maximal and minimal
domains of \(\widetilde\Delta_t\) are independent of \(t\).
More precisely, for \(|t|<\delta\),
\[
\mc D_{\max}(\widetilde\Delta_t)=\mc D_{\max}(\Delta_X),
\qquad
\mc D_{\min}(\widetilde\Delta_t)=\mc D_{\min}(\Delta_X).
\]
The corresponding graph norms are uniformly equivalent.
\end{thm}

\begin{proof}
It suffices to prove equality of the maximal domains together
with the uniform graph-norm comparison
\begin{equation}\label{eq:common-graph-norm-comparison}
C^{-1}\|u\|_{\Delta_X}
\leq\|u\|_{\widetilde\Delta_t}
\leq C\|u\|_{\Delta_X}.
\end{equation}
Indeed, the minimal domains are the closures of the same space
\(C_c^\infty(X')\) in the respective graph norms. Once
\eqref{eq:common-graph-norm-comparison} holds on the common
maximal domain, these closures coincide.

Set \(U:=X\setminus K_k\), where \(K_k\) is the fixed compact
set defined in the proof of Lemma~\ref{lem:metric-equivalence}.
Then \(U\) is open and \(\widetilde g_t=g\) on \(U\).
Each \(U_{kj}\) contains the compact set
\(\{P\in U_{kj}:|z_{kj}(P)|^{n_{kj}}\leq3\epsilon_k/4\}\).
Thus
\[
\{U,U_{k1},\ldots,U_{ks_k}\}
\]
is a fixed finite open cover of \(X\).
In the estimates below, we use \(dA_g\) for the \(L^2\)-norms
unless otherwise specified. By Lemma~\ref{lem:metric-equivalence},
these norms are uniformly equivalent to those defined by
\(dA_{\widetilde g_t}\). Each graph norm still uses the measure
associated with its operator. For an \(L^2\)-function on \(X'\),
\(\supp v\) denotes the closure in \(X\) of its essential support
in \(X'\); in particular, it may contain conic points. We first
compare functions supported in a single member of the cover, and
then combine the estimates.

First, suppose that \(v\) belongs to either maximal domain and
\(\supp v\Subset U\). Both metrics and differential expressions
agree on \(U\), while \(v\) and both distributional Laplacians
vanish outside \(\supp v\). It follows that \(v\) belongs to
the other maximal domain as well, with \(\supp v\Subset U\), and
\(\|v\|_{\widetilde\Delta_t}=\|v\|_{\Delta_X}\).

Next, fix \(1\leq j\leq s_k\) and suppose that \(v\) belongs
to either maximal domain with support contained in a fixed compact
subset of \(U_{kj}\). Write \(\zeta=z_{kj}\) and \(n=n_{kj}\),
and let
\(V_j:=\{P\in U_{kj}:|\zeta(P)|^n<\epsilon_k/2\}\).
On \(V_j\cap X'\), the conformal relation gives
\(\widetilde\Delta_t=a_t^{-1}\Delta_X\).
For the present \(v\), with \(\supp v\Subset U_{kj}\), the
uniform bounds for \(a_t\) and \(a_t^{-1}\) give
\[
\begin{aligned}
\|\widetilde\Delta_tv\|_{L^2(V_j\cap X',dA_g)}
&\leq C\|\Delta_Xv\|_{L^2(V_j\cap X',dA_g)},\\
\|\Delta_Xv\|_{L^2(V_j\cap X',dA_g)}
&\leq C\|\widetilde\Delta_tv\|_{L^2(V_j\cap X',dA_g)}.
\end{aligned}
\]

To treat the remaining part of the support, choose
\(3\epsilon_k/4<b_j<b_j'<\epsilon_k\) so that the fixed compact
set containing \(\supp v\) lies in \(\{|\zeta|^n<b_j\}\), and set
\(A_j:=\{\epsilon_k/3<|\zeta|^n<b_j\}\) and
\(A_j':=\{\epsilon_k/4<|\zeta|^n<b_j'\}\).
Both annuli are contained in \(U_{kj}\cap X'\), with
\(A_j\Subset A_j'\), and
\(\supp v\subset V_j\cup A_j\Subset U_{kj}\).
Throughout the following local estimates, every function \(v\)
is subject to this support condition.
Use the fixed coordinate \(\zeta\) to define the Sobolev norms
on these annuli. On \(\overline{A_j'}\), the density of \(dA_g\)
relative to Euclidean area is bounded above and below by positive
constants. Shrink \(\delta\), if necessary, so that the closed parameter disc
\(\{|t|\leq\delta\}\) lies within the range where the metric family
is defined and smooth on \(X'\). Joint smoothness on the product
of this disc with \(\overline{A_j'}\) gives uniform bounds for
the coefficients of both operators and their derivatives of the
orders needed below. Together with metric equivalence, it also
gives uniform ellipticity on this fixed annulus. The constants
in the interior estimates can therefore be chosen independently
of \(t\).
Thus, let \(L=\Delta_X\) or \(L=\widetilde\Delta_t\), and take
\(v\in\mc D_{\max}(L)\) with
\(\supp v\subset V_j\cup A_j\), as in the local comparison above.
Interior elliptic regularity gives
\[
\|v\|_{H^2(A_j)}
\leq C\left(
\|Lv\|_{L^2(A_j',dA_g)}
+\|v\|_{L^2(A_j',dA_g)}\right).
\]
To relate this estimate to the graph norm, first suppose that
\(v\in\mc D_{\max}(\Delta_X)\) with
\(\supp v\subset V_j\cup A_j\). Since \(A_j'\subset X'\),
\[
\begin{aligned}
\|\Delta_Xv\|_{L^2(A_j',dA_g)}
+\|v\|_{L^2(A_j',dA_g)}
&\leq \|\Delta_Xv\|_{L^2(X',dA_g)}
+\|v\|_{L^2(X',dA_g)}\\
&\leq \sqrt{2}\left(
\|\Delta_Xv\|_{L^2(X',dA_g)}^2
+\|v\|_{L^2(X',dA_g)}^2\right)^{1/2}
\\&=\sqrt{2}\|v\|_{\Delta_X}.
\end{aligned}
\]
If instead \(v\in\mc D_{\max}(\widetilde\Delta_t)\) with
\(\supp v\subset V_j\cup A_j\),
Lemma~\ref{lem:metric-equivalence} gives
\[
\begin{aligned}
\|\widetilde\Delta_tv\|_{L^2(A_j',dA_g)}
+\|v\|_{L^2(A_j',dA_g)}
&\leq C\left(
\|\widetilde\Delta_tv\|_{L^2(X',dA_{\widetilde g_t})}
+\|v\|_{L^2(X',dA_{\widetilde g_t})}\right)\\
&\leq C\sqrt{2}\|v\|_{\widetilde\Delta_t}.
\end{aligned}
\]
Consequently, the \(H^2(A_j)\)-norm is controlled by whichever
graph norm is already known to be finite. No comparison of the
two graph norms is used in this step.

On \(A_j\), each differential expression is a sum of derivatives
of order at most two with uniformly bounded coefficients.
The area density is also bounded there. Applying these coefficient
bounds and then the corresponding interior estimate gives,
for \(\supp v\subset V_j\cup A_j\),
\[
\begin{aligned}
\|\widetilde\Delta_tv\|_{L^2(A_j,dA_g)}
&\leq C\|v\|_{H^2(A_j)}
\leq C'\|v\|_{\Delta_X},
&&\substack{v\in\mc D_{\max}(\Delta_X),\\
\supp v\subset V_j\cup A_j},\\
\|\Delta_Xv\|_{L^2(A_j,dA_g)}
&\leq C\|v\|_{H^2(A_j)}
\leq C'\|v\|_{\widetilde\Delta_t},
&&\substack{v\in\mc D_{\max}(\widetilde\Delta_t),\\
\supp v\subset V_j\cup A_j}.
\end{aligned}
\]
All constants are independent of \(t\).

Now let \(v\in\mc D_{\max}(\Delta_X)\) with
\(\supp v\subset V_j\cup A_j\), as above. The conformal
comparison on \(V_j\cap X'\) and the first annular estimate
show that \(\widetilde\Delta_tv\) has finite squared integral
over each of \(V_j\cap X'\) and \(A_j\), including the part
approaching the conic point. These sets cover
\(\supp v\cap X'\), and \(\widetilde\Delta_tv\) vanishes as a
distribution outside \(\supp v\). Thus it belongs to
\(L^2(X',dA_g)\), with
\[
\begin{aligned}
\|\widetilde\Delta_tv\|_{L^2(X',dA_g)}^2
&\leq
\|\widetilde\Delta_tv\|_{L^2(V_j\cap X',dA_g)}^2
+\|\widetilde\Delta_tv\|_{L^2(A_j,dA_g)}^2\\
&\leq C\|\Delta_Xv\|_{L^2(V_j\cap X',dA_g)}^2
+C'\|v\|_{\Delta_X}^2
\leq C''\|v\|_{\Delta_X}^2.
\end{aligned}
\]
Metric equivalence therefore yields
\(v\in\mc D_{\max}(\widetilde\Delta_t)\), still with
\(\supp v\subset V_j\cup A_j\), and
\[
\|v\|_{\widetilde\Delta_t}^2
\leq C\left(
\|v\|_{L^2(X',dA_g)}^2
+\|\widetilde\Delta_tv\|_{L^2(X',dA_g)}^2\right)
\leq C'\|v\|_{\Delta_X}^2.
\]
Conversely, starting with
\(v\in\mc D_{\max}(\widetilde\Delta_t)\) with the same support
condition \(\supp v\subset V_j\cup A_j\), use the reverse
conformal comparison and the second annular estimate. The same
support argument gives \(\Delta_Xv\in L^2(X',dA_g)\), with
\[
\begin{aligned}
\|\Delta_Xv\|_{L^2(X',dA_g)}^2
&\leq C\|\widetilde\Delta_tv\|_{L^2(V_j\cap X',dA_g)}^2
+C'\|v\|_{\widetilde\Delta_t}^2\\
&\leq C''\|v\|_{\widetilde\Delta_t}^2.
\end{aligned}
\]
Here the last step again uses metric equivalence to compare
\(dA_g\) with \(dA_{\widetilde g_t}\).
Together with the same comparison for the \(L^2\)-norm of \(v\),
this proves \(v\in\mc D_{\max}(\Delta_X)\), with
\(\supp v\subset V_j\cup A_j\), and
\(\|v\|_{\Delta_X}\leq C\|v\|_{\widetilde\Delta_t}\).
Thus the two domain inclusions and graph-norm bounds hold for
functions supported in \(U_{kj}\). The constants may depend on the
fixed compact set containing the support, but are independent of
\(v\) and \(t\).

Finally, we pass to an arbitrary function on \(X\) by localization.
Choose a smooth partition of unity
\(\{\chi_0,\chi_1,\ldots,\chi_{s_k}\}\) subordinate to this
cover, with \(\supp\chi_0\Subset U\) and
\(\supp\chi_j\Subset U_{kj}\) for \(1\leq j\leq s_k\).
Near \(P_{kj}\), only \(\chi_j\) can be nonzero, so
\(\chi_j=1\) there. Near every conic point outside the fiber
over \(Q_k\), we have \(\chi_0=1\). Thus every cutoff is
constant near every conic point. All cutoffs are independent
of \(t\).

We record the localization estimate for these cutoffs. Let
\(L\) denote either \(\Delta_X\) or \(\widetilde\Delta_t\),
and take \(u\in\mc D_{\max}(L)\). The product rule gives
\(L(\chi_\alpha u)=\chi_\alpha Lu+[L,\chi_\alpha]u\).
The first-order commutator is supported in a fixed compact
subset of \(X'\). Choose fixed open sets
\(E_\alpha\Subset F_\alpha\Subset X'\) containing that
support, and use fixed smooth charts to define their Sobolev
norms. Interior elliptic regularity gives
\[
\begin{aligned}
\|[L,\chi_\alpha]u\|_{L^2(X',dA_g)}
&\leq C_\alpha\|u\|_{H^1(E_\alpha)}\\
&\leq C_\alpha\left(
\|Lu\|_{L^2(F_\alpha,dA_g)}
+\|u\|_{L^2(F_\alpha,dA_g)}\right)
\leq C_\alpha'\|u\|_L.
\end{aligned}
\]
These estimates apply because \(u,Lu\in L^2(X',dA_g)\)
imply \(u\in H^2_{\mathrm{loc}}(X')\). Their constants are
uniform in \(t\), since the coefficients and ellipticity
constants are uniformly controlled on \(\overline{F_\alpha}\).
Together with
\(\|\chi_\alpha u\|_{L^2(X',dA_g)}
\leq\|\chi_\alpha\|_\infty\|u\|_{L^2(X',dA_g)}\),
the product rule proves
\(\chi_\alpha u\in\mc D_{\max}(L)\) and
\(\|\chi_\alpha u\|_L\leq C_\alpha''\|u\|_L\).
Summing over the finite family gives
\(\sum_{\alpha=0}^{s_k}\|\chi_\alpha u\|_L
\leq C\|u\|_L\).

To apply the local comparisons, note that
\[
\supp(\chi_0u)\subset\supp\chi_0\Subset U,
\qquad
\supp(\chi_ju)\subset\supp\chi_j\Subset U_{kj}
\quad(1\leq j\leq s_k).
\]
For each \(j\), choose the radii \(b_j,b_j'\) in the local
argument using the fixed compact set \(\supp\chi_j\).
The resulting annuli are independent of \(u\) and \(t\), and
\(\supp(\chi_ju)\subset V_j\cup A_j\). Hence each local
comparison has a constant independent of \(u\) and \(t\).
Since there are only finitely many pieces, these constants
can be replaced by a single constant.

First let \(u\in\mc D_{\max}(\Delta_X)\). The cutoff
estimates give \(\chi_\alpha u\in\mc D_{\max}(\Delta_X)\)
for every \(\alpha\). For \(\alpha=0\), substitute
\(v=\chi_0u\) into the comparison for \(\supp v\Subset U\).
It follows that \(\chi_0u\in\mc D_{\max}(\widetilde\Delta_t)\)
and \(\|\chi_0u\|_{\widetilde\Delta_t}
=\|\chi_0u\|_{\Delta_X}\).
For \(1\leq j\leq s_k\), substitute \(v=\chi_ju\) into
the forward comparison on \(U_{kj}\). Its support hypothesis
holds because \(\supp(\chi_ju)\subset V_j\cup A_j\).
Hence \(\chi_ju\in\mc D_{\max}(\widetilde\Delta_t)\), and
\[
\|\chi_ju\|_{\widetilde\Delta_t}
\leq C_j\|\chi_ju\|_{\Delta_X}
\leq C_j'\|u\|_{\Delta_X}.
\]
Here the first inequality is the local graph-norm comparison,
and the second is the cutoff estimate above for \(L=\Delta_X\).
All pieces therefore belong to
\(\mc D_{\max}(\widetilde\Delta_t)\). Since
\(u=\sum_{\alpha=0}^{s_k}\chi_\alpha u\), their sum belongs
to this domain. Taking \(C=\max\{1,C_1,\ldots,C_{s_k}\}\),
the triangle inequality and the cutoff estimates yield
\[
\begin{aligned}
\|u\|_{\widetilde\Delta_t}
&\leq\|\chi_0u\|_{\widetilde\Delta_t}
+\sum_{j=1}^{s_k}\|\chi_ju\|_{\widetilde\Delta_t}\\
&\leq\|\chi_0u\|_{\Delta_X}
+\sum_{j=1}^{s_k}C_j\|\chi_ju\|_{\Delta_X}\\
&\leq C\sum_{\alpha=0}^{s_k}\|\chi_\alpha u\|_{\Delta_X}
\leq C'\|u\|_{\Delta_X}.
\end{aligned}
\]

Conversely, let \(u\in\mc D_{\max}(\widetilde\Delta_t)\).
The cutoff estimates with \(L=\widetilde\Delta_t\) first give
\(\chi_\alpha u\in\mc D_{\max}(\widetilde\Delta_t)\).
For \(\chi_0u\), the comparison on \(U\) again gives membership
in \(\mc D_{\max}(\Delta_X)\) and equality of the two graph
norms, since \(\supp(\chi_0u)\Subset U\).
For each \(j\geq1\), apply the reverse comparison on \(U_{kj}\)
to \(v=\chi_ju\), using
\(\supp(\chi_ju)\subset V_j\cup A_j\). This gives
\(\chi_ju\in\mc D_{\max}(\Delta_X)\) and
\[
\|\chi_ju\|_{\Delta_X}
\leq D_j\|\chi_ju\|_{\widetilde\Delta_t}
\leq D_j'\|u\|_{\widetilde\Delta_t},
\]
where the second inequality uses the cutoff estimate for
\(L=\widetilde\Delta_t\). These constants are independent
of \(u\) and \(t\). Summing the pieces proves
\(u\in\mc D_{\max}(\Delta_X)\), and, with
\(D=\max\{1,D_1,\ldots,D_{s_k}\}\),
\[
\begin{aligned}
\|u\|_{\Delta_X}
&\leq\|\chi_0u\|_{\Delta_X}
+\sum_{j=1}^{s_k}\|\chi_ju\|_{\Delta_X}\\
&\leq\|\chi_0u\|_{\widetilde\Delta_t}
+\sum_{j=1}^{s_k}D_j\|\chi_ju\|_{\widetilde\Delta_t}\\
&\leq D\sum_{\alpha=0}^{s_k}\|\chi_\alpha u\|_{\widetilde\Delta_t}
\leq D'\|u\|_{\widetilde\Delta_t}.
\end{aligned}
\]
This proves equality of the maximal domains and
\eqref{eq:common-graph-norm-comparison}, and hence also equality
of the minimal domains as explained above.
\end{proof}

With the common graph domain established, we now study the
parameter dependence of the operators on this fixed space.

\begin{thm}\label{thm:laplacian-family-differentiability}
Equip \(\mc D_{\max}(\Delta_X)\) with the graph norm of
\(\Delta_X\). Then \(t\mapsto\widetilde\Delta_t\) is a
\(C^1\) family in
\(\mathcal L(\mc D_{\max}(\Delta_X),L^2(X',dA_g))\),
with respect to the real coordinates \((\Re t,\Im t)\).
After shrinking \(\delta\), there is a constant \(C>0\) such that
\[
\|(\widetilde\Delta_t-\Delta_X)u\|_{L^2(X',dA_g)}
\leq C|t|\,\|u\|_{\Delta_X},
\qquad u\in\mc D_{\max}(\Delta_X),\quad |t|<\delta.
\]
\end{thm}

\begin{proof}
We use the open cover and the fixed partition of unity from the
proof of Theorem~\ref{thm:common-laplacian-domains}. Thus
\[
X=U\cup\bigcup_{j=1}^{s_k}U_{kj},
\qquad
1=\chi_0+\sum_{j=1}^{s_k}\chi_j,
\]
where \(\supp\chi_0\Subset U\),
\(\supp\chi_j\Subset U_{kj}\), and all cutoffs are independent
of \(t\). For \(u\in\mc D_{\max}(\Delta_X)\), set
\(u_j:=\chi_j u\) for \(0\leq j\leq s_k\). The localization estimate proved
there gives
\[
\sum_{j=0}^{s_k}\|u_j\|_{\Delta_X}
\leq C\|u\|_{\Delta_X}.
\]
Since \(\widetilde\Delta_t=\Delta_X\) on \(U\),
\((\widetilde\Delta_t-\Delta_X)u_0=0\). It therefore remains
to study the difference on each \(U_{kj}\).

Fix \(j\), write \(\zeta=z_{kj}\) and \(n=n_{kj}\), and use
the sets \(V_j\), \(A_j\), and \(A_j'\) chosen in the proof of
Theorem~\ref{thm:common-laplacian-domains}, now with
\(\supp u_j\subset V_j\cup A_j\). On \(V_j\cap X'\),
\[
\widetilde\Delta_t-\Delta_X
=(a_t^{-1}-1)\Delta_X.
\]
The explicit formula for \(a_t\) shows that the map
\(t\mapsto a_t^{-1}\) is \(C^1\), with values in
\(C^\infty(\overline V_j)\), and that \(a_0=1\). After shrinking
the parameter disc, the mean-value theorem gives
\[
\|a_t-1\|_{L^\infty(V_j)}\leq C|t|.
\]
Since \(a_t\) is uniformly bounded away from zero, this also gives
\(\|a_t^{-1}-1\|_{L^\infty(V_j)}\leq C'|t|\).
Consequently,
\[
\|(\widetilde\Delta_t-\Delta_X)u_j\|_{L^2(V_j\cap X',dA_g)}
\leq C|t|\,\|\Delta_Xu_j\|_{L^2(V_j\cap X',dA_g)}
\leq C|t|\,\|u_j\|_{\Delta_X}.
\]

It remains to control the fixed transition annulus \(A_j\).
This extra step is necessary because \(H_{kj,t}\) need not be
holomorphic there, so the pulled-back metric is not in general a
scalar multiple of \(g\). In the fixed coordinate \(\zeta\), write
\[
\widetilde\Delta_t-\Delta_X
=\sum_{|\alpha|\leq2}b_{\alpha,j}(t,\zeta)\partial^\alpha
\qquad\text{on }A_j'.
\]
Here \(\alpha\) is a real multi-index in the coordinates
\((\Re\zeta,\Im\zeta)\).
The smooth dependence of \(H_{kj,t}\) on \((t,\zeta)\), proved in
Appendix~\ref{app:local-interpolation}, implies that every
coefficient \(b_{\alpha,j}\) is \(C^1\) in the real variables
\((\Re t,\Im t)\), uniformly on \(\overline{A_j'}\), and
\(b_{\alpha,j}(0,\zeta)=0\). Hence
\[
\max_{|\alpha|\leq2}
\|b_{\alpha,j}(t,\cdot)\|_{L^\infty(A_j')}
\leq C_j|t|.
\]
The annular elliptic estimate used in the proof of
Theorem~\ref{thm:common-laplacian-domains} gives
\(\|u_j\|_{H^2(A_j)}\leq C\|u_j\|_{\Delta_X}\). Therefore
\[
\|(\widetilde\Delta_t-\Delta_X)u_j\|_{L^2(A_j,dA_g)}
\leq C_j|t|\,\|u_j\|_{H^2(A_j)}
\leq C_j'|t|\,\|u_j\|_{\Delta_X}.
\]
The two local regions cover \(\supp u_j\cap X'\). Summing first
over these regions and then over the partition of unity yields
\[
\begin{aligned}
\|(\widetilde\Delta_t-\Delta_X)u\|_{L^2(X',dA_g)}
&\leq\sum_{j=1}^{s_k}
\|(\widetilde\Delta_t-\Delta_X)u_j\|_{L^2(X',dA_g)}\\
&\leq C|t|\sum_{j=1}^{s_k}\|u_j\|_{\Delta_X}
\leq C'|t|\,\|u\|_{\Delta_X}.
\end{aligned}
\]

The same localization proves the asserted differentiability.
Indeed, on \(V_j\) the real difference quotients are those of the
multiplier \(a_t^{-1}\), acting on \(\Delta_Xu_j\). On \(A_j\)
they are the difference quotients of the finitely many coefficients
\(b_{\alpha,j}(t,\cdot)\), acting on derivatives of \(u_j\) of
order at most two. Their uniform convergence on the fixed local
sets, together with the two estimates above, implies convergence
in
\(\mathcal L(\mc D_{\max}(\Delta_X),L^2(X',dA_g))\).
Continuity of the first real partial derivatives follows in exactly
the same way from the continuity of the corresponding coefficient
derivatives. Thus
\(t\mapsto\widetilde\Delta_t\) is a \(C^1\) family in the stated
operator space.
\end{proof}

The common quotient
\(\mc D_{\max}(\Delta_X)/\mc D_{\min}(\Delta_X)\) is identified
with the critical asymptotic space \(V_S\) introduced earlier.
We denote its projection by
\(\pi_S:\mc D_{\max}(\Delta_X)\to V_S\).
Since the local conic coordinates are fixed near \(S\), the
asymptotic projections \(\pi_{S,t}\) introduced above coincide with
this common projection \(\pi_S\).
Let \(\omega_{S,t}\) be the induced Green form, defined by
\[
\omega_{S,t}([u],[v])
:=\langle\widetilde\Delta_tu,v\rangle_{\widetilde g_t}
 -\langle u,\widetilde\Delta_tv\rangle_{\widetilde g_t},
\qquad u,v\in\mc D_{\max}(\Delta_X).
\]
Here \(\langle\cdot,\cdot\rangle_h\) denotes the inner product
in \(L^2(X',dA_h)\), linear in its first argument.

\begin{prop}
For every \(|t|<\delta\), the Green form satisfies
\(\omega_{S,t}=\omega_S\), where \(\omega_S\) is the
Green form associated with \(\Delta_X\).
\end{prop}

\begin{proof}
Choose pairwise disjoint neighborhoods \(U_P\) of the points
\(P\in S\) on which the conformal relations above hold, and
cutoffs \(\chi_P\in C_c^\infty(U_P)\) equal to \(1\) near \(P\).
For \(u\in\mc D_{\max}(\Delta_X)\), elliptic regularity away
from \(S\) gives
\((1-\sum_{P\in S}\chi_P)u\in\mc D_{\min}(\Delta_X)\).
Thus the Green form is the sum of the local pairings of
\(u_P:=\chi_Pu\) and \(v_P:=\chi_Pv\).

On \(U_P':=U_P\setminus\{P\}\), the function \(a_t\) is real
and positive, and
\(\widetilde\Delta_t=a_t^{-1}\Delta_X\),
\(dA_{\widetilde g_t}=a_t\,dA_g\). Hence
\[
\begin{aligned}
\omega_{S,t}([u_P],[v_P])
&=\int_{U_P'}
\bigl(a_t^{-1}\Delta_Xu_P\,\overline{v_P}
 -u_P\,\overline{a_t^{-1}\Delta_Xv_P}\bigr)
a_t\,dA_g
\\
&=\int_{U_P'}
\bigl(\Delta_Xu_P\,\overline{v_P}
 -u_P\,\overline{\Delta_Xv_P}\bigr)dA_g
\\
&=\omega_S([u_P],[v_P]).
\end{aligned}
\]
Summing over \(P\in S\) proves the assertion.
\end{proof}

We now use the fixed asymptotic projection and the invariance of the
Green form to identify the Friedrichs condition throughout the
family. The fixed conic coordinates determine the same regular and
singular critical asymptotic modes for every pulled-back Laplacian.
In particular, the Friedrichs
Lagrangian subspace \(\mc F:=\mc F_0=\mc F_X\subset(V_S,\omega_S)\)
is the same for every pulled-back Laplacian. In the notation
introduced above, \(\widetilde{\mc F}_t=\mc F\).
Indeed, the positive conformal change near each conic point
preserves the local Dirichlet energy. The Friedrichs condition
therefore requires the vanishing of the same singular critical
coefficients. Since the maximal domain and the asymptotic
projection are also fixed, the Friedrichs operator domains
coincide as subsets of the common maximal domain.
In this and the following sections, choose \(\mc F^{\#}\) to be
the span of the singular modes in the fixed distinguished coordinates
\(z_{kj}\), and write \(M^X:=M_{\mc F,\mc F^{\#}}^X\).
Thus all subsequent Weyl matrices use the bases
\((\gamma_{kj},\gamma_{kj}^{\#})\).
The common Friedrichs domain is
\[
\mc D_{\mc F}
:=\{u\in\mc D_{\max}(\Delta_X):\pi_S(u)\in\mc F\}.
\]
We write \(\widetilde\Delta_{t,\mc F}
:=\widetilde\Delta_t|_{\mc D_{\mc F}}\) for the corresponding
Friedrichs extension, and
\(\Delta_{X,\mc F}:=\Delta_X|_{\mc D_{\mc F}}\) for the
base operator.

Thus \(\mc D_{\widetilde{\mc F}_t}=\mc D_{\mc F}\). From now on,
we abbreviate
\(\widetilde R_{\mc F,t}(\lambda)
:=\widetilde R_{\widetilde{\mc F}_t,t}(\lambda)
=(\lambda-\widetilde\Delta_{t,\mc F})^{-1}\).
Under this abbreviation, the pullback and trace identities
established at the beginning of the section remain in force. We can
therefore study the trace on the fixed surface \(X\), with the common
Friedrichs domain \(\mc D_{\mc F}\). By metric equivalence, the
varying \(L^2\)-spaces may also be identified, as vector spaces, with
\(L^2(X',dA_g)\).

\begin{lem}\label{lem:resolvent-differentiation}
Let \(\mathscr H:=L^2(X',dA_g)\) and
\(\mathscr D:=\mc D_{\mc F}\), and equip \(\mathscr D\) with the graph norm of
\(\Delta_{X,\mc F}\). After identifying the underlying vector
spaces as above, the map
\(t\longmapsto\widetilde\Delta_{t,\mc F}\)
is \(C^1\) from a disc \(D_0(\delta)\subset\mb C\) into
the space \(\mathcal L(\mathscr D,\mathscr H)\) of bounded
linear operators. Here \(C^1\) refers to the real coordinates
\((\Re t,\Im t)\). If
\(\lambda\in\rho(\Delta_{X,\mc F})\), then, after decreasing
\(\delta\) if necessary,
\(\lambda\in\rho(\widetilde\Delta_{t,\mc F})\) for \(|t|<\delta\), and
\[
\partial_t\widetilde R_{\mc F,t}(\lambda)
=
\widetilde R_{\mc F,t}(\lambda)
(\partial_t\widetilde\Delta_{t,\mc F})
\widetilde R_{\mc F,t}(\lambda).
\]
For every integer \(m>1\), the family
\(\widetilde R_{\mc F,t}(\lambda)^m\) is \(C^1\) in trace norm and
\[
\partial_t
\Tr\widetilde R_{\mc F,t}(\lambda)^m
=
m\Tr\left(
\widetilde R_{\mc F,t}(\lambda)^m
(\partial_t\widetilde\Delta_{t,\mc F})
\widetilde R_{\mc F,t}(\lambda)
\right).
\]
\end{lem}

\begin{proof}
By Theorem~\ref{thm:laplacian-family-differentiability},
\(A_t:=\widetilde\Delta_{t,\mc F}\) is \(C^1\) in
\(\mathcal L(\mathscr D,\mathscr H)\). Fix
\(\lambda\in\rho(A_0)\) and set \(R_0:=(\lambda-A_0)^{-1}\).
Then \(R_0\in\mathcal L(\mathscr H,\mathscr D)\), and hence
\(B_t:=(A_t-A_0)R_0\in\mathcal L(\mathscr H)\).
Since \(A_t\to A_0\) in
\(\mathcal L(\mathscr D,\mathscr H)\), we have
\(\|B_t\|_{\mathcal L(\mathscr H)}\to0\). After decreasing
\(\delta\), we may therefore assume that
\(\|B_t\|_{\mathcal L(\mathscr H)}<1\) for every
\(|t|<\delta\). Hence \(I-B_t\) is invertible, with
\[
(I-B_t)^{-1}=\sum_{j=0}^{\infty}B_t^j,
\]
where the series converges in operator norm. Moreover,
\[
\lambda-A_t
=\bigl(I-(A_t-A_0)R_0\bigr)(\lambda-A_0)
=(I-B_t)(\lambda-A_0).
\]
It follows that \(\lambda\in\rho(A_t)\) and
\(\widetilde R_{\mc F,t}(\lambda)=R_0(I-B_t)^{-1}\).
Since inversion is \(C^1\) on the open set of invertible bounded
operators, this proves that
\(t\mapsto\widetilde R_{\mc F,t}(\lambda)\) is \(C^1\) in
\(\mathcal L(\mathscr H,\mathscr D)\). Differentiating
\((\lambda-A_t)\widetilde R_{\mc F,t}(\lambda)=I\) with respect
to either real coordinate \(q=\Re t,\Im t\) gives
\[
\partial_q\widetilde R_{\mc F,t}(\lambda)
=
\widetilde R_{\mc F,t}(\lambda)
(\partial_qA_t)
\widetilde R_{\mc F,t}(\lambda).
\]
Taking the corresponding complex linear combination proves the
same identity for the Wirtinger derivative
\(\partial_t=\tfrac12(\partial_{\Re t}
-\sqrt{-1}\,\partial_{\Im t})\).

The same two-dimensional Weyl bound makes \(R_0\)
Hilbert--Schmidt. Since \(t\mapsto(I-B_t)^{-1}\) is \(C^1\) in
operator norm, the ideal estimate
\[
\|R_0C\|_{\mathcal S_2(\mathscr H)}
\leq
\|R_0\|_{\mathcal S_2(\mathscr H)}
\|C\|_{\mathcal L(\mathscr H)}
\]
shows from the factorization that
\(t\mapsto\widetilde R_{\mc F,t}(\lambda)\) is \(C^1\) in
Hilbert--Schmidt norm. The product estimate for two
Hilbert--Schmidt operators
\cite[Chapter~X, \S1.3, equation~(1.14), p.~521]{Kato1995}
then gives \(C^1\)-dependence of every integer power \(m>1\) in
trace norm. For either real coordinate \(q\), the product rule gives
\[
\partial_q\widetilde R_{\mc F,t}(\lambda)^m
=
\sum_{\ell=0}^{m-1}
\widetilde R_{\mc F,t}(\lambda)^{\ell+1}
(\partial_qA_t)
\widetilde R_{\mc F,t}(\lambda)^{m-\ell}.
\]
Here
\((\partial_qA_t)\widetilde R_{\mc F,t}(\lambda)\) is bounded on
\(\mathscr H\), while
\(\widetilde R_{\mc F,t}(\lambda)^m\) is trace class. Cyclicity of
the trace
\cite[Chapter~X, \S1.4, equation~(1.26), p.~524]{Kato1995}
therefore shows that all \(m\) terms have the same trace and yields
\[
\partial_q\Tr\widetilde R_{\mc F,t}(\lambda)^m
=m\Tr\left(
\widetilde R_{\mc F,t}(\lambda)^m
(\partial_qA_t)
\widetilde R_{\mc F,t}(\lambda)
\right).
\]
Taking the same complex linear combination gives the asserted
formula for \(\partial_t\). See also \cite{Simon2005} for these
trace-ideal properties.

Finally, put \(J_t:=dA_{\widetilde g_t}/dA_g\).
By metric equivalence, \(J_t^{\pm1/2}\) are bounded multipliers, and
\[
U_t:=M_{J_t^{1/2}}:
L^2(X',dA_{\widetilde g_t})
\longrightarrow\mathscr H
\]
is unitary. Under the vector-space identification used above,
\(U_t\) is also a bounded invertible multiplier on \(\mathscr H\).
Thus, if \(T\) is trace class on the varying \(L^2\)-space, unitary
invariance and then invariance under bounded similarity give
\[
\Tr_{L^2(dA_{\widetilde g_t})}T
=\Tr_{\mathscr H}(U_tTU_t^{-1})
=\Tr_{\mathscr H}T.
\]
In the last expression, \(T\) denotes the same algebraic operator
after the underlying vector space of
\(L^2(X',dA_{\widetilde g_t})\) has been identified with
\(\mathscr H\).
Thus the trace computed on the fixed space agrees with the trace in
the varying \(L^2\)-space.
\end{proof}

Set
\[
\Delta_{\mc F}:=\Delta_{X,\mc F},
\qquad
\dot{\widetilde\Delta}_{\mc F}
:=
\left.
\frac{\partial\widetilde\Delta_{t,\mc F}}{\partial t}
\right|_{t=0},
\qquad
R_{\mc F}(\lambda)
:=
\widetilde R_{\mc F,0}(\lambda).
\]
By Lemma~\ref{lem:resolvent-differentiation},
\[
\left.
\frac{\partial}{\partial t}
\Tr\widetilde R_{\mc F,t}(\lambda)^m
\right|_{t=0}
=
m\Tr\left(
\dot{\widetilde\Delta}_{\mc F}
R_{\mc F}(\lambda)^{m+1}
\right).
\]

To evaluate this trace, let
\(\{(\lambda_i^{\mc F},\phi_i^{\mc F})\}_{i\geq0}\)
be a complete orthonormal system of eigenpairs of
\(\Delta_{\mc F}\) in
\(L^2(X',dA_{g})\). Then
\[
\Tr\left(
\dot{\widetilde\Delta}_{\mc F}
R_{\mc F}(\lambda)^{m+1}
\right)
=
\sum_i
\left\langle
\dot{\widetilde\Delta}_{\mc F}
R_{\mc F}(\lambda)^{m+1}\phi_i^{\mc F},
\phi_i^{\mc F}
\right\rangle_{L^2(X',dA_{g})}.
\]
Since
\(
R_{\mc F}(\lambda)\phi_i^{\mc F}
=(\lambda-\lambda_i^{\mc F})^{-1}\phi_i^{\mc F},
\)
we obtain
\[
\Tr\left(
\dot{\widetilde\Delta}_{\mc F}
R_{\mc F}(\lambda)^{m+1}
\right)
=
\sum_i
\frac{
\left\langle
\dot{\widetilde\Delta}_{\mc F}\phi_i^{\mc F},
\phi_i^{\mc F}
\right\rangle_{L^2(X',dA_{g})}
}{
(\lambda-\lambda_i^{\mc F})^{m+1}
}.
\]
Having established differentiability of the resolvent trace, we now
turn to the local geometric computation of the infinitesimal operator
variation. We first introduce the vector field that generates the
deformation on the fixed surface and then compute
\(
\left\langle
\dot{\widetilde\Delta}_{\mc F}u,v
\right\rangle_{L^2(X',dA_{g})}
\)
for \(u,v\in\mc D_{\mc F}\). Since the variation is supported in fixed neighborhoods of the
ramification points, it suffices to compute the corresponding local
contributions. More precisely, on \(X_k^o\),
we have \(\Theta_t=\id\) and \(\widetilde g_t=g\).

Let
\(\widetilde\varphi_t:=\varphi_t\circ\Theta_t:
X\to\mb P^1\).
Thus \(\widetilde\varphi_t\) is the family of smooth maps obtained from
\(\varphi_t:X_t\to\mb P^1\) after identifying the varying surfaces \(X_t\)
with the fixed smooth surface \(X\) through the trivialization
\(\Theta_t\).

On the punctured surface \(X'=X\setminus S\), the map \(\varphi\)
is unramified. Hence
\(d\varphi:T_PX'\to T_{\varphi(P)}\mb P^1\) is an isomorphism for
every \(P\in X'\).
We extend \(d\varphi\) complex linearly when applying it to complex
vector fields. We define the
infinitesimal deformation vector field \(\mc X\) on \(X'\) by
\[
d\varphi(\mc X)
=
\left.
\frac{\partial}{\partial t}
\widetilde\varphi_t
\right|_{t=0}
=
\left.
\frac{\partial}{\partial t}
(\varphi_t\circ\Theta_t)
\right|_{t=0}.
\]
For each fixed \(P\in X'\), the map \(t\mapsto\widetilde\varphi_t(P)\)
is holomorphic, so \(\mc X\) is of type \((1,0)\).
Thus \(\mc X\) is the infinitesimal lift to the fixed surface
\(X'\) of the deformation of the meromorphic maps
\(\varphi_t\), regarded as a section of the complexified tangent bundle
\(TX'\otimes_{\mb R}\mb C\).

We now compute the local expression of \(\mc X\). On
\(U_{kj}'\), write \(\zeta=z_{kj}\). Since
\(w_k\circ\varphi_t=t+z_{kj,t}^{n_{kj}}\) and
\(z_{kj,t}\circ\Theta_t=H_{kj,t}\circ z_{kj}\), we have
\[
w_k\circ\widetilde\varphi_t
=
t+H_{kj,t}(\zeta)^{n_{kj}}
=
\zeta^{n_{kj}}
+
\bigl(1-\beta(|\zeta|^{n_{kj}})\bigr)t.
\]
Hence
\(\left. \frac{\partial}{\partial t} (w_k\circ\widetilde\varphi_t) \right|_{t=0} = 1-\beta(|\zeta|^{n_{kj}})\).
Since \(w_k\circ\varphi=\zeta^{n_{kj}}\), we have
\(d(w_k\circ\varphi)
=n_{kj}\zeta^{n_{kj}-1}d\zeta\).
Writing the local expression of \(\mc X\) on
\(U_{kj}'\) as
\(\mc X=b_{kj}(\zeta)\partial/\partial\zeta\), the definition of
\(\mc X\) gives
\(n_{kj}\zeta^{n_{kj}-1}b_{kj}(\zeta)
=1-\beta(|\zeta|^{n_{kj}})\). Therefore, on
\(U_{kj}'\),
\[
\mc X
=
\frac{1-\beta(|\zeta|^{n_{kj}})}
{n_{kj}\zeta^{n_{kj}-1}}
\frac{\partial}{\partial\zeta}.
\]

Recall the local differential expression \(\Delta_t^{(kj)}\)
defined above. Its pullback to the fixed domain \(U_{kj}'\) is
\[
\widetilde\Delta_t
=
\widehat H_{kj,t}^*
\Delta_t^{(kj)}
(\widehat H_{kj,t}^*)^{-1}
\qquad\text{on }U_{kj}'.
\]

Using
\(H_{kj,t}(\zeta)^{n_{kj}}
=\zeta^{n_{kj}}-\beta(|\zeta|^{n_{kj}})t\)
and differentiating the pullback identity at \(t=0\), a direct
computation gives, for every smooth function \(u\) compactly
supported in \(U_{kj}'\),
\[
\left.
\frac{\partial}{\partial t}
\bigl(\widetilde\Delta_{t,c}u\bigr)
\right|_{t=0}
=
\mc X\bigl(\Delta_X^{(kj)}u\bigr)
-
\Delta_X^{(kj)}(\mc Xu).
\]
Hence
\[
\dot{\widetilde\Delta}
=
[\mc X,\Delta_X^{(kj)}]
\qquad\text{on }U_{kj}',
\]
where \([A,B]=AB-BA\).
Here \(\dot{\widetilde\Delta}\) denotes the derivative at \(t=0\)
of the differential expression on \(X'\); on the common Friedrichs
domain, its action is denoted by \(\dot{\widetilde\Delta}_{\mc F}\).

\begin{lem}\label{lem:Xphi-max}
Let \(\phi\in\mc D_{\mc F}\) satisfy
\(\Delta_{\mc F}\phi=\lambda\phi\). Then
\(\mc X\phi\in\mc D_{\max}(\Delta_X)\). In particular,
\(\pi_S(\mc X\phi)\) and the Green form
\(\mk q_X(\mc X\phi,\phi)\) are well defined.
\end{lem}

\begin{proof}
The vector field \(\mc X\) is smooth on \(X'\) and vanishes
outside the fixed deformation region \(K_k\). It also vanishes near conic
points not lying over \(Q_k\). Thus it is enough to work near a
ramification point \(P_{kj}\).
Write \(n=n_{kj}\) and \(\zeta=z_{kj}\). On \(U_{kj}'\),
we have
\(\mc X=(1-\beta(|\zeta|^n))(n\zeta^{n-1})^{-1}
\partial_\zeta\) and
\(g=\rho(\zeta,\overline\zeta)|\zeta|^{2n-2}|d\zeta|^2\),
where \(\rho\) is smooth and strictly positive. Since the operator
domain of the nonnegative Friedrichs extension is contained in its
quadratic-form domain, the local Euclidean Dirichlet integral of
\(\phi\) is finite. Consequently,
\[
\begin{aligned}
\int_{U_{kj}'}|\mc X\phi|^2\,dA_g
&=\frac1{n^2}\int_{U_{kj}'}
\rho\,|1-\beta(|\zeta|^n)|^2|\partial_\zeta\phi|^2
\,\frac{\sqrt{-1}}2\,d\zeta\wedge d\overline\zeta\\
&\leq C\int_{U_{kj}'}|\partial_\zeta\phi|^2
\,\frac{\sqrt{-1}}2\,d\zeta\wedge d\overline\zeta
<\infty.
\end{aligned}
\]
Outside neighborhoods of these finitely many ramification points,
\(\mc X\phi\) is smooth and has compact support in \(X'\).
Thus \(\mc X\phi\in L^2(X',dA_g)\).

Theorem~\ref{thm:laplacian-family-differentiability} gives
\(\dot{\widetilde\Delta}_{\mc F}\phi\in L^2(X',dA_g)\)
by differentiability on the common graph domain.
The commutator identity on \(X'\) therefore yields, in the
distributional sense,
\[
\dot{\widetilde\Delta}_{\mc F}\phi
=
[\mc X,\Delta_X]\phi
=
\lambda\mc X\phi-\Delta_X(\mc X\phi).
\]
Hence \(\Delta_X(\mc X\phi)=\lambda\mc X\phi-
\dot{\widetilde\Delta}_{\mc F}\phi\in L^2\).
By the definition of the maximal domain,
\(\mc X\phi\in\mc D_{\max}(\Delta_X)\).
\end{proof}

The commutator identity and Lemma~\ref{lem:Xphi-max} give
\(\langle\dot{\widetilde\Delta}_{\mc F}\phi_i^{\mc F},
\phi_i^{\mc F}\rangle
=-\mk q_X(\mc X\phi_i^{\mc F},\phi_i^{\mc F})\).
Only the ramification points over \(Q_k\) contribute, since
\(\mc X\) vanishes near the other conic points. Hence
\[
\left\langle\dot{\widetilde\Delta}_{\mc F}\phi_i^{\mc F},
\phi_i^{\mc F}\right\rangle
=-\sum_{\substack{1\leq j\leq s_k\\ n_{kj}>1}}
\omega_{P_{kj}}\bigl(
\pi_{P_{kj}}(\mc X\phi_i^{\mc F}),
\pi_{P_{kj}}\phi_i^{\mc F}\bigr).
\]

For each ramification point \(P_{kj}\), let
\[
\gamma_{kj}
=
(f_{kj,\ell})_{\ell\in J_{n_{kj}}},
\qquad
\gamma_{kj}^{\#}
=
(f_{kj,\ell}^{\#})_{\ell\in J_{n_{kj}}}
\]
be the normalized asymptotic bases induced by the coordinate
\(\zeta=z_{kj}\), as in Sections~\ref{sec:conic-analysis}
and~\ref{sec:local-model-comparison}.
Write \(\mc F_{X,P_{kj}}:=\operatorname{span}_{\mb C}\gamma_{kj}\)
for the local regular critical subspace. Write
\[
\pi_{P_{kj}}\phi_i^{\mc F}
=
\sum_{\ell\in J_{n_{kj}}}
c_{i,kj,\ell}\,f_{kj,\ell}.
\]
Since \(\phi_i^{\mc F}\) satisfies the Friedrichs boundary
condition, only the regular basis elements occur in this expansion.

Set \(n=n_{kj}\). Near \(P_{kj}\), the cutoff vanishes, so
\(\mc X=(n\zeta^{n-1})^{-1}\partial_\zeta\).
For \(1\leq p<n\), the normalizations
\(f_{kj,p}=\zeta^p/\sqrt{4\pi p}\) and
\(f_{kj,p-n}^{\#}=\zeta^{p-n}/\sqrt{4\pi(n-p)}\) give
\[
\mc Xf_{kj,p}
=\frac{\sqrt{p(n-p)}}{n}f_{kj,p-n}^{\#},
\qquad
\mc Xf_{kj,-p}=\mc Xf_{kj,0}=0.
\]
To justify extracting critical data after differentiation, we use
the local regularity of Friedrichs eigenfunctions.
The finite Dirichlet energy of \(\phi_i^{\mc F}\), together with
its \(L^2\)-integrability on annuli away from the origin, gives local
Euclidean \(H^1\)-regularity across \(\zeta=0\).
Choose logarithmic cutoffs that vanish near the origin, converge to
\(1\) away from it, and have Euclidean Dirichlet energies tending
to zero. Testing the eigenfunction equation with these cutoffs
times a smooth test function, the cutoff-gradient terms tend to
zero by the Cauchy--Schwarz inequality. The equation therefore
extends weakly across \(\zeta=0\) as
\(\partial_\zeta\partial_{\overline\zeta}\phi_i^{\mc F}
=-\tfrac14\lambda_i^{\mc F}\rho(\zeta,\overline\zeta)
|\zeta|^{2n-2}\phi_i^{\mc F}\), where
\(g=\rho(\zeta,\overline\zeta)|\zeta|^{2n-2}|d\zeta|^2\)
and \(\rho\) is smooth and positive.
Elliptic regularity for this equation implies that
\(\phi_i^{\mc F}\) is smooth at \(\zeta=0\).
Its Taylor polynomial of degree \(2n-1\) has no mixed terms,
because the right-hand side vanishes to order at least \(2n-2\).
Thus we may write
\[
\phi_i^{\mc F}(\zeta)
=a_0+\sum_{r=1}^{2n-1}
\bigl(a_r\zeta^r+b_r\overline\zeta^{\,r}\bigr)
+O(|\zeta|^{2n}),
\]
with the corresponding differentiated remainder of order
\(O(|\zeta|^{2n-1})\).
Applying \(\mc X=(n\zeta^{n-1})^{-1}\partial_\zeta\), the terms
with \(1\leq r<n\) give the singular critical modes; those with
\(n\leq r<2n\) give regular critical modes, and the remainder is
\(O(|\zeta|^n)\). Since
\(\mc X\phi_i^{\mc F}\in\mc D_{\max}(\Delta_X)\) by
Lemma~\ref{lem:Xphi-max}, this identifies its critical asymptotic
data. In particular,
\[
\pi_{P_{kj}}(\mc X\phi_i^{\mc F})
\equiv
\frac{1}{n_{kj}}
\sum_{p=1}^{n_{kj}-1}
\sqrt{p(n_{kj}-p)}\,
c_{i,kj,p}\,
f_{kj,p-n_{kj}}^{\#}
\pmod{\mc F_{X,P_{kj}}}.
\]
The omitted regular terms pair trivially with the Friedrichs
data. By the normalized boundary pairing, we obtain
\[
\omega_{P_{kj}}\bigl(
\pi_{P_{kj}}(\mc X\phi_i^{\mc F}),
\pi_{P_{kj}}\phi_i^{\mc F}\bigr)
=-\frac{1}{n_{kj}}\sum_{p=1}^{n_{kj}-1}
\sqrt{p(n_{kj}-p)}\,
c_{i,kj,p}\,\overline{c_{i,kj,p-n_{kj}}}.
\]
Consequently,
\[
\left\langle
\dot{\widetilde\Delta}_{\mc F}\phi_i^{\mc F},
\phi_i^{\mc F}
\right\rangle
=
\sum_{\substack{1\le j\le s_k\\ n_{kj}>1}}
\frac{1}{n_{kj}}
\sum_{p=1}^{n_{kj}-1}
\sqrt{p(n_{kj}-p)}\,
c_{i,kj,p}\,
\overline{c_{i,kj,p-n_{kj}}}.
\]
Write \(M_{kj}^X(\lambda):=M_{P_{kj},P_{kj}}^X(\lambda)\) for the
\((P_{kj},P_{kj})\)-diagonal block of \(M^X(\lambda)\).
Its matrix representation with respect to the domain basis
\(\gamma_{kj}^{\#}\) and the codomain basis \(\gamma_{kj}\) is the
local \(S\)-matrix
\[
S_{kj}^{X,(\gamma_{kj},\gamma_{kj}^{\#})}(\lambda)
=
\left(
S_{kj,\ell p}^{X,(\gamma_{kj},\gamma_{kj}^{\#})}(\lambda)
\right)_{\ell,p\in J_{n_{kj}}}.
\]
Equivalently, for each \(p\in J_{n_{kj}}\),
\[
M_{kj}^X(\lambda)f_{kj,p}^{\#}
=
\sum_{\ell\in J_{n_{kj}}}
S_{kj,\ell p}^{X,(\gamma_{kj},\gamma_{kj}^{\#})}(\lambda)
f_{kj,\ell}.
\]
We derive the spectral representation in the boundary normalization
used here. Let \(u_{kj,p}(\lambda)\) be the Weyl solution whose
singular boundary datum is \(f_{kj,p}^{\#}\), with all other
singular coefficients zero. The normalized Green pairing gives
\(\mk q_X(u_{kj,p}(\lambda),\phi_i^{\mc F})
=-\overline{c_{i,kj,p}}\). Since
\(\Delta_Xu_{kj,p}(\lambda)=\lambda u_{kj,p}(\lambda)\), it follows
that
\(\langle u_{kj,p}(\lambda),\phi_i^{\mc F}\rangle
=\overline{c_{i,kj,p}}/(\lambda_i^{\mc F}-\lambda)\).
Moreover, \(\partial_\lambda u_{kj,p}(\lambda)\) has zero singular
boundary datum and satisfies
\((\Delta_X-\lambda)\partial_\lambda u_{kj,p}(\lambda)
=u_{kj,p}(\lambda)\). Pairing with
\(u_{kj,\ell}(\overline\lambda)\) therefore gives
\(\partial_\lambda
S_{kj,\ell p}^{X,(\gamma_{kj},\gamma_{kj}^{\#})}(\lambda)
=\langle u_{kj,p}(\lambda),u_{kj,\ell}(\overline\lambda)\rangle\).
Expanding both Weyl solutions in the orthonormal eigenbasis and
differentiating yields, for each integer \(m\geq1\),
\[
\frac{d^m}{d\lambda^m}
S_{kj,\ell p}^{X,(\gamma_{kj},\gamma_{kj}^{\#})}(\lambda)
=m!\sum_i
\frac{c_{i,kj,\ell}\,\overline{c_{i,kj,p}}}
{(\lambda_i^{\mc F}-\lambda)^{m+1}}.
\]
The interchange of differentiation and summation is standard, but
we briefly indicate the required domination. Fix \(\lambda_*<0\).
Parseval's identity applied to \(u_{kj,a}(\lambda_*)\) shows that
the sequence
\(\bigl(c_{i,kj,a}/(\lambda_i^{\mc F}-\lambda_*)\bigr)_i\)
belongs to \(\ell^2\) for every boundary index \(a\). On each compact
subset \(K\) of the resolvent set, the quotients
\((\lambda_i^{\mc F}-\lambda_*)/
(\lambda_i^{\mc F}-\lambda)\) are uniformly bounded for
\(\lambda\in K\) and all \(i\). Hence Cauchy--Schwarz gives locally
uniform absolute convergence of the series for
\(\partial_\lambda S_{kj,\ell p}\). Since
\(\operatorname{dist}(K,\Spec\Delta_{X,\mc F})>0\), every additional
power of \((\lambda_i^{\mc F}-\lambda)^{-1}\) is uniformly bounded
on \(K\). The same argument therefore gives normal convergence of
all the differentiated series, and termwise holomorphic
differentiation yields the displayed identity for every \(m\geq1\).
Reindex the boundary pairing by replacing \(p\) with
\(n_{kj}-p\). Combining this identity with
Lemma~\ref{lem:resolvent-differentiation} and
\((\lambda-\lambda_i^{\mc F})^{-(m+1)}
=(-1)^{m+1}(\lambda_i^{\mc F}-\lambda)^{-(m+1)}\)
yields the trace variation formula
\begin{equation}\label{eq:trace-resolvent-variation}
\left.\frac{\partial}{\partial z_k}
\Tr R_{\mc F}(\lambda)^m\right|_{[\varphi]}
=\frac{(-1)^{m+1}}{(m-1)!}
\sum_{\substack{1\leq j\leq s_k\\ n_{kj}>1}}
\frac{1}{n_{kj}}\sum_{p=1}^{n_{kj}-1}
\sqrt{p(n_{kj}-p)}\,
\frac{d^m}{d\lambda^m}
S_{kj,n_{kj}-p,-p}^{X,(\gamma_{kj},\gamma_{kj}^{\#})}(\lambda),
\end{equation}
valid for every integer \(m>1\). Here the Hurwitz derivative
acts on the local trace function and is evaluated at
\([\varphi]\); it agrees with the \(t\)-derivative of the
pulled-back trace at zero by unitary invariance. Since the base
point was arbitrary, the formula holds throughout the coordinate
ball wherever \(\lambda\) lies in the resolvent set.

The same operator-theoretic conclusions hold for the simultaneous
deformation \(\mf t\in\mb C^N\). Indeed, the trivializations act in
disjoint deformation neighborhoods, and the coefficient estimates in
Theorem~\ref{thm:laplacian-family-differentiability} apply to each
first real Hurwitz coordinate derivative. Thus
\(\mf t\mapsto\widetilde\Delta_{\mf t,\mc F}\) is \(C^1\) from the
coordinate ball to \(\mathcal L(\mathscr D,\mathscr H)\). The
resolvent factorization then gives \(C^1\)-dependence in both
\(\mathcal L(\mathscr H,\mathscr D)\) and Hilbert--Schmidt norm,
locally uniformly on compact subsets of a common resolvent set. The
corresponding resolvent powers are \(C^1\) in trace norm.

\section{The \(S\)-Matrix at Zero Energy}\label{sec:zero-energy}
Since \(\Delta_{X,\mc F}\) is nonnegative, we have
\((-\infty,0)\subset\rho(\Delta_{X,\mc F})\). In what follows, the
spectral parameter \(\lambda\) is restricted to the negative real axis.
The identification of the zero-energy coefficients with the
Schiffer bidifferential follows the computation of Hillairet,
Kalvin, and Kokotov \cite{HillairetKalvinKokotov2018}, which also
establishes the holomorphic continuation of the relevant nonzero-mode
blocks to a neighborhood of zero. We give a self-contained
zero-energy limiting argument in the boundary-data normalization and
notation used here as \(\lambda\to0^-\).

Fix a ramification point \(P_{kj}\) lying over \(Q_k\), with
\(n_{kj}>1\), and, for simplicity,
write \(P_k:=P_{kj}\), \(n_k:=n_{kj}\), \(\gamma_k:=\gamma_{kj}\), and
\(\gamma_k^{\#}:=\gamma_{kj}^{\#}\) throughout this section.
Likewise, write \(f_{k,r}:=f_{kj,r}\) and
\(f_{k,r}^{\#}:=f_{kj,r}^{\#}\) for \(r\in J_{n_k}\).
All local expansions below use the distinguished coordinate
\(z:=z_{kj}\) at \(P_k\).
We first prove zero-energy convergence in the maximal graph norm.
We then recall the construction of harmonic representatives with
prescribed singular data using the Schiffer bidifferential. Their
expansions determine the limiting \(S\)-matrix entries, including
those in the resolvent variation.

Let \(P_0:L^2(X',dA_g)\to\ker\Delta_{X,\mc F}\) be the orthogonal
projection onto the zero eigenspace. It is known that for every
\(F\in L^2(X',dA_g)\),
\[
P_0F
=
\frac{1}{\operatorname{Area}(X,g)}
\int_X F\,dA_g.
\]

\begin{lem}
Let \(F\in\mc D_{\max}(\Delta_X)\) satisfy
\(\Delta_XF=0\) on \(X'\), and set
\(F(\lambda):=Y_{X,\mc F}(\lambda)F\). Then
\[
F(\lambda)\longrightarrow (I-P_0)F
\qquad\text{as }\lambda\to0^-
\]
in \(\mc D_{\max}(\Delta_X)\) with respect to the graph norm.
\end{lem}

\begin{proof}
Since \(\Delta_XF=0\) on \(X'\), the definition of
\(Y_{X,\mc F}(\lambda)\) gives
\(F(\lambda)=F-\lambda R_{X,\mc F}(\lambda)F\).
Write
\(R_{X,\mc F}^{\perp}(\lambda)
=R_{X,\mc F}(\lambda)-\lambda^{-1}P_0\). Then
\(F(\lambda)-(I-P_0)F
=-\lambda R_{X,\mc F}^{\perp}(\lambda)F\).
Let \(\lambda_1^{\mc F}>0\) denote the first positive eigenvalue of
\(\Delta_{X,\mc F}\). For \(\lambda<0\),
\[
\|R_{X,\mc F}^{\perp}(\lambda)\|_{\mathcal L(L^2(X'))}
\leq \frac{1}{\lambda_1^{\mc F}-\lambda}.
\]
Hence
\[
\|F(\lambda)-(I-P_0)F\|_{L^2(X')}
\leq
\frac{|\lambda|}{\lambda_1^{\mc F}-\lambda}
\|F\|_{L^2(X')}
\longrightarrow0.
\]
Moreover, \(\Delta_XF(\lambda)=\lambda F(\lambda)\) and
\(\Delta_X(I-P_0)F=0\). Since
\[
\|F(\lambda)\|_{L^2(X')}
\leq
\left(1+\frac{|\lambda|}{\lambda_1^{\mc F}-\lambda}\right)
\|F\|_{L^2(X')},
\]
it follows that
\[
\|\Delta_X(F(\lambda)-(I-P_0)F)\|_{L^2(X')}
=
|\lambda|\,\|F(\lambda)\|_{L^2(X')}
\longrightarrow0.
\]
Therefore \(F(\lambda)\to(I-P_0)F\) in the graph norm of
\(\mc D_{\max}(\Delta_X)\).
\end{proof}

\begin{cor}
Let \(F\in\mc D_{\max}(\Delta_X)\) satisfy
\(\Delta_XF=0\) on \(X'\), and let \(F(\lambda)\) be defined as above.
Then
\[
\lim_{\lambda\to0^-}\pi_S(F(\lambda))
=
\pi_S((I-P_0)F).
\]
Equivalently, for every \(P\in S\),
\[
\lim_{\lambda\to0^-}\pi_P(F(\lambda))
=
\pi_P((I-P_0)F).
\]
\end{cor}

\begin{proof}
By the preceding lemma,
\(F(\lambda)\to(I-P_0)F\) in
\(\mc D_{\max}(\Delta_X)\) with respect to the graph norm.
The result follows from the continuity of \(\pi_S\) with respect to
the graph norm, established in Section~\ref{sec:conic-analysis}.
The equivalent statement for \(\pi_P\), \(P\in S\), follows by
taking the corresponding component.
\end{proof}

For \(r\in J_{n_k}\setminus\{0\}\), suppose that
\(F_{k,r}\in\mc D_{\max}(\Delta_X)\) satisfies
\(\Delta_XF_{k,r}=0\) on \(X'\) and
\(\widetilde\pi_{\mc F^\#}(F_{k,r})=f_{k,r}^\#\). Write
\[
\pi_{P_k}F_{k,r}
=
f_{k,r}^\#
+
\sum_{\ell\in J_{n_k}}a_{k,\ell,r}f_{k,\ell}.
\]
Set
\(F_{k,r}(\lambda):=Y_{X,\mc F}(\lambda)F_{k,r}\). Then
\[
\pi_{P_k}F_{k,r}(\lambda)
=
f_{k,r}^\#
+
\sum_{\ell\in J_{n_k}}
S_{k,\ell,r}^{X,(\gamma_k,\gamma_k^\#)}(\lambda)f_{k,\ell}.
\]
By the preceding corollary,
\(\pi_{P_k}F_{k,r}(\lambda)\to
\pi_{P_k}(I-P_0)F_{k,r}\) as \(\lambda\to0^-\).
Since \(P_0F_{k,r}\) is constant, it affects only the
\(f_{k,0}\)-coefficient. Hence, for \(\ell\neq0\),
\[
\lim_{\lambda\to0^-}
S_{k,\ell,r}^{X,(\gamma_k,\gamma_k^\#)}(\lambda)
=
a_{k,\ell,r}.
\]
Moreover, since
\(f_{k,0}=1/\sqrt{2\pi}\) and
\(P_0F_{k,r}
=\operatorname{Area}(X,g)^{-1}\int_XF_{k,r}\,dA_g\),
we obtain
\[
\lim_{\lambda\to0^-}
S_{k,0,r}^{X,(\gamma_k,\gamma_k^\#)}(\lambda)
=
a_{k,0,r}
-
\frac{\sqrt{2\pi}}{\operatorname{Area}(X,g)}
\int_XF_{k,r}\,dA_g.
\]

The preceding discussion yields the following zero-energy formula for the local \(S\)-matrix.

\begin{thm}
Let \(r\in J_{n_k}\setminus\{0\}\), and suppose that
\(F_{k,r}\in\mc D_{\max}(\Delta_X)\) satisfies
\(\Delta_XF_{k,r}=0\) on \(X'\) and
\(\widetilde\pi_{\mc F^\#}(F_{k,r})=f_{k,r}^\#\). If
\(\pi_{P_k}F_{k,r}
=
f_{k,r}^\#
+
\sum_{\ell\in J_{n_k}}a_{k,\ell,r}f_{k,\ell}\),
then
\(\lim_{\lambda\to0^-}
S_{k,\ell,r}^{X,(\gamma_k,\gamma_k^\#)}(\lambda)\)
exists for every \(\ell\in J_{n_k}\), and
\[
\lim_{\lambda\to0^-}
S_{k,\ell,r}^{X,(\gamma_k,\gamma_k^\#)}(\lambda)
=
\begin{cases}
a_{k,\ell,r}, & \ell\neq0,\\[4pt]
a_{k,0,r}
-\displaystyle\frac{\sqrt{2\pi}}{\operatorname{Area}(X,g)}
\int_XF_{k,r}\,dA_g, & \ell=0.
\end{cases}
\]
\end{thm}

Suppose that \(F_{k,r},\widetilde F_{k,r}\in
\mc D_{\max}(\Delta_X)\) satisfy
\(\Delta_XF_{k,r}=\Delta_X\widetilde F_{k,r}=0\) on \(X'\) and
\[
\widetilde\pi_{\mc F^\#}(F_{k,r})
=
\widetilde\pi_{\mc F^\#}(\widetilde F_{k,r})
=
f_{k,r}^\#.
\]
Set
\(\widehat F_{k,r}:=F_{k,r}-\widetilde F_{k,r}\).
Then
\(\widehat F_{k,r}\in\mc D_{\max}(\Delta_X)\),
\(\Delta_X\widehat F_{k,r}=0\) on \(X'\), and
\(\widetilde\pi_{\mc F^\#}(\widehat F_{k,r})=0\).
Hence
\(\widehat F_{k,r}\in\mc D_{\mc F}\), and therefore
\(
\widehat F_{k,r}\in\ker\Delta_{X,\mc F}.
\)
Since \(X\) is connected,
\(\ker\Delta_{X,\mc F}\) consists of the constant functions.
Thus such harmonic representatives are unique up to an additive
constant. It follows that the coefficients \(a_{k,\ell,r}\) with
\(\ell\neq0\) are independent of the choice of the harmonic
representative \(F_{k,r}\), while the combination
\[
a_{k,0,r}
-
\frac{\sqrt{2\pi}}{\operatorname{Area}(X,g)}
\int_XF_{k,r}\,dA_g
\]
is independent of this choice as well.

We now construct the required harmonic representatives
\(F_{k,r}\), following Hillairet, Kalvin, and Kokotov
\cite[Section~4]{HillairetKalvinKokotov2018}, expressed in the present
coordinates and normalization.

On a sufficiently small neighborhood \(U\) of \(P_k\), use the
distinguished coordinate \(z\) fixed above. For
\(1\le p\le n_k-1\), define the meromorphic differential
\[
\eta_{k,p}
:=
-\frac{1}{(p-1)!}
\left.
\frac{\partial^{p-1}}{\partial\zeta^{p-1}}
\frac{\mc S(\,\cdot\,,\zeta)}{d\zeta}
\right|_{\zeta=0}.
\]
Then \(\eta_{k,p}\) is holomorphic on \(X\setminus\{P_k\}\) and has
a pole of order \(p+1\) at \(P_k\). In the coordinate \(z\),
\[
\eta_{k,p}
=
\left(
-\frac{p}{z^{p+1}}
-
\frac{1}{(p-1)!}
\left.
\frac{\partial^{p-1}}{\partial\zeta^{p-1}}
H_{\mc S}(z,\zeta)
\right|_{\zeta=0}
\right)dz.
\]
We now turn from abstract zero-energy convergence to the explicit
harmonic representative needed for the Schiffer calculation. We
construct a single-valued function \(G_{k,p}\), harmonic on
\(X\setminus\{P_k\}\), such that
\[
\partial G_{k,p}=\eta_{k,p}.
\]
Since \(\eta_{k,p}\) has principal part
\(-pz^{-(p+1)}dz\) at \(P_k\), the function \(G_{k,p}\) has a local
expansion of the form
\[
G_{k,p}(z)
=
z^{-p}
+
h_{k,p}(z)
+
\overline{g_{k,p}(z)},
\]
where \(h_{k,p}\) and \(g_{k,p}\) are holomorphic near \(z=0\).
The regular holomorphic part of \(G_{k,p}\) is determined directly
by \(\eta_{k,p}\). 

As above, write \(C=\operatorname{Im}B\) and
\(C^{ij}:=(C^{-1})_{ij}\). Since
\(\mc S=W-\pi\sum_{i,j=1}^g C^{ij}\omega_i\otimes\omega_j\),
define
\[
\tau_{k,p}
:=
\frac{1}{(p-1)!}
\left.
\frac{\partial^{p-1}}{\partial\zeta^{p-1}}
\frac{W(\,\cdot\,,\zeta)}{d\zeta}
\right|_{\zeta=0},
\qquad
v_{j;k,p}
:=
\frac{1}{(p-1)!}
\left.
\frac{\partial^{p-1}}{\partial\zeta^{p-1}}
\frac{\omega_j(\zeta)}{d\zeta}
\right|_{\zeta=0}.
\]
The same formula defines the jet \(v_{j;k,q}\) for every integer
\(q\geq1\).
Then
\[
\eta_{k,p}
=
-\tau_{k,p}
+
\pi\sum_{i,j=1}^g
C^{ij}v_{j;k,p}\omega_i.
\]

By the normalization of the canonical bidifferential \(W\),
\(\int_{\alpha_\ell}W(\,\cdot\,,Q)=0\) and
\(\int_{\beta_\ell}W(\,\cdot\,,Q)=2\pi\sqrt{-1}\,\omega_\ell(Q)\).
Therefore,
\[
\int_{\alpha_\ell}\tau_{k,p}=0,
\qquad
\int_{\beta_\ell}\tau_{k,p}
=
2\pi\sqrt{-1}\,v_{\ell;k,p}.
\]
It follows that
\[
\begin{aligned}
\int_{\alpha_\ell}\eta_{k,p}
&=
\pi\sum_{j=1}^g
C^{\ell j}v_{j;k,p},
\\
\int_{\beta_\ell}\eta_{k,p}
&=
-2\pi\sqrt{-1}\,v_{\ell;k,p}
+
\pi\sum_{i,j=1}^g
C^{ij}v_{j;k,p}B_{i\ell}.
\end{aligned}
\]
Let
\(
\theta_{k,p}
=
\eta_{k,p}
-
\pi\sum_{i,j=1}^g
C^{ij}v_{j;k,p}\overline{\omega_i}.
\)
Then \(\theta_{k,p}\) is a smooth closed \(1\)-form on
\(X\setminus\{P_k\}\), and, for every \(\ell=1,\ldots,g\),
\[
\int_{\alpha_\ell}\theta_{k,p}=0,
\qquad
\int_{\beta_\ell}\theta_{k,p}=0.
\]
Moreover, since \(\eta_{k,p}\) has zero residue at \(P_k\),
the integral of \(\theta_{k,p}\) around a sufficiently small
positively oriented circle centered at \(P_k\) also vanishes.

Define \(G_{k,p}:X\setminus\{P_k\}\to\mb C\) by
\(G_{k,p}(Q)=\int_{P_*}^{Q}\theta_{k,p}\).
Since all periods of \(\theta_{k,p}\) vanish, the
integral is independent of the choice of path and \(G_{k,p}\) is
well-defined. Moreover, \(dG_{k,p}=\theta_{k,p}\).
Since \(\eta_{k,p}\) is holomorphic and
\(
-\pi\sum_{i,j=1}^g
C^{ij}v_{j;k,p}\overline{\omega_i}
\)
is antiholomorphic, it follows that
\[
\partial G_{k,p}=\eta_{k,p},
\qquad
\bar\partial G_{k,p}
=
-\pi\sum_{i,j=1}^g
C^{ij}v_{j;k,p}\overline{\omega_i}.
\]
\begin{lem}
For \(1\le p\le n_k-1\), the function \(G_{k,p}\) is harmonic on
\(X\setminus\{P_k\}\) and has principal part \(z^{-p}\) at \(P_k\).
Moreover, near \(P_k\),
\[
G_{k,p}
=
z^{-p}+c_{k,p}
+\sum_{\ell\geq1}
\left(
c_{k,p;\ell}z^\ell+d_{k,p;\ell}\bar z^\ell
\right),
\]
where
\[
\begin{aligned}
c_{k,p;\ell}
&=
-\frac{1}{\ell!(p-1)!}
\left.
\frac{\partial^{\ell+p-2}}
{\partial z^{\ell-1}\partial\zeta^{p-1}}
H_{\mc S}(z,\zeta)
\right|_{(z,\zeta)=(0,0)},
\\
d_{k,p;\ell}
&=
-\frac{\pi}{\ell}
\sum_{i,j=1}^g
C^{ij}
v_{j;k,p}\overline{v_{i;k,\ell}}.
\end{aligned}
\]
\end{lem}

\begin{proof}
The identities for \(\partial G_{k,p}\) and
\(\bar\partial G_{k,p}\) show that the former is holomorphic and
the latter is antiholomorphic. Hence
\(\partial\bar\partial G_{k,p}=0\) on
\(X\setminus\{P_k\}\), so \(G_{k,p}\) is harmonic there.

The local expression for \(\eta_{k,p}=\partial G_{k,p}\) has leading
term \(-pz^{-p-1}dz\). Integration therefore gives the principal
part \(z^{-p}\), together with an undetermined constant
\(c_{k,p}\). For \(\ell\geq1\), comparison of the coefficient of
\(z^{\ell-1}dz\) in \(\partial G_{k,p}=\eta_{k,p}\) gives
\[
\ell c_{k,p;\ell}
=-\frac{1}{(\ell-1)!(p-1)!}
\left.
\frac{\partial^{\ell+p-2}}
{\partial z^{\ell-1}\partial\zeta^{p-1}}
H_{\mc S}(z,\zeta)
\right|_{(z,\zeta)=(0,0)},
\]
which is the asserted formula for \(c_{k,p;\ell}\).
Near \(P_k\), write
\(\omega_i/dz=\sum_{q\geq1}v_{i;k,q}z^{q-1}\).
The identity for \(\bar\partial G_{k,p}\) then gives
\[
\ell d_{k,p;\ell}
=-\pi\sum_{i,j=1}^g
C^{ij}v_{j;k,p}\overline{v_{i;k,\ell}},
\]
which proves the second coefficient formula.
\end{proof}

Let
\(
H_{\mc S}(z,\zeta)
=
\sum_{r,s\geq0}
C_{k,rs}^{\mc S}z^r\zeta^s
\)
be the Taylor expansion of \(H_{\mc S}\) near \((P_k,P_k)\). Then
\[
C_{k,rs}^{\mc S}
=
\frac{1}{r!s!}
\left.
\frac{\partial^{r+s}H_{\mc S}(z,\zeta)}
{\partial z^r\partial\zeta^s}
\right|_{(z,\zeta)=(0,0)}.
\]
\begin{thm}
For \(1\le p,\ell\le n_k-1\), the zero-energy limits of the local
\(S\)-matrix entries
\(S_{k,\ell,-p}^{X,(\gamma_k,\gamma_k^{\#})}(\lambda)\) and
\(S_{k,-\ell,-p}^{X,(\gamma_k,\gamma_k^{\#})}(\lambda)\)
exist and are given by
\[
\begin{aligned}
\lim_{\lambda\to0^-}
S_{k,\ell,-p}^{X,(\gamma_k,\gamma_k^{\#})}(\lambda)
&=
-\frac{1}{\sqrt{p\ell}}\,
C_{k,\ell-1,p-1}^{\mc S},
\\
\lim_{\lambda\to0^-}
S_{k,-\ell,-p}^{X,(\gamma_k,\gamma_k^{\#})}(\lambda)
&=
-\frac{\pi}{\sqrt{p\ell}}
\sum_{i,j=1}^g
C^{ij}
v_{j;k,p}\overline{v_{i;k,\ell}}.
\end{aligned}
\]
\end{thm}

\begin{proof}
Set
\(
F_{k,-p}
=
(4\pi p)^{-1/2}G_{k,p}.
\)
Since \(1\le p\le n_k-1\), the principal part \(z^{-p}\) is
\(L^2\) with respect to the conic metric near \(P_k\). Hence
\(F_{k,-p}\in L^2(X')\). Since \(\Delta_XF_{k,-p}=0\) on \(X'\), it
follows from the definition of the maximal domain that
\(F_{k,-p}\in\mc D_{\max}(\Delta_X)\). Moreover,
\(
\widetilde{\pi}_{\mc F^{\#}}F_{k,-p}
=
f_{k,-p}^{\#}.
\)
By the preceding zero-energy convergence result, the limiting regular
boundary coefficients are determined by the corresponding coefficients
in the local expansion of \(F_{k,-p}\). Since
\[
G_{k,p}
=
z^{-p}+c_{k,p}
+\sum_{\ell\ge1}
\left(
c_{k,p;\ell}z^\ell
+
d_{k,p;\ell}\bar z^\ell
\right),
\]
the coefficients of \(f_{k,\ell}\) and \(f_{k,-\ell}\) in the boundary
expansion of \(F_{k,-p}\) are respectively
\[
\sqrt{\frac{\ell}{p}}\,c_{k,p;\ell},
\qquad
\sqrt{\frac{\ell}{p}}\,d_{k,p;\ell}.
\]
Substituting the formulas for \(c_{k,p;\ell}\) and \(d_{k,p;\ell}\)
gives the two asserted limits.
\end{proof}

The coefficients \(C^{\mc S}_{k,rs}\) are related to the Schiffer
projective connection by
\(H_{\mc S}(z,z)=\frac16S_{\mathrm{Sch}}(z)\). Therefore, for every
\(m\ge0\),
\[
\sum_{r+s=m}C^{\mc S}_{k,rs}
=
\frac{1}{6m!}
\left.
\frac{d^m}{dz^m}S_{\mathrm{Sch}}(z)
\right|_{z=0}.
\]
In particular, since
\((n_k-p-1)+(p-1)=n_k-2\), we obtain
\[
\sum_{p=1}^{n_k-1}
C^{\mc S}_{k,n_k-p-1,p-1}
=
\frac{1}{6(n_k-2)!}
\left.
\frac{d^{n_k-2}}{dz^{n_k-2}}
S_{\mathrm{Sch}}(z)
\right|_{z=0}.
\]
Thus the sum of the coefficients relevant to the resolvent variation
is determined by the \((n_k-2)\)-jet of the Schiffer projective
connection at \(P_k\).

This completes the zero-energy calculation needed in the determinant
variation. We now combine it with the high-energy comparison and the
resolvent trace formula.

\section{Zeta Determinants on Hurwitz Spaces}\label{sec:hurwitz-determinants}

We now combine the operator-theoretic and complex-analytic results.
After writing the zeta function as a Mellin transform of the reduced
resolvent trace, we use its locally uniform control at zero and infinity.
An endpoint identity then converts the trace variation into the
difference between the zero- and high-energy Weyl coefficients,
which can be integrated in the Hurwitz coordinates.

Recall from the introduction that
\(\mathcal D([\varphi])
=\operatorname{Det}_{\zeta}(\Delta_{[\varphi],\mc F})\), with the
zero eigenvalue omitted \cite{RaySinger1971}. We first verify
regularity of the zeta function at zero.
In the spherical conic coordinates used in the
model comparison, the metric near a ramification point of index
\(n\) has the form
\(4n^2|w|^{2n-2}(1+|w|^{2n})^{-2}|dw|^2\).
It is therefore admissible in the sense of
\cite[Definition~2.1 and Remark~2.2]{Kalvin2021}.
The heat-trace expansion in \cite[Lemma~2.5]{Kalvin2021} has no
\(t^0\log t\) term. Subtracting the constant contribution of the
kernel preserves this property, so the reduced spectral zeta
function is regular at \(s=0\). Thus the determinant above is
well defined. For the underlying conic heat-kernel construction,
see also \cite{Mooers1999}.

Equivalent coverings have unitarily equivalent Friedrichs
Laplacians, so \(\mathcal D\) is independent of the representative.
We compute its logarithmic Hurwitz derivatives from the zero- and
high-energy \(S\)-matrix coefficients, following the variational
approach of
\cite{HillairetKokotov2013,HillairetKalvinKokotov2018,KalvinKokotov2019,Kalvin2019}.

Related determinant comparison formulas for conformal changes of
conic metrics are obtained in \cite{Kalvin2021}; for Polyakov
variation formulas with cone angles below \(2\pi\), see also
\cite{AldanaKirstenRowlett2026}.

Fix a Hurwitz coordinate ball \(\mc B([\varphi],\delta)\).
We use the family \(\varphi_{\mf t}:X_{\mf t}\to\mb P^1\)
with coordinates \(\mf z+\mf t\), and the operators and
trivializations of Section~\ref{sec:resolvent-variation}.
For variation in the \(k\)-th coordinate, write
\(\mf t=t e_k\) and \(\varphi_t=\varphi_{t e_k}\); then
\(\partial_t|_{t=0}=\partial_{z_k}|_{[\varphi]}\).
All \(S\)-matrices below use the distinguished conic coordinates.
We suppress only the basis labels, writing
\[
S_{kj,\ell p}^{X_{\mf t}}(\lambda)
:=
S_{kj,\ell p}^{X_{\mf t},
(\gamma_{kj,\mf t},\gamma_{kj,\mf t}^{\#})}(\lambda).
\]
At \(\mf t=0\), we write \(S_{kj,\ell p}^{X}(\lambda)\).

For each \(\mf t\) in this coordinate ball, set
\(g_{\mf t}:=\varphi_{\mf t}^*ds_{\mathrm{rd}}^2\), and denote the
Friedrichs Laplacian on
\(L^2(X_{\mf t}',dA_{g_{\mf t}})\) by
\(\Delta_{\mf t,\mc F_{\mf t}}\). Let
\(\lambda_i(\mf t)=\lambda_i^{\mc F_{\mf t}}\), \(i\geq0\), be its
eigenvalues, counted with multiplicity. Since \(X_{\mf t}\) is
connected, \(\lambda_0(\mf t)=0\) and
\(\dim\ker\Delta_{\mf t,\mc F_{\mf t}}=1\). Let
\[
P_{0,\mf t}:L^2(X_{\mf t}',dA_{g_{\mf t}})
\longrightarrow
\ker\Delta_{\mf t,\mc F_{\mf t}}
\]
be the orthogonal projection onto the zero eigenspace.

For \(\lambda\in\rho(\Delta_{\mf t,\mc F_{\mf t}})\), define the
reduced resolvent by
\[
R_{\mc F_{\mf t}}^{\perp}(\lambda)
:=
R_{\mc F_{\mf t}}(\lambda)(I-P_{0,\mf t})
:
L^2(X_{\mf t}',dA_{g_{\mf t}})
\longrightarrow
\mc D_{\mc F_{\mf t}}\cap
\bigl(\ker\Delta_{\mf t,\mc F_{\mf t}}\bigr)^\perp .
\]
When \(\mf t=t e_k\), we abbreviate these objects by replacing
\(\mf t\) with \(t\); thus, for example,
\(\Delta_{t,\mc F_t}:=\Delta_{t e_k,\mc F_{t e_k}}\) and
\(R_{\mc F_t}^{\perp}:=R_{\mc F_{t e_k}}^{\perp}\).
For \(r>0\), the spectral decomposition gives
\[
\Tr R_{\mc F_t}^{\perp}(-r)^2
=
\Tr R_{\mc F_t}(-r)^2
-
\frac{1}{r^2}
=
\sum_{i=1}^{\infty}
\frac{1}{(r+\lambda_i(t))^2}.
\]

For \(1<\Re s<2\), absolute convergence allows us to interchange
the sum and the integral. Using the beta-integral identity
\[
\int_0^\infty
\frac{r^{1-s}}{(r+\mu)^2}\,dr
=
\frac{\pi(1-s)}{\sin(\pi s)}\,\mu^{-s},
\qquad \mu>0,
\]
we obtain
\[
\zeta_{\Delta_{t,\mc F_t}}(s)
=
\frac{\sin(\pi s)}{\pi(1-s)}
\int_0^\infty
r^{1-s}
\Tr R_{\mc F_t}^{\perp}(-r)^2\,dr.
\]
We next record the endpoint estimates that allow the Mellin integral
to be differentiated in the Hurwitz coordinates and continued to
\(s=0\).

The endpoint estimates needed below follow from the standard
resolvent estimates used in the preceding sections. We record
briefly why these estimates are uniform in the present
family. After pulling the family back to the fixed surface \(X\),
the metrics are uniformly equivalent and their coefficients vary
smoothly on compact subsets of the Hurwitz coordinate ball. Hence
the Davies--Gaffney estimate is locally uniform in the Hurwitz
coordinates. The same is true of the interior elliptic and annular
estimates.
The local-model comparison therefore gives, for \(q=0,1,2\) and
every \(M>0\),
\[
\partial_r^q\left(
S_{kj,n_{kj}-p,-p}^{X_{\mf t}}(-r)
-S_{kj,n_{kj}-p,-p}^{X_{\mf t}}(-\infty)
\right)=O(r^{-M})
\qquad(r\to\infty),
\]
locally uniformly in \(\mf t\), for \(1\leq k\leq N\),
\(1\leq j\leq s_k\) with \(n_{kj}>1\), and
\(1\leq p<n_{kj}\), where
\[
S_{kj,n_{kj}-p,-p}^{X_{\mf t}}(-\infty)
:=
\frac{\sqrt{p(n_{kj}-p)}}{n_{kj}}
\frac{\overline{z_k+t_k}}{1+|z_k+t_k|^2}.
\]

Near \(r=0\), the zero eigenvalue is separated by writing the
resolvent as the sum of its rank-one zero-mode contribution and a
part holomorphic at zero. It follows that the relevant \(S\)-matrix
entries and their first two \(r\)-derivatives are locally uniformly
bounded, as are the reduced trace
\(\Tr R_{\mc F_{\mf t}}^\perp(-r)^2\) and its first real
Hurwitz-coordinate derivatives. Together with the preceding
high-energy estimates and the trace variation formula, these bounds
justify differentiation under the Mellin integral for
\(1<\Re s<2\), analytic continuation of the resulting identity to
a neighborhood of \(s=0\), and the interchange of Hurwitz
differentiation with differentiation in \(s\) at \(s=0\).
Accordingly, we may differentiate the Mellin integral and apply
\eqref{eq:trace-resolvent-variation} with \(m=2\). We obtain,
for \(1<\Re s<2\),
\[
\begin{aligned}
\left.
\frac{\partial}{\partial t}
\zeta_{\Delta_{t,\mc F_t}}(s)
\right|_{t=0}
&=
-\frac{\sin(\pi s)}{\pi(1-s)}
\sum_{\substack{1\leq j\leq s_k\\ n_{kj}>1}}
\frac{1}{n_{kj}}
\sum_{p=1}^{n_{kj}-1}
\sqrt{p(n_{kj}-p)}\\
&\qquad\times
\int_0^\infty
r^{1-s}
\partial_r^2 S_{kj,n_{kj}-p,-p}^{X}(-r)\,dr.
\end{aligned}
\]

\begin{lem}\label{lem:mellin-endpoints}
Let \(f\in C^\infty((0,\infty);\mb C)\). Assume that the limits
\(f_0:=\lim_{r\to0^+}f(r)\) and
\(f_\infty:=\lim_{r\to\infty}f(r)\) exist. Suppose moreover that
\(f'(r)\) and \(f''(r)\) remain bounded as \(r\to0^+\), and that
\(f(r)=f_\infty+O(r^{-\infty})\) as \(r\to\infty\), with
\(f^{(j)}(r)=O(r^{-\infty})\) for \(j=1,2\). Then
\[
K_f(s)
:=
-\frac{\sin(\pi s)}{\pi(1-s)}
\int_0^\infty r^{1-s}f''(r)\,dr
\]
extends holomorphically to a neighborhood of \(s=0\), with
\(K_f(s)=s(f_\infty-f_0)+O(s^2)\) as \(s\to0\).
\end{lem}

\begin{proof}
The integral defining \(K_f\) converges locally uniformly for
\(\Re s<2\), since \(f''\) is bounded near zero and rapidly
decreasing at infinity. It is holomorphic on this half-plane, and
the prefactor is holomorphic near \(s=0\). Thus \(K_f(0)=0\).
Differentiating at zero and integrating by parts, we obtain
\[
K_f'(0)
=-\int_0^\infty r f''(r)\,dr
=\int_0^\infty f'(r)\,dr
=f_\infty-f_0.
\]
The boundary term \(rf'(r)\) vanishes at both endpoints by the
assumptions. The Taylor expansion at \(s=0\) gives the result.
\end{proof}

Since the relevant entries of the \(S\)-matrix satisfy the assumptions
of the preceding lemma, we apply it to
\(f(r)=S_{kj,n_{kj}-p,-p}^{X}(-r)\).
It follows that
\[
\begin{aligned}
\left.
\frac{\partial}{\partial t}
\zeta_{\Delta_{t,\mc F_t}}(s)
\right|_{t=0}
&=
s
\sum_{\substack{1\leq j\leq s_k\\ n_{kj}>1}}
\frac{1}{n_{kj}}
\sum_{p=1}^{n_{kj}-1}
\sqrt{p(n_{kj}-p)}
\\
&\qquad\times
\left(
S_{kj,n_{kj}-p,-p}
^{X}(-\infty)
-
S_{kj,n_{kj}-p,-p}
^{X}(0)
\right)
+O(s^2).
\end{aligned}
\]

Here
\(S_{kj,n_{kj}-p,-p}^{X}(0)\)
and
\(S_{kj,n_{kj}-p,-p}^{X}(-\infty)\)
denote the limits as \(r\to0^+\) and \(r\to\infty\), respectively.

By the high-energy asymptotics obtained above,
\[
S_{kj,n_{kj}-p,-p}^{X}(-\infty)
=
\frac{\sqrt{p(n_{kj}-p)}}{n_{kj}A_k}\,\overline{z_k}.
\]
Restoring the full point index, write \(C^{\mc S}_{kj,rs}\) for
the Taylor coefficients denoted by \(C^{\mc S}_{k,rs}\) in
Section~\ref{sec:zero-energy}. The zero-energy computation gives
\[
S_{kj,n_{kj}-p,-p}^{X}(0)
=
-\frac{1}{\sqrt{p(n_{kj}-p)}}
C^{\mc S}_{kj,n_{kj}-p-1,p-1}.
\]
Therefore,
\[
\begin{aligned}
\left.
\frac{\partial}{\partial t}
\log\mathcal D([\varphi_t])
\right|_{t=0}
&=
-\sum_{\substack{1\leq j\leq s_k\\ n_{kj}>1}}
\left[
\frac{\overline{z_k}}{n_{kj}^2A_k}
\sum_{p=1}^{n_{kj}-1}p(n_{kj}-p)
+
\frac{1}{n_{kj}}
\sum_{p=1}^{n_{kj}-1}
C^{\mc S}_{kj,n_{kj}-p-1,p-1}
\right].
\end{aligned}
\]

For each ramified \(P_{kj}\), put \(n=n_{kj}\). Using
\(\sum_{p=1}^{n-1}p(n-p)=n(n^2-1)/6\) and
\[
\sum_{p=1}^{n-1}C^{\mc S}_{kj,n-p-1,p-1}
=
\frac{1}{6(n-2)!}
\left.
\frac{d^{n-2}}{dz_{kj}^{\,n-2}}
S_{\mathrm{Sch}}(z_{kj})
\right|_{z_{kj}=0},
\]
we may sum the coefficients using \(A_k=1+|z_k|^2\) and
\(\sum_{j=1}^{s_k}n_{kj}=d\) to obtain
\[
\begin{aligned}
\left.
\frac{\partial}{\partial t}
\log\mathcal D([\varphi_t])
\right|_{t=0}
&=
-\frac{1}{6}
\left(d-\sum_{j=1}^{s_k}\frac{1}{n_{kj}}\right)
\frac{\overline{z_k}}{1+|z_k|^2}
\\
&\quad
-\sum_{\substack{1\leq j\leq s_k\\ n_{kj}>1}}
\frac{1}{6n_{kj}(n_{kj}-2)!}
\left.
\frac{d^{n_{kj}-2}}{dz_{kj}^{\,n_{kj}-2}}
S_{\mathrm{Sch}}(z_{kj})
\right|_{z_{kj}=0}.
\end{aligned}
\]

This is the \(k\)-th Hurwitz coordinate derivative at
\([\varphi]\), since \(\left.\partial_t\right|_{t=0}
=\left.\partial_{z_k}\right|_{[\varphi]}\).
By the relation established above,
\[
\left.
\frac{1}{6n_{kj}(n_{kj}-2)!}
\frac{d^{n_{kj}-2}}{dz_{kj}^{\,n_{kj}-2}}
S_{\mathrm{Sch}}(z_{kj})
\right|_{z_{kj}=0}
=
\operatorname{Res}_{P_{kj}}
\frac{\mc S_{\varphi}}{d_k\varphi},
\]
while
\[
\frac{\partial}{\partial z_k}\log\tau_{\mathrm{Sch}}
=
-\sum_{\substack{1\leq j\leq s_k\\ n_{kj}>1}}
\operatorname{Res}_{P_{kj}}
\frac{\mc S_{\varphi}}{d_k\varphi}.
\]
Therefore,
\[
\left.
\frac{\partial}{\partial z_k}\log\mathcal D
\right|_{[\varphi]}
=
-\frac{1}{6}
\left(d-\sum_{j=1}^{s_k}\frac{1}{n_{kj}}\right)
\frac{\overline{z_k}}{1+|z_k|^2}
+
\left.
\frac{\partial}{\partial z_k}\log\tau_{\mathrm{Sch}}
\right|_{[\varphi]}.
\]

In either standard affine coordinate on \(\mb P^1\), write
\(ds_{\mathrm{rd}}^2=\rho(z,\overline z)|dz|^2\), where
\(\rho(z,\overline z)=4(1+|z|^2)^{-2}\). Then
\(\partial_z\log\rho(z,\overline z)
=-2\overline z/(1+|z|^2)\).
The coefficient \(c_k\) introduced above can be written as
\[
c_k
=
\sum_{\substack{1\leq j\leq s_k\\ n_{kj}>1}}
\frac{n_{kj}^2-1}{12n_{kj}}
=
\frac{1}{12}
\left(d-\sum_{j=1}^{s_k}\frac{1}{n_{kj}}\right).
\]
Since the base point was arbitrary and the ramification profiles
are fixed on the coordinate ball, we obtain, for \(1\leq k\leq N\),
\[
\frac{\partial}{\partial z_k}\log\mathcal D
=
\frac{\partial}{\partial z_k}\log\tau_{\mathrm{Sch}}
+
c_k
\frac{\partial}{\partial z_k}
\log\rho(z_k,\overline{z_k}).
\]
The parameter-dependent resolvent estimates above show that
\(\log\mathcal D\) is \(C^1\) in the real Hurwitz coordinates.
The right-hand sides are smooth. Since \(\log\mathcal D\) is
real valued, the conjugate equations give its
\(\partial_{\bar z_k}\)-derivatives. Thus all its first real
partial derivatives are smooth, and \(\mathcal D\) is smooth
on the Hurwitz space.
We summarize the result as follows.

\begin{thm}\label{thm:hurwitz-determinant}
The function
\(\mathcal D:H_{g,d,N}^{\mb P^1}\to\mb R_{>0}\)
defined above is smooth, and \(d\log\mathcal D\) is a globally
defined exact real \(1\)-form.
On \(\mc B([\varphi],\delta)\), let \(B\) and \(\tau_B\) denote
the period matrix and the Bergman tau-function associated with the
transported homological marking fixed above. Then
\[
d\log\mathcal D
=
d\log\left(
\det\operatorname{Im}B\,|\tau_B|^2
\prod_{k=1}^N
\rho(z_k,\overline{z_k})^{c_k}
\right),
\]
where \(c_k=\frac1{12}\sum_{j=1}^{s_k}(n_{kj}-n_{kj}^{-1})\).
Consequently, on \(\mc B([\varphi],\delta)\), there exists a constant
\(C>0\) such that
\[
\mathcal D
=
C\,\det\operatorname{Im}B\,|\tau_B|^2
\prod_{k=1}^N
\rho(z_k,\overline{z_k})^{c_k}.
\]
When \(g=0\), \(\det\operatorname{Im}B\) is understood as the empty
determinant, equal to \(1\).
\end{thm}

Although the factors in the formula use local choices, their
logarithmic differential represents the global form
\(d\log\mathcal D\).

\clearpage
\section{Conclusion}\label{sec:conclusion}

We have established the local determinant formula for the Friedrichs
Laplacian of a spherical pullback metric on Hurwitz spaces with
arbitrary ramification profiles. The formula includes several
ramification points over a single branch value. Each point of
index \(n\) contributes the weight \((n-n^{-1})/12\) to the
corresponding metric factor, and the contributions within each
fiber add to give \(c_k\). The remaining factor is expressed
through the period matrix and the Bergman tau-function.

The proof implements the variational approach on a fixed smooth
surface. It adapts the cutoff-identification idea of
\cite{HillairetKalvinKokotov2018,KalvinKokotov2019} to the
arbitrary-ramification gluing construction. Compatible
trivializations give common maximal, minimal, and Friedrichs domains,
and the resulting graph-norm estimates yield trace-norm
differentiability of resolvent powers. The local
comparison with spherical conic models controls the diagonal
Weyl blocks at high energy and the spectral derivatives needed in the determinant
variation. At zero energy, the relevant coefficients are determined
by the Schiffer bidifferential. Their combination is expressed by
residues at the ramification points, which the Rauch formulas
identify with logarithmic derivatives of the Schiffer tau-function.
Uniform estimates at both spectral endpoints justify the passage
from the resolvent variation to the zeta-regularized determinant.

Although the period and tau-function factors are written using a
local homological marking and Hurwitz coordinates, the logarithmic
differential of their product with the metric factors is the
globally defined form \(d\log\mathcal D\).
Integration determines the determinant up
to a positive multiplicative constant on each coordinate ball.
This constant cancels in determinant ratios relative to a fixed
reference covering in the same coordinate ball.
The matrix comparison also retains the individual Weyl
coefficients beyond the scalar combination entering the
determinant formula.

The determinant formula proved here concerns the Friedrichs
realization. For a fixed conic surface, the results of
\cite[Theorem~5.3 and Corollary~5.4]{LiouWeyl2026} compare its
positive-spectrum zeta determinant with that of another self-adjoint
realization satisfying the zeta-regularity hypotheses imposed there.
The comparison is expressed through the finite-dimensional boundary
determinant determined by the Weyl curve, together with a correction
from the omitted nonpositive spectrum. Thus the Friedrichs
determinant computed here serves as a reference determinant for
studying other such self-adjoint realizations. A corresponding
formula over the Hurwitz space would additionally require control of
the comparison constant and of the chosen Lagrangian boundary
condition under deformation.

\clearpage
\appendix

\section{Estimates for the Local Interpolation Diffeomorphisms}
\label{app:local-interpolation}

We supply the estimates and global arguments used in
Section~\ref{sec:smooth-trivializations}. Fix \(k\) and use the
notation introduced there. In particular, \(H_{kj,t}\) is the local
interpolation map defined using the cutoff \(\beta\).

We first recall two elementary facts about smooth maps
\(F:\mb C\to\mb C\). Identifying \(\mb C\) with \(\mb R^2\),
we regard \((dF)_{z_0}\) as an \(\mb R\)-linear map.
Since \(F_x=F_z+F_{\bar z}\) and
\(F_y=i(F_z-F_{\bar z})\), we have
\((dF)_{z_0}(\xi)=F_z(z_0)\xi+F_{\bar z}(z_0)\bar\xi\)
for every \(\xi\in\mb C\).
Hence its Euclidean operator norm satisfies
\(\|(dF)_{z_0}\|\leq|F_z(z_0)|+|F_{\bar z}(z_0)|\).
Moreover, its real Jacobian determinant is
\(J_F(z_0)=|F_z(z_0)|^2-|F_{\bar z}(z_0)|^2\).
In particular, \(F\) is an orientation-preserving local
diffeomorphism near \(z_0\) whenever
\(|F_z(z_0)|>|F_{\bar z}(z_0)|\).

Fix \(j\), and continue to write \(n=n_{kj}\). Define
\[
F_t(\zeta)
:=
\begin{cases}
\displaystyle
1-\beta(|\zeta|^n)\frac{t}{\zeta^n},
& \zeta\neq0,\\[8pt]
1,
& \zeta=0.
\end{cases}
\]
Then \(F_t\) is smooth on \(\mb C\), takes values in
\(D_1(1/2)\), and satisfies
\(H_{kj,t}(\zeta)=\zeta F_t(\zeta)^{1/n}\).
Since \(H_{kj,t}\) is the identity near the origin,
\((H_{kj,t})_\zeta(0)=1\) and
\((H_{kj,t})_{\bar\zeta}(0)=0\).

For \(\zeta\neq0\), we have
\(\partial_\zeta|\zeta|^n
=(n/2)|\zeta|^{n-2}\bar\zeta\) and
\(\partial_{\bar\zeta}|\zeta|^n
=(n/2)|\zeta|^{n-2}\zeta\).
It follows that
\begin{align*}
(F_t)_\zeta
&=
-\frac{nt}{2}\beta'(|\zeta|^n)|\zeta|^{n-2}
\frac{\bar\zeta}{\zeta^n}
+
nt\beta(|\zeta|^n)\frac{1}{\zeta^{n+1}},
\\
(F_t)_{\bar\zeta}
&=
-\frac{nt}{2}\beta'(|\zeta|^n)|\zeta|^{n-2}
\frac{\zeta}{\zeta^n}.
\end{align*}
Using the chain rule and the identity
\(F_t+\beta(|\zeta|^n)t/\zeta^n=1\), we obtain
\begin{align*}
(H_{kj,t})_\zeta
&=
F_t^{1/n}
+\frac{\zeta}{n}F_t^{1/n-1}(F_t)_\zeta
\\
&=
F_t^{1/n-1}
\left(
1-\frac{t}{2}\beta'(|\zeta|^n)|\zeta|^{n-2}
\frac{\bar\zeta}{\zeta^{n-1}}
\right),
\\
(H_{kj,t})_{\bar\zeta}
&=
\frac{\zeta}{n}F_t^{1/n-1}(F_t)_{\bar\zeta}
\\
&=
-\frac{t}{2}F_t^{1/n-1}
\beta'(|\zeta|^n)|\zeta|^{n-2}\zeta^{2-n}.
\end{align*}

Set \(r_{kj,t}:=H_{kj,t}-\id_{\mb C}\).
The operator-norm estimate above gives
\[
\|(dr_{kj,t})_\zeta\|
\leq
\left|(H_{kj,t})_\zeta(\zeta)-1\right|
+
\left|(H_{kj,t})_{\bar\zeta}(\zeta)\right|.
\]
We now estimate the right-hand side uniformly in \(\zeta\).
By construction,
\(|F_t(\zeta)-1|\leq2|t|/\epsilon_k<1/2\), and hence
\(1/2<|F_t(\zeta)|<3/2\).
Set \(C_n:=\frac{n-1}{n}2^{2-1/n}\).
On \(D_1(1/2)\), the derivative of the chosen branch of
\(w\mapsto w^{1/n-1}\) has modulus at most \(C_n\).
Since this disc is convex, integration along the line segment
from \(1\) to \(F_t(\zeta)\) yields
\[
\left|F_t(\zeta)^{1/n-1}-1\right|
\leq
C_n|F_t(\zeta)-1|
\leq
\frac{2C_n}{\epsilon_k}|t|.
\]
Also,
\(|F_t(\zeta)^{1/n-1}|\leq2^{1-1/n}\).

For \(\zeta\neq0\), both
\(|\zeta|^{n-2}\bar\zeta/\zeta^{n-1}\) and
\(|\zeta|^{n-2}\zeta^{2-n}\) have modulus \(1\).
Therefore,
\begin{align*}
\left|(H_{kj,t})_\zeta(\zeta)-1\right|
&\leq
\left(
\frac{2C_n}{\epsilon_k}
+
2^{-1/n}\|\beta'\|_\infty
\right)|t|,
\\
\left|(H_{kj,t})_{\bar\zeta}(\zeta)\right|
&\leq
2^{-1/n}\|\beta'\|_\infty|t|.
\end{align*}
These inequalities also hold at \(\zeta=0\), where their
left-hand sides vanish.

For each \(j\), set
\(C_{kj}:=2C_{n_{kj}}/\epsilon_k
+2^{1-1/n_{kj}}\|\beta'\|_\infty\), and let
\(C:=\max_{1\leq j\leq s_k}C_{kj}\).
Then \(C>0\) is independent of \(t\) and \(\zeta\), and
\begin{equation}\label{eq:H-uniform-C1}
\sup_{\zeta\in\mb C}\|(dr_{kj,t})_\zeta\|
\leq C|t|,
\qquad 1\leq j\leq s_k.
\end{equation}
Shrink \(\delta\), if necessary, so that
\(0<\delta<\min\{\epsilon_k/4,1/(2C)\}\).
For every \(|t|<\delta\), every \(j\), and every
\(\zeta\in\mb C\), we then have
\(\|(dH_{kj,t})_\zeta-I_2\|\leq C|t|<1/2\).
Thus \((dH_{kj,t})_\zeta\) is invertible, and the inverse
function theorem shows that \(H_{kj,t}\) is a local
diffeomorphism.

\begin{prop}\label{prop:H-global-diffeomorphism}
For every \(|t|<\delta\) and \(1\leq j\leq s_k\), the map
\(H_{kj,t}:\mb C\to\mb C\) is an orientation-preserving smooth
diffeomorphism.
\end{prop}

\begin{proof}
Fix \(j\) and \(t\).
For each \(\eta\in\mb C\), we show that
\(H_{kj,t}(\zeta)=\eta\) has a unique solution.
Define \(\Phi_t:\mb C\to\mb C\) by
\(\Phi_t(\zeta):=\eta-r_{kj,t}(\zeta)\).
The uniform derivative bound and the mean value estimate give
\[
|\Phi_t(\zeta_1)-\Phi_t(\zeta_2)|
=
|r_{kj,t}(\zeta_1)-r_{kj,t}(\zeta_2)|
\leq
C|t|\,|\zeta_1-\zeta_2|.
\]
Since \(C|t|<1/2\), the map \(\Phi_t\) is a contraction.
The completeness of \(\mb C\) and the Banach fixed-point theorem
therefore imply that \(\Phi_t\) has a unique fixed point
\(\zeta_0\).
The identity \(\zeta_0=\eta-r_{kj,t}(\zeta_0)\) is equivalent
to \(H_{kj,t}(\zeta_0)=\eta\).
Thus \(H_{kj,t}\) is bijective.
Since it is also a local diffeomorphism, its inverse is smooth,
so \(H_{kj,t}\) is a smooth diffeomorphism.

To verify that it preserves orientation, fix \(\zeta_0\in\mb C\)
and consider the path of real matrices
\(\Phi(s):=(1-s)I_2+s(dH_{kj,t})_{\zeta_0}\),
\(0\leq s\leq1\).
Since \(\|\Phi(s)-I_2\|\leq sC|t|<1\), every \(\Phi(s)\)
is invertible.
The function \(s\mapsto\det\Phi(s)\) is therefore continuous
and nowhere zero.
As \(\det\Phi(0)=1\), it is positive throughout \([0,1]\).
In particular,
\(\det(dH_{kj,t})_{\zeta_0}>0\).
Since \(\zeta_0\) was arbitrary, \(H_{kj,t}\) preserves
orientation on \(\mb C\).
\end{proof}

\begin{prop}\label{prop:H-domain-diffeomorphism}
For every \(|t|<\delta\) and \(1\leq j\leq s_k\), the map
\(H_{kj,t}\) restricts to an orientation-preserving smooth
diffeomorphism
\(H_{kj,t}:\Omega_{kj}(0)\to\Omega_{kj}(t)\).
In particular,
\(H_{kj,t}(\Omega_{kj}(0))=\Omega_{kj}(t)\).
\end{prop}

\begin{proof}
Fix \(j\), and write \(n=n_{kj}\).
Recall that
\(\Omega_{kj}(t)=\{\zeta\in\mb C:|t+\zeta^n|<\epsilon_k\}\).
By definition,
\[
t+H_{kj,t}(\zeta)^n
=
\zeta^n+\bigl(1-\beta(|\zeta|^n)\bigr)t
\qquad (\zeta\in\mb C).
\]
Let \(\zeta\in\Omega_{kj}(0)\), so that
\(|\zeta|^n<\epsilon_k\).
If \(|\zeta|^n<3\epsilon_k/4\), then
\[
|t+H_{kj,t}(\zeta)^n|
\leq
|\zeta|^n+|t|
<
\frac34\epsilon_k+\frac14\epsilon_k
=
\epsilon_k.
\]
If \(3\epsilon_k/4\leq|\zeta|^n<\epsilon_k\), then
\(\beta(|\zeta|^n)=1\), and hence
\(|t+H_{kj,t}(\zeta)^n|=|\zeta|^n<\epsilon_k\).
Thus
\(H_{kj,t}(\Omega_{kj}(0))\subset\Omega_{kj}(t)\).

Conversely, let \(\eta\in\Omega_{kj}(t)\).
Since \(H_{kj,t}\) is a diffeomorphism of \(\mb C\),
there is a unique \(\zeta_0\in\mb C\) such that
\(H_{kj,t}(\zeta_0)=\eta\).
If \(|\zeta_0|^n\geq\epsilon_k\), then
\(\beta(|\zeta_0|^n)=1\), so
\(|t+\eta^n|=|\zeta_0|^n\geq\epsilon_k\),
contrary to \(\eta\in\Omega_{kj}(t)\).
Therefore \(\zeta_0\in\Omega_{kj}(0)\), proving the reverse
inclusion.

The restriction is consequently a diffeomorphism from
\(\Omega_{kj}(0)\) onto \(\Omega_{kj}(t)\).
It preserves orientation because \(H_{kj,t}\) does so on
\(\mb C\).
\end{proof}

\section*{Acknowledgements}

The author thanks Professor Chin-Lung Wang of the Department of
Mathematics at National Taiwan University, who introduced him to the
study of conic Riemann surfaces and thereby led him to the problem
investigated in this paper.

The series of works on conic Laplacians, \(S\)-matrices, and spectral
determinants by Hillairet, Kalvin, and Kokotov has had a fundamental
influence on this paper
\cite{HillairetKokotov2013,HillairetKalvinKokotov2018,
KalvinKokotov2019,Kalvin2019}.
The simple-ramification result of Kalvin and Kokotov
\cite{KalvinKokotov2019} motivated the present investigation, and
their variational framework, cutoff identifications, and zero-energy
calculation provide essential foundations for it.

Building on this framework, the present paper treats general branched
data, develops the
fixed-domain operator analysis required in the present setting, and
reformulates the zero-energy calculation in the boundary-data
normalization used here, with the local contributions grouped over
each branch value. The Bergman tau-function and its variational
formulas are taken from Kokotov and Korotkin
\cite{KokotovKorotkin2004}.

Building on these works, the project developed over several years
through research notes, calculations, continued study of the
literature, and successive drafts. The Weyl-curve viewpoint revealed
the local \(S\)-matrix comparison, while Hurwitz coordinates clarified
how to carry it out for arbitrary ramification and how to treat the
varying conic metrics on a fixed smooth surface.

The author is deeply grateful to his family for their help and
support during difficult times.

ChatGPT (OpenAI) was used during the revision of the manuscript for
editorial assistance, including language, exposition, and notation
consistency. The author reviewed all incorporated suggestions and
assumes full responsibility for the mathematical arguments and the
content of the paper.

\end{document}